\documentclass[11pt]{amsart}

\usepackage[top=27 mm, bottom=20 mm, left=25 mm, right=25mm]{geometry}

\usepackage[all]{xy}
\usepackage{mathrsfs}
\usepackage{amsmath,amsthm, amssymb, amsgen,amsxtra,amsfonts, amsbsy} 
\usepackage[all]{xy}
\usepackage{verbatim}

\usepackage{graphicx}
\usepackage{caption,rotating}
\usepackage{color}
\usepackage{subcaption}

\usepackage{tikz}
\usetikzlibrary{trees}
\usetikzlibrary{decorations.pathmorphing}
\usetikzlibrary{decorations.markings}

\usepackage{multirow}

\usepackage{times}

\usepackage[font=small,labelfont=small]{caption}

\usepackage{enumitem}

\usepackage{mathtools,stackengine}

\usepackage{pdflscape}
\usepackage{array}

\usepackage{arydshln}

\usepackage{float}
\restylefloat{table}

\theoremstyle{definition}
\newtheorem{Definition}{Definition}[section]

\newtheorem{Example}[Definition]{Example}

\newtheorem{Remark}[Definition]{Remark}
\newtheorem{Remark/Notation}[Definition]{Remark/Notation}
\newtheorem{Notation}[Definition]{Notation}

\newtheorem*{Acknowledgements}{Acknowledgements}

\theoremstyle{plain}
\newtheorem{Theorem}[Definition]{Theorem}

\newtheorem*{Main Theoremx}{Main Theorem}
\newtheorem{Proposition}[Definition]{Proposition}

\newtheorem{Lemma}[Definition]{Lemma}
\newtheorem{Corollary}[Definition]{Corollary}

\newtheoremstyle{voiditstyle}{3pt}{3pt}{\itshape}{\parindent}%
{\bfseries}{.}{ }{\thmnote{#3}}%
\theoremstyle{voiditstyle}

\newtheoremstyle{voidromstyle}{3pt}{3pt}{\rm}{\parindent}%
{\bfseries}{.}{ }{\thmnote{#3}}%
\theoremstyle{voidromstyle}

\newcommand{\cal}{\mathcal}

\numberwithin{equation}{section}

\newcommand{\bbP}{{\mathbb{P}}}
\newcommand{\bbG}{{\mathbb{G}}}

\newcommand{\GL}{\operatorname{GL}}

\newcommand{\PGL}{\operatorname{PGL}}
\newcommand{\Aut}{\operatorname{Aut}}

\newcommand{\Spec}{\operatorname{Spec}}

\newcommand{\Pic}{\operatorname{Pic}}

\newcommand{\Stab}{\operatorname{Stab}}

\newcommand{\bsm}{\left(\begin{smallmatrix}}
\newcommand{\esm}{\end{smallmatrix}\right)}

\newcommand{\beq}{\begin{equation}}
\newcommand{\eeq}{\end{equation}}
\usepackage[
           colorlinks=true,
            urlcolor=blue,       % \href{...}{...} external (URL)
            filecolor=blue,     % \href{...} local file
            linkcolor=blue,       % \ref{...} and \pageref{...}
            citecolor=blue,         % \cite{...}
            pdftitle= {Title},
            pdfauthor={Author},
            pdfsubject={Subject},
            pdfkeywords={Key}
            pagebackref,
            bookmarks=false,
            bookmarksopen=true]
            {hyperref}
\makeatletter
\def\@tocline#1#2#3#4#5#6#7{\relax
  \ifnum #1>\c@tocdepth % then omit
  \else
    \par \addpenalty\@secpenalty\addvspace{#2}%
    \begingroup \hyphenpenalty\@M
    \@ifempty{#4}{%
      \@tempdima\csname r@tocindent\number#1\endcsname\relax
    }{%
      \@tempdima#4\relax
    }%
    \parindent\z@ \leftskip#3\relax \advance\leftskip\@tempdima\relax
    \rightskip\@pnumwidth plus4em \parfillskip-\@pnumwidth
    #5\leavevmode\hskip-\@tempdima
      \ifcase #1
       \or\or \hskip 2em \or \hskip 2em \else \hskip 3em \fi%
      #6\nobreak\relax
    \dotfill\hbox to\@pnumwidth{\@tocpagenum{#7}}\par
    \nobreak
    \endgroup
  \fi}
\makeatother

\begin{document}

\title[RDP del Pezzo surfaces with global vector fields in odd characteristic]{RDP del Pezzo surfaces with global vector fields\\ in odd characteristic}

\author{Gebhard Martin}
\address{Mathematisches Institut der Universit\"at Bonn, Endenicher Allee 60, 53115 Bonn, Germany}
\curraddr{}
\email{gmartin@math.uni-bonn.de}

\author{Claudia Stadlmayr}
\address{TU M\"unchen, Zentrum Mathematik - M11, Boltzmannstra{\ss}e 3, 85748 Garching bei M\"unchen, Germany}
\curraddr{}
\email{claudia.stadlmayr@ma.tum.de}

\date{\today}

\maketitle

\begin{abstract}
We classify RDP del Pezzo surfaces with global vector fields over arbitrary algebraically closed fields of characteristic $p \neq 2$. In characteristic $0$, every RDP del Pezzo surface $X$ is \emph{equivariant}, that is, $\Aut_X = \Aut_{\widetilde{X}}$, where $\widetilde{X}$ is the minimal resolution of $X$, hence the classification of RDP del Pezzo surfaces with global vector fields is equivalent to the classification of weak del Pezzo surfaces with global vector fields.

In this article, we show that if $p \neq 2,3,5,7$, then it is still true that every RDP del Pezzo surface is equivariant. We classify the non-equivariant RDP del Pezzo surfaces in characteristic $p = 3,5,7$, giving explicit equations for every such RDP del Pezzo surface in all possible degrees.
As an application, we construct regular non-smooth RDP del Pezzo surfaces over imperfect fields of characteristic $7$, thereby showing that the known bound $p \leq 7$ for the characteristics, where such a surface can exist, is sharp.
\end{abstract}

\tableofcontents

\vspace{-1cm}
\section{Introduction}
We are working over an algebraically closed field $k$ of characteristic $p \geq 0$. Let $X$ be a del Pezzo surface with at worst rational double points and let $\pi: \widetilde{X} \to X$ be its minimal resolution, so that, by definition, $\widetilde{X}$ is a weak del Pezzo surface. 
Since $-K_X$ is ample, $\Aut_X$ is an affine group scheme of finite type, hence its group of automorphisms $\Aut(X):= \Aut_X(k)$ is infinite if and only if the automorphism scheme $\Aut_X$ is positive-dimensional. Moreover, by Blanchard's Lemma \cite[Theorem 7.2.1]{Brion1}, and since $X$ is the anti-canonical model of $\widetilde{X}$, there is a closed immersion of group schemes
$\pi_*: \Aut_{\widetilde{X}} \hookrightarrow  \Aut_X$. We call $X$ \emph{equivariant}, if $\pi_*$ is an isomorphism. Summarizing, for all characteristics, there is the following chain of implications:
\begin{equation}\label{eq: equivalence infinite autogroup}
|\Aut(X)| = |\Aut(\widetilde{X})| = \infty \hspace{2mm} \Longrightarrow \hspace{2mm} H^0(\widetilde{X},T_{\widetilde{X}}) \neq 0 \hspace{2mm} \Longrightarrow \hspace{2mm} H^0(X,T_X) \neq 0
\end{equation}

Over the complex numbers, every RDP del Pezzo surface $X$ is equivariant, so in particular we have $H^0(X,T_X) = H^0(\widetilde{X},T_{\widetilde{X}})$, and by Cartier's theorem $\Aut_X^0$ is smooth, hence it is positive-dimensional if and only if $H^0(X,T_X) \neq 0$. In other words, in characteristic $0$, all implications in \eqref{eq: equivalence infinite autogroup} are in fact equivalences. 

In our recent article \cite{MartinStadlmayr}, we obtained the classification of weak del Pezzo surfaces with global vector fields over algebraically closed fields of arbitrary characteristic (over the complex numbers, an independent proof was given by Cheltsov and Prokhorov in \cite{CheltsovProkhorov}). By \eqref{eq: equivalence infinite autogroup}, this includes the classification of all RDP del Pezzo surfaces with infinite automorphism group, but we note that if $p = 2,3$, there are RDP del Pezzo surfaces with finite automorphism group whose minimal resolution has global vector fields, so the first implication in \eqref{eq: equivalence infinite autogroup} is not an equivalence precisely if $p = 2,3$. In other words, we have
\begin{equation}\label{eq: equivalence infinite autogroup2}
|\Aut(X)| = |\Aut(\widetilde{X})| = \infty \hspace{2mm} \xRightarrow{\stackrel{\hspace{1mm} p \neq 2,3 \hspace{1mm}}{\Longleftarrow}} \hspace{2mm} H^0(\widetilde{X},T_{\widetilde{X}}) \neq 0 \hspace{2mm} \xRightarrow{\hspace{5mm}} \hspace{2mm} H^0(X,T_X) \neq 0
\end{equation}
and a classification of $\widetilde{X}$ with $H^0(\widetilde{X},T_{\widetilde{X}}) \neq 0$ in all characteristics.

Now, the missing piece is a classification of RDP del Pezzo surfaces $X$ with $H^0(X,T_X) \neq 0$.
As a first step, we extend Hirokado's \cite{Hirokado} results on the liftability of vector fields to resolutions of rational double points to group scheme actions on RDP del Pezzo surfaces as follows: 
\begin{Theorem}[= Theorem \ref{thm: equivariantmain}] \label{thm: equivariantmainintroduction}
Let $X$ be an RDP del Pezzo surface and let $\pi: \widetilde{X} \to X$ be its minimal resolution. Assume that one of the following conditions holds:
\begin{enumerate}
    \item 
    $p \not \in \{ 2,3,5,7\}$,
    \item
    $p = 7$ and $X$ does not contain an RDP of type $A_6$.
    \item
    $p = 5$ and $X$ does not contain an RDP of type $A_4$ or $E_8^0$.
    \item
    $p=3$ and $X$ does not contain an RDP of type $A_2, A_5, A_8, E_6^0, E_6^1, E_7^0, E_8^0$ or $E_8^1$.
    \item
    $p = 2$ and $X$ does not contain an RDP of type $A_1,A_3,A_5,A_7,D_n^r,E_6^0,E_7^0,E_7^1,E_7^2,E_7^3,E_8^0,E_8^1,E_8^2$ or $E_8^3$.
\end{enumerate}
Then, $\Aut_X = \Aut_{\widetilde{X}}$, and thus, in particular, $H^0(\widetilde{X},T_{\widetilde{X}}) = H^0(X,T_X)$. Therefore, $H^0(X,T_X) \neq 0$ if and only if $X$ is the anti-canonical model of one of the surfaces in the classification tables of \cite{MartinStadlmayr}.
\end{Theorem}

Thus, in order to classify RDP del Pezzo surfaces $X$ with global vector fields, we may restrict our attention to RDP del Pezzo surfaces containing a configuration $\Gamma$ of RDPs excluded in Theorem \ref{thm: equivariantmain}. In Theorem \ref{thm: lattice theoretic reduction criterion}, we give a criterion for $X$ to be the blow-up of an RDP del Pezzo surface of higher degree with the same configuration $\Gamma$.
In the language of the Minimal Model Program, this means that we give a sufficient criterion for the existence of a $K_{\widetilde{X}}$-negative extremal ray on $\widetilde{X}$ which lies in the orthogonal complement of the exceptional locus over $\Gamma$.
Using Blanchard's Lemma, this allows us to set up an inductive argument for the classification of non-equivariant RDP del Pezzo surfaces $X$ with RDP configuration $\Gamma$. This strategy will be carried out in Sections \ref{sec: char7}, \ref{sec: char5}, and \ref{sec: char3}, for characteristic $p = 7$, $p = 5$, and $p = 3$, respectively. The following theorem is obtained by combining Theorems \ref{thm: mainchar7}, \ref{thm: mainchar5}, and \ref{thm: mainchar3}.

\begin{Theorem} \label{thm: nonequivariantmain}
Let $X$ be an RDP del Pezzo surface and let $\pi: \widetilde{X} \to X$ be its minimal resolution. Assume that $H^0(X,T_X) \neq 0$.
Then, the following hold:
\begin{enumerate}
    \item If $p = 7$ and $X$ contains an RDP of type $A_6$, then $X$ is one of the $2$ surfaces in Table \ref{Table Eqn and Aut - char 7}.
    \item If $p = 5$ and $X$ contains an RDP of type $A_4$ or $E_8^0$, then $X$ is one of the $9$ surfaces in Table \ref{Table Eqn and Aut - char 5}.
    \item If $p = 3$ and $X$ contains an RDP of type $A_2, A_5, A_8, E_6^0, E_6^1, E_7^0, E_8^0$ or $E_8^1$, then $X$ is a member of one of the $56$ families of surfaces in Tables \ref{Table Eqn and Aut - char 3, deg at least 4}, \ref{Table Eqn and Aut - char 3, deg 3}, \ref{Table Eqn and Aut - char 3, deg 2}, and \ref{Table Eqn and Aut - char 3, deg 1}.
\end{enumerate}
For each of these surfaces, we have $h^0(X,T_X) > h^0(\widetilde{X},T_{\widetilde{X}})$ and in particular $\Aut_X \neq \Aut_{\widetilde{X}}$.
\end{Theorem}

\begin{Remark}
The reason why we do not treat the case $p = 2$ is due to the sheer amount of RDP del Pezzo surfaces with global vector fields in characteristic $2$.
Indeed, by Theorem \ref{thm: equivariantmain}, it is unclear whether $\Aut_{\widetilde{X}} = \Aut_X$ as soon as $X$ has a single node and in fact, even for the quadratic cone $\{x_0^2 - x_1x_2=0\} \subseteq \mathbb{P}^3$ in characteristic $2$ it is not true that every vector field lifts to its minimal resolution; consider for example $x_3\partial_{x_0}$. However, in principle, our approach would also work if $p = 2$.
\end{Remark}

Comparing Tables \ref{Table Eqn and Aut - char 7}, \ref{Table Eqn and Aut - char 5}, \ref{Table Eqn and Aut - char 3, deg at least 4}, \ref{Table Eqn and Aut - char 3, deg 3}, \ref{Table Eqn and Aut - char 3, deg 2}, and \ref{Table Eqn and Aut - char 3, deg 1} with the classification in \cite{MartinStadlmayr}, we see that in characteristics $p = 3,5$ and $7$, there exists an RDP del Pezzo surface $X$ with $H^0(X,T_X) \neq 0$ whose minimal resolution admits no non-trivial global vector fields. In other words, we have the following picture, where the implications from right to left hold \emph{only} in the indicated characterstics:
 \begin{equation}\label{eq: equivalence infinite autogroup3}
|\Aut(X)| = |\Aut(\widetilde{X})| = \infty \hspace{2mm} \xRightarrow{\stackrel{\hspace{1mm} p \neq 2,3 \hspace{1mm}}{\Longleftarrow}} \hspace{2mm} H^0(\widetilde{X},T_{\widetilde{X}}) \neq 0 \hspace{2mm} \xRightarrow{\stackrel{\hspace{1mm}p \neq 2,3,5,7 \hspace{1mm}}{\Longleftarrow}} \hspace{2mm}H^0(X,T_X) \neq 0.  
\end{equation}

\begin{Acknowledgements}
Research of the first named author was supported by the DFG Research Grant MA 8510/1-1.
The second named author gratefully acknowledges funding by the DFG Sachbeihilfe LI 1906/5-1 ``Geometrie von rationalen Doppelpunkten'' and support by the doctoral program TopMath and the TUM Graduate School.
\end{Acknowledgements}

\section{An application: Regular inseparable twists of RDP del Pezzo surfaces}
 A \emph{twisted form} of a $k$-scheme $X$ over a field extension $L \supseteq k$ is a scheme $Y$ over $L$ such that $Y_{\bar{L}} \cong X_{\bar{L}}$, where $\bar{L}$ is an algebraic closure of $L$.
If $X$ is a proper scheme over  $k$, then smoothness of $\Aut_X$ is intimately related with properties of twisted forms of $X$, as the following proposition shows. Even though this proposition should be well-known, we include the proof for the convenience of the reader.

\begin{Proposition} \label{prop: twists}
Let $k \subseteq L$ be a field extension.
Let $X$ be a proper scheme over $k$ and let $Y$ be a twisted form of $X$ over $L$. Assume that $\Aut_X$ is smooth. Then, the following hold: 
\begin{enumerate}
\item If $L$ is separably closed, then $Y \cong X_L$.
\item If $\mathcal{P}$ is a property of schemes that is stable under field extensions and local in the \'etale topology, then, if $X$ satisfies $\mathcal{P}$, also $Y$ satisfies $\mathcal{P}$.
\end{enumerate}
\end{Proposition}

\begin{proof}
Let us first prove Claim (1). As explained for example in \cite[p.134]{Milne}, an isomorphism $\bar{\varphi}: Y_{\bar{L}} \cong X_{\bar{L}}$ gives rise to a \v{C}ech cocycle on the fppf site of $\Spec L$, hence to an element in $\check{H}^1_{{\rm fppf}}(\Spec L,\Aut_{X_L})$. 
By \cite[Chapter III: Theorem 4.3.(b), Corollary 4.7]{Milne}, the smoothness of $\Aut_X$ implies that $$\check{H}^1_{{\rm fppf}}(\Spec L,\Aut_{X_L}) = \check{H}^1_{{\rm \acute{e}t}}(\Spec L,\Aut_{X_L})$$ and since $L$ is separably closed, the latter is trivial. Hence, $Y$ and $X_L$ are already isomorphic over $L$.

For Claim (2), let $L^{sep}$ be the separable closure of $L$ in $\bar{L}$. By (1), we have $X_{L^{sep}} \cong Y_{L^{sep}}$. Since $X$ is proper, this isomorphism is defined over a finite subextension $L \subseteq L' \subseteq L^{sep}$, so that $X_{L'} \cong Y_{L'}$. The morphism $\Spec L' \to \Spec L$ is finite and \'etale. Hence, by our assumptions on $\mathcal{P}$, if $X$ satisfies $\mathcal{P}$, then $X_{L'}$ satisfies $\mathcal{P}$, and thus also $Y$ satisfies $\mathcal{P}$.
\end{proof}

Choosing for $\mathcal{P}$ the property that the singular locus is non-empty and specializing to the case where $X$ is an RDP del Pezzo surface, we obtain the following:

\begin{Corollary}
Let $X$ be an RDP del Pezzo surface over $k$. Let $k \subseteq L$ be a field extension and let $Y$ be a twisted form of $X$ over $L$. If $X$ has at least one singular point and $Y$ is regular, then $\Aut_X$ is non-smooth.
\end{Corollary}

At a first glance, these twists seem rather hard to get a grip on geometrically, but it turns out that they can be written down explicitly if one has explicit descriptions of $X$ and $\Aut_X$. For example, consider $\mu_p \subseteq \PGL_{n,k}$ embedded diagonally with weights $(a_0,a_1,\hdots,a_{n-1},1)$.
We can write $\mu_p$ as the kernel of the surjective homomorphism
\begin{eqnarray*}
f: &  \mathbb{G}_m^n & \to \mathbb{G}_m^n \\
& (u_0,\hdots,u_n) & \mapsto  (u_0u_n^{-a_0},\hdots,u_{n-1}u_n^{-a_{n-1}},u_n^p).
\end{eqnarray*}
By \cite[Proposition 4.5]{Milne} and Hilbert $90$, for every field extension $k \subseteq L$, this yields a short exact sequence of abelian groups
$$
0 \to \mathbb{G}_m^n(L) \overset{f(L)}{\to} \mathbb{G}_m^n(L) \overset{d}{\to}  \check{H}^1_{{\rm fppf}}(\Spec L,\mu_p) \to 0.
$$
Here, for $g \in \mathbb{G}_m^n(L)$, the element $d(g)$ is defined by choosing $\bar{g} \in \mathbb{G}_m^n(\bar{L})$ such that $f(\bar{L})(\bar{g}) = g_{\bar L}$ and setting $d(g)$ to be the image of the cocycle $(\bar{g} \otimes 1)^{-1}(1 \otimes \bar{g}) \in \mu_p(\bar{L} \otimes_L \bar{L})$.
If $X \subseteq \mathbb{P}^n_{k}$ is a subvariety stabilized by $\mu_p$, then $\bar{g}^{-1}(X_{\bar{L}}) \subseteq \mathbb{P}^n_{\bar{L}}$ is defined over $L$ and the cocycle one associates to this twisted form is in the same class as $d(g)$. In other words, we can realize every twist of $X$ corresponding to an element of $\check{H}^1_{{\rm fppf}}(\Spec L,\mu_p)$ by choosing $\bar{g} \in \mathbb{G}_m^n(\bar{L})$ such that $f(\bar{L})(\bar{g})$ is defined over $L$ and translating $X$ along $\bar{g}^{-1}$.
 
\begin{Example} \label{ex: non-smooth twist}
Consider the quartic curve $Q = \{x^3y + y^3z + z^3x = 0\} \subseteq \mathbb{P}^2_k$ in characteristic $7$. It is stable under the $\mu_7$-action with weights $(4,2,1)$. Let $L = k(t)$ and consider $\bar{g} = (t^{4/7},t^{2/7},t^{1/7}) \in \mathbb{G}_m^2(\overline{L}) \subseteq \PGL_{3}(\overline{L})$. Then,
$$
\bar{g}^{-1}(Q) = \{t^2 x^3y + ty^3z + tz^3x = 0\} = \{tx^3y + y^3z + z^3x = 0\}.
$$
Spreading out over $\mathbb{A}^1_k - \{0\} = \Spec k[t,t^{-1}]$, we obtain a fibered surface $S$ over $\mathbb{A}^1_k - \{0\}$ which is easily checked to be smooth using the Jacobian criterion. Its generic fiber is a twisted form of $Q$ over $k(t)$
and this twisted form is regular, because it is the generic fiber of a flat morphism between smooth $k$-schemes. Taking the double cover of $(\mathbb{A}^1 - \{0\}) \times \mathbb{P}^2$ branched over $S$, we obtain a smooth fibered threefold $T$ over $\mathbb{A}^1 - \{0\}$ whose generic fiber is the regular del Pezzo surface $Y$ of degree $2$ given by the equation
$$
\{ w^2 = tx^3y + y^3z + z^3x \} \subseteq \mathbb{P}_{k(t)}(1,1,1,2).
$$
As before, $Y$ is regular, being the generic fiber of a flat morphism of smooth $k$-schemes. Observe, however, that $Y$ is not smooth, because $Y_{\overline{k(t)}} \cong X$, where $X$ is the del Pezzo surface of degree $2$ with a singularity of type $A_6$ given in our Table \ref{Table Eqn and Aut - char 7}. 
\end{Example}

\begin{Remark}
Example \ref{ex: non-smooth twist} shows that the bound $p \leq 7$ given in \cite[Proposition 5.2]{BernasconiTanaka} for the characteristics in which non-smooth regular RDP del Pezzo surfaces can exist is sharp. Using the approach explained in the beginning of Example \ref{ex: non-smooth twist}, it is not hard to construct similar examples if $p = 2,3,5$, but since this is not the topic of this article, we leave these constructions to the interested reader. Finally, we note that it is no mere coincidence that the Klein quartic in characteristic $7$ appears in this context and refer the reader to \cite{stoehr} for a closer study of this curve and its regular twists.
\end{Remark}
 
\section{Preliminaries on (RDP) del Pezzo surfaces}
In this section, we recall the definition of RDP del Pezzo surfaces and weak del Pezzo surfaces, which occur as minimal resolutions of RDP del Pezzo surfaces, as well as their basic properties.

\begin{Definition} \label{Def delPezzos}
Let $X$ and $\widetilde{X}$ be projective surfaces.
\begin{itemize}
\item
$X$ is a \emph{del Pezzo surface} if it is smooth and $-K_X$ is ample.
\item
$\widetilde{X}$ is a \emph{weak del Pezzo surface} if it is smooth and $-K_{\widetilde{X}}$ is big and nef.
\item 
$X$ is an \emph{RDP del Pezzo surface} if all its singularities are rational double points and $-K_X$ is ample.
\end{itemize}
\mbox{In all the above cases, the number $\deg (X)= K_X^2$ (resp. $\deg (\widetilde{X})= K_{\widetilde{X}}^2$) is called the \emph{degree of} $X$ (resp. $\widetilde{X}$).}
\end{Definition}

Recall that $1 \leq \deg(X) = \deg(\widetilde{X}) \leq 9$ and that every weak del Pezzo surface of degree $d$ and different from $\mathbb{P}^1 \times \mathbb{P}^1$ and the second Hirzebruch surface $\mathbb{F}_2$ can be realized as a blow-up of $\mathbb{P}^2$ in $9-d$ (possibly infinitely near) points in almost general position.

As mentioned in the beginning of this section, weak del Pezzo surfaces arise as the minimal resolutions of RDP del Pezzo surfaces and, conversely, every RDP del Pezzo surface $X$ is the anti-canonical model of a weak del Pezzo surface $\widetilde{X}$. The linear systems $|-nK_{\widetilde{X}}|$ are well studied (see e.g. \cite[Proposition 2.14, Theorem 2.15]{BernasconiTanaka} for proofs in positive characteristic). We denote the morphism induced by a linear system $|D|$ by $\varphi_{|D|}$ and recall that a curve singularity is called \emph{simple} if its completion is isomorphic to one of the normal forms given in \cite[Section 1]{GreuelKroening}.

\begin{Theorem} \label{thm: anti-cano embeddings/coverings}
Let $\widetilde{X}$ be a weak del Pezzo surface of degree $d$.
\begin{enumerate}
    \item If $d \geq 3$, then $\varphi_{|-K_{\widetilde{X}}|}$ factors as
    $
    \widetilde{X} \xrightarrow{\pi} X \xhookrightarrow{\varphi_{|-K_{X}|}} \mathbb{P}^d,
    $
    where $\varphi_{|-K_{X}|}$ is a closed immersion that realizes $X$ as a surface of degree $d$.
    \item If $d = 2$, then $\varphi_{|-K_{\widetilde{X}}|}$ factors as
    $
    \widetilde{X} \xrightarrow{\pi}  X \xrightarrow{\varphi_{|-K_{X}|}} \mathbb{P}^2,
    $
    where $\varphi_{|-K_{X}|}$ is finite flat of degree $2$. 
    
    \noindent  If $p \neq 2$, then $\varphi_{|-K_{X}|}$ is branched over a quartic curve $Q$ with simple singularities.
    \item If $d = 1$, then the $\varphi_{|-2K_{\widetilde{X}}|}$ factors as
    $
    \widetilde{X} \xrightarrow{\pi}  X \xrightarrow{\varphi_{|-2K_{X}|}} \mathbb{P}(1,1,2) \subseteq \mathbb{P}^3,
    $
    where $\mathbb{P}(1,1,2)$ is the quadratic cone and $\varphi_{|-2K_{X}|}$ is finite flat of degree $2$. 
    
    \noindent If $p \neq 2$, then $\varphi_{|-2K_{X}|}$ is branched over a sextic curve $S$ with simple singularities.
\end{enumerate}
\end{Theorem}

Next, we recall the notion of a \emph{marking} of a weak del Pezzo surface $\widetilde{X} \not \in \{\mathbb{P}^1 \times \mathbb{P}^1, \mathbb{F}_2\}$ and explain how to describe the negative curves on $\widetilde{X}$ in terms of such a marking.

\begin{itemize}
\item A \emph{marking} of $\widetilde{X}$ is an isomorphism $\phi: {\rm I}^{1,9-d} \overset{\sim}{\to} \Pic(\widetilde{X})$, where ${\rm I}^{1,9-d}$ is the lattice of rank $10 - d$ with quadratic form given by the diagonal matrix $(1,-1,\hdots,-1)$ with respect to a basis $e_0,\hdots,e_{9-d}$.
\item A realization $\pi: \widetilde{X} \to \mathbb{P}^2$ of $\widetilde{X}$ as an iterated blow-up of $\mathbb{P}^2$ induces a marking $\phi$ with $\phi(e_0) = \pi^* \mathcal{O}_{\mathbb{P}^2}(1)$ and $\phi(e_i)$ is the class of the preimage in $\widetilde{X}$ of the $i$-th point blown up by $\pi$. A marking that arises in this way is called \emph{geometric}.
\item If $\phi$ is a geometric marking of $\widetilde{X}$, then $\phi^{-1}(K_{\widetilde{X}}) = (-3,1,\hdots,1) =: k_{9-d}$.
\item The lattice $E_{9-d}$ is defined as $\langle k_{9-d} \rangle^\perp \subseteq {\rm I}^{1,9-d}$. 
\begin{itemize}
    \item For $d = 1,2,3$, the lattices $E_{9-d}$ are precisely the three exceptional irreducible root lattices.
    \item For $d = 4,5,6$, there are identifications $E_5 = D_5$, $E_4 = A_4$, $E_3 = A_2 \oplus A_1$.
    \item For $d = 7,8,$ the lattices $E_{9-d}$ are no root lattices. Every maximal root lattice contained in $E_2$ is isomorphic to $A_1$, and $E_1$ does not contain any $(-2)$-vectors.
\end{itemize}
\end{itemize}

Following \cite[Section 8.2]{Dolgachev-classical}, we let
$$
{\rm Exc}_{9-d} := \{ v \in {\rm I }^{1,9-d} \mid v^2 = -1, v.k_{9-d} = -1\} \subseteq {\rm I}^{1,9-d}
$$
be the subset of \emph{exceptional vectors}. Let $\mathcal{R}$ be a set of linearly independent $(-2)$-vectors in $E_{9-d}$, and define the cone
$$
C_{\mathcal{R}} := \{ v \in {\rm I}^{1,9-d} \otimes \mathbb{R} \mid v.w \geq 0 \text{ for all } w \in \mathcal{R} \}.
$$
For a sublattice $\Lambda$ of $E_{9-d}$, we denote the Weyl group of $\Lambda$ by $W(\Lambda)$. That is, $W(\Lambda)$ is the subgroup of the orthogonal group ${\rm O}({\rm I}^{1,9-d})$ generated by reflections along $(-2)$-vectors in $\Lambda$. With this notation, $W(\Lambda)$ preserves ${\rm Exc}_{9-d}$ and, for $\Lambda = \langle \mathcal{R} \rangle$, $C_{\mathcal{R}}$ is a fundamental domain for the action of $W(\Lambda)$ on ${\rm I}^{1,9-d} \otimes \mathbb{R}$.

\begin{Lemma} \label{lemma: (-1)-curves lattice theoretic}
With the above notation, we have the following description of certain sets of $(-1)$-curves on $\widetilde{X}$:
\begin{enumerate}
    \item 
    If $\mathcal{R}$ is the pre-image of the set of \emph{all} $(-2)$-curves on $\widetilde{X}$ under a geometric marking $\phi$, then $\phi$ induces a bijection
$$
\{ (-1)\text{-curves on } \widetilde{X} \}   \hspace{2mm} \longleftrightarrow \hspace{2mm} C_{\mathcal{R}} \cap {\rm Exc}_{9-d} \cong {\rm Exc}_{9-d}/W(\Lambda) .
$$

    \item
    If $\mathcal{R}' \subseteq \mathcal{R}$ with $\Lambda' := \langle \mathcal{R}' \rangle$, then $\phi$ induces a bijection
$$
\{ (-1)\text{-curves on } \widetilde{X} \text{ disjoint from } \phi(\mathcal{R}') \} \hspace{2mm} \longleftrightarrow \hspace{2mm} C_{\mathcal{R}} \cap {\rm Exc}_{9-d} \cap (\Lambda')^{\perp} = C_{\mathcal{R}} \cap {\rm Exc}_{9-d}^{W(\Lambda ')}.
$$

    \item
    If, moreover, $\Lambda'$ is a sum of connected components of $\Lambda$, then $\phi$ induces a bijection
    $$
\{ (-1)\text{-curves on } \widetilde{X} \text{ disjoint from } \phi(\mathcal{R}') \} \hspace{2mm} \longleftrightarrow \hspace{2mm} ({\rm Exc}_{9-d}^{W(\Lambda')})/W(\Lambda).
$$
\end{enumerate}
\end{Lemma}

\begin{proof}
For $(1)$ and $(2)$, see \cite[Lemma 8.2.22 and Proposition 8.2.34]{Dolgachev-classical}. 
To prove $(3)$, we note that we have an orthogonal decomposition $\Lambda = \Lambda' \oplus \Lambda''$, where $\Lambda = \langle \mathcal{R} \rangle$ and $\Lambda'' = \langle \mathcal{R} \setminus \mathcal{R}' \rangle$. Therefore, the $W(\Lambda)$-action preserves ${\rm Exc}_{9-d} \cap (\Lambda')^{\perp} = {\rm Exc}_{9-d}^{W(\Lambda')}$ and, by $(1)$, we can write
$$
C_{\mathcal{R}} \cap {\rm Exc}_{9-d} \cap (\Lambda')^\perp = (C_{\mathcal{R}} \cap {\rm Exc}_{9-d})^{W(\Lambda')} \cong ({\rm Exc}_{9-d}^{W(\Lambda')})/W(\Lambda).
$$
\end{proof}

\section{Group scheme actions on anti-canonical models}
The purpose of this section is to recall some basic facts about group scheme actions on (blow-ups of) normal projective surfaces and to describe the automorphism scheme of an RDP del Pezzo surface in terms of the anti-(bi-)canonical morphisms recalled in Theorem \ref{thm: anti-cano embeddings/coverings}.

Quite generally, the key tool to control the behaviour of group scheme actions under birational morphisms is Blanchard's Lemma \cite[Theorem 7.2.1]{Brion1}.

 \begin{Theorem}\emph{(Blanchard's Lemma)} \label{thm: BlanchardLemma}
Let $\pi: \widetilde{X} \to X$ be a morphism of proper schemes with $\pi_* \cal{O}_{\widetilde{X}} = \cal{O}_X$. Then, $\pi$ induces a homomorphism of group schemes $\pi_*: \Aut_{\widetilde{X}}^0 \to \Aut_X^0$. If $\pi$ is birational, then $\pi_*$ is a closed immersion.
\end{Theorem}

Given an action of a group scheme $G$ on a scheme $X$ and a closed subscheme $Z \subseteq X$, we let $\Stab_{G}(Z) \subseteq \Aut_X$ be the \emph{stabilizer subgroup scheme} of $Z$. The following proposition describes the image of $\pi_*$ if $\pi$ is a blow-up of a closed point on a normal surface with at worst rational double points.
\begin{Proposition} \label{prop: StabiLemma}
Let $X$ be a normal surface with at worst rational double points and let $\pi: \widetilde{X} \to X$ be the blow-up of $X$ in a closed point $P$. Then, $\pi_*(\Aut_{\widetilde{X}}^0) = (\Stab_{\Aut_X}(P))^0$.
\end{Proposition}

\begin{proof}
By \cite[Proposition 2.7]{Martin}, it suffices to find an infinitesimal rigid subscheme $E \subseteq \widetilde{X}$ whose schematic image is $P$. Thus, let $E$ be the exceptional divisor of $\pi$, with scheme structure given by the inverse image ideal sheaf of $P$. In particular, $E$ is a Cartier divisor on $\widetilde{X}$.

If $P$ is smooth, then $E$ is a $(-1)$-curve, hence $E$ is infinitesimally rigid. 
If $P$ is not smooth, then $P$ is a rational double point. In particular, $\pi$ factors the minimal resolution $\pi': \widetilde{X}' \to X$ and $\pi'^* \omega_{X} \cong \omega_{\widetilde{X}'}$. Since $\widetilde{X}$ has at worst rational double points as well and $\widetilde{X}' \to \widetilde{X}$ is its minimal resolution, we obtain $\pi^* \omega_X \cong \omega_{\widetilde{X}}$ using the projection formula. 
Now, $\omega_X$ is trivial in a neighborhood of $P$, so $\omega_{\widetilde{X}}$ is trivial in a neighborhood of $E$. Thus, the normal sheaf $\mathcal{N}_{E/\widetilde{X}}$ of $E$ in $\widetilde{X}$ coincides with $\omega_E$ by adjunction.

On the other hand, the rational double point $P \in X$ has multiplicity $2$ and embedding dimension $3$, hence $E$ is isomorphic to a (possibly non-reduced) conic in $\mathbb{P}^2$. As such, it satisfies $\omega_E \cong \mathcal{O}_{\mathbb{P}^2}(-1)|_E$. Hence, $$h^0(E,\mathcal{N}_{E/\widetilde{X}}) = h^0(E,\mathcal{O}_{\mathbb{P}^2}(-1)|_E) = 0,$$ so, by \cite[Proposition 3.2.1.(ii)]{Sernesi}, $E$ is infinitesimally rigid. This finishes the proof. 
\end{proof}

If $X$ is an RDP del Pezzo surface, then we have the following description of $\Aut_X$ in terms of the anti-canonical morphisms $\varphi_{|-nK_X|}$ of $X$.

\begin{Proposition} \label{prop: ShortExactAutSequences}
Let $X$ be an RDP del Pezzo surface of degree $d$. Then, $\varphi_{|-nK_X|}$ is $\Aut_X$-equivariant for all $n \geq 0$ and the following hold.
\begin{enumerate}
    \item If $d \geq 3$, then $\Aut_X = \Stab_{{\rm PGL}_{d+1}}(X)$.
    \item If $d = 2$ and $p \neq 2$, then there is an exact sequence of group schemes
    $$
    1 \to \underline{\mathbb{Z}/2\mathbb{Z}} \to \Aut_X \to \Stab_{{\rm PGL}_3}(Q)\to 1,
    $$
    where $Q$ is the branch quartic of the anti-canonical morphism $X \to \mathbb{P}^2$.
    \item If $d = 1$ and $p \neq 2$, then there is an exact sequence of group schemes
    $$
    1 \to \underline{\mathbb{Z}/2\mathbb{Z}} \to \Aut_X \to \Stab_{\Aut_{\mathbb{P}(1,1,2)}}(S) \to 1,
    $$
    where $S$ is the branch sextic of the anti-bi-canonical morphism $X \to \mathbb{P}(1,1,2) \subseteq \mathbb{P}^3$.
\end{enumerate}
\end{Proposition}

\begin{proof}
By \cite[Remark 2.15.(iv)]{BrionNotes-linearization}, the line bundles $\omega_X^{\otimes (- n)}$ admit natural $\Aut_X$-linearizations for all $n \geq 0$ and hence the natural action of $\Aut_X$ on the space of global sections of $\omega_X^{\otimes (-n)}$ induces a homomorphism $f_n: \Aut_X \to {\rm PGL}_{N+1}$ for $N = \dim(H^0(X,\omega_X^{\otimes (-n)})) - 1$ making the rational map $\varphi_{|-nK_X|}:X \dashrightarrow \mathbb{P}^N$ $\Aut_X$-equivariant.

If $d \geq 3$, then the anti-canonical map is an embedding by Theorem \ref{thm: anti-cano embeddings/coverings}, hence $f_1$ is a monomorphism. By \cite[Lemma 2.5]{Martin}, $f_1$ factors through $\Stab_{{\rm PGL}_{d+1}}(X)$. Conversely, restricting the $G$-action on $\mathbb{P}^d$ to $X$ yields a left-inverse $g_1:  \Stab_{{\rm PGL}_{d+1}}(X) \to \Aut_X$ to $f_1$. Since $X$ is not contained in a proper linear subspace of $\mathbb{P}^d$ and the fixed locus of a subgroup scheme of ${\rm PGL}_{d+1}$ is a linear subspace, $g_1$ has to be a monomorphism. Hence $f_1$ is an isomorphism.

If $d =2$ and $p \neq 2$, then the anti-canonical map is a finite flat cover of $\mathbb{P}^2$ branched over a quartic curve $Q \subseteq \mathbb{P}^2$ by Theorem \ref{thm: anti-cano embeddings/coverings}. Let $K$ be the kernel of $f_1$. Restricting the action of $K$ on $X$ to the generic point of $X$ yields a $k(\mathbb{P}^2)$-linear action of $K_{k(\mathbb{P}^2)}$ on the degree $2$ field extension $k(X)$ of $k(\mathbb{P}^2)$. Since $p \neq 2$, the field extension $k(\mathbb{P}^2) \subseteq k(X)$ is Galois, which shows $K = \underline{\mathbb{Z}/2\mathbb{Z}}$.
Since $K$ is normal in $\Aut_X$, the action of $\Aut_X$ on $X$ preserves the fixed locus $X^K$, hence, by \cite[Lemma 2.5]{Martin}, the induced action of $\Aut_X$ on $\mathbb{P}^2$ preserves the scheme-theoretic image of $X^K$ under $\varphi_{|-K_X|}$, which is nothing but $Q$. Hence, $f_1$ factors through $\Stab_{{\rm PGL}_3}(Q)$. In order to show faithful flatness of $f_1': \Aut_X \to \Stab_{{\rm PGL}_3}(Q)$, we write $X$ as $\{ w^2 = Q(x,y,z) \} \subseteq \mathbb{P}(1,1,1,2)$. For every $k$-algebra $R$ and every automorphism $\sigma$ of $\mathbb{P}^2_R$ preserving $Q_R = \{Q(x,y,z) = 0 \} \subseteq \mathbb{P}^2_R$, that is, mapping $Q(x,y,z)$ to $\lambda Q(x,y,z)$ for some $\lambda \in R^{\times}$, we can pass to the faithfully flat ring extension $R' := R[\sqrt{\lambda}]$ of $R$ and lift $\sigma$ to an automorphism of $X_{R'}$ by mapping $w$ to $\pm \sqrt{\lambda} w$. Hence, $f_1'$ is faithfully flat and thus the sequence in (2) is exact. 

If $d = 1$ and $p \neq 2$, we can apply essentially the same argument as in the previous paragraph to the anti-bi-canonical morphism $\varphi_{|-2K_{X}|}: X \to \mathbb{P}^3$: Indeed, the argument for $K = \underline{\mathbb{Z}/2\mathbb{Z}}$ is exactly the same as in the previous paragraph. To prove faithful flatness of $f_2': \Aut_X \to \Stab_{\Aut_{\mathbb{P}(1,1,2)}}(S)$, we can write $X$ as a hypersurface in $\mathbb{P}(1,1,2,3)$ by identifying the quadratic cone with $\mathbb{P}(1,1,2)$ and then argue as above.
\end{proof}

\section{On equivariant resolutions}
It is an immediate consequence of the uniqueness of the minimal resolution $\widetilde{X}$ of a projective surface $X$ that the action of the automorphism \emph{group} $\Aut_X(k)$ on $X$ lifts to $\widetilde{X}$. Over the complex numbers, this implies that the action of the automorphism \emph{scheme} $\Aut_X$ lifts to $\widetilde{X}$. In general, this is no longer true in positive characteristic. 
In this section, we will study this phenomenon.

\begin{Definition}
Let $\pi: \widetilde{X} \to X$ be a proper birational morphism of schemes.
\begin{enumerate}
    \item The morphism $\pi$ is called \emph{$T_X$-equivariant} if the natural map $\pi_* T_{\widetilde{X}} \to T_X$ is an isomorphism.
    \item Assume additionally that $X$ is proper and $\pi_* \mathcal{O}_{\widetilde{X}} = \mathcal{O}_X$. Then, $\pi$ is called \emph{$\Aut_X^0$-equivariant} if the closed immersion $\pi_*: \Aut_{\widetilde{X}}^0 \hookrightarrow \Aut_X^0$ induced by Blanchard's Lemma is an isomorphism.
\end{enumerate}
\end{Definition}

\begin{Remark}\label{remark: T_X is local notion}
Note that $T_X$-equivariance is local on $X$ and implies $H^0(\widetilde{X},T_{\widetilde{X}}) \cong H^0(X,T_X)$. If $\pi$ is $T_X$-equivariant, then $\pi_*: \Aut_{\widetilde{X}}^0 \hookrightarrow \Aut_X^0$ is an isomorphism on tangent spaces.
\end{Remark}

The study of the $T_X$-equivariance of the minimal resolution of a rational double point has been initiated by Wahl \cite{Wahl} and extended to all positive characteristics by Hirokado \cite{Hirokado}. There, $T_X$-equivariance is simply called ``equivariance''.
For the convenience of the reader, we will recall the classification of RDPs whose minimal resolution is not $T_X$-equivariant (see \cite[Theorem 1.1]{Hirokado}).

\begin{Proposition} \label{prop: Hirokados RDPs with non lift vector fields}
Let $\pi: \widetilde{X} \to X$ be the minimal resolution of a rational double point $(X, x)$. Then, $\pi$ is not $T_X$ equivariant if and only if $(X,x)$ is of type
\begin{enumerate}
    \item $A_n$ with $p \mid (n+1)$,
    \item $E_8^0$ if $p = 5$,
    \item $E_6^0, E_6^1, E_7^0, E_8^0, E_8^1$ if $p = 3$, or
    \item $D_n^r, E_6^0, E_7^0, E_7^1, E_7^2, E_7^3, E_8^0, E_8^1, E_8^2, E_8^3$ if $p = 2$.
\end{enumerate}
\end{Proposition}

In the next sections, we would like to apply the notions of $\Aut_X^0$-equivariance and $T_X$-equivariance to RDP del Pezzo surfaces and their partial resolutions.

\begin{Definition}
Let $X$ be a proper surface. A \emph{partial resolution} of $X$ is a proper birational morphism $\pi: \widetilde{X} \to X$ such that the minimal resolution of $X$ factors through $\pi$.
\end{Definition}

\begin{Proposition} \label{prop: Wahl Hirokado}
Let $X$ be a normal proper surface and let $\pi: \widetilde{X} \to X$ be a partial resolution of $X$. Assume that there exists an open subset $U \subseteq X$ such that $\pi^{-1}(U) \to U$ is an isomorphism and all singularities in $X \setminus U$ admit a $T_X$-equivariant minimal resolution. Then, $\pi$ is $T_X$-equivariant.
\end{Proposition}

\begin{proof}
Let $U_{sing}$ be the set of singular points in $U$ and let $V = X \setminus U_{sing}$. Let $\psi: V' \to V$ be the minimal resolution of $V$. In particular, $V' \supseteq (\psi)^{-1}(U \setminus U_{sing}) \to (U \setminus U_{sing}) \subseteq U$ is an isomorphism, which we can use to glue $V'$ and $U$ along the smooth locus $(U \setminus U_{sing})$ of $U$ to a projective surface $X'$ together with a proper birational morphism $\pi': X' \to X$.
Since $\pi$ is a partial resolution of $X$, $\pi'$ factors through $\pi$ by construction.

Moreover, by the assumption on the singularities in $X \setminus U$ and Remark $\ref{remark: T_X is local notion}$, $\pi'$ is $T_X$-equivariant, that is, the composition
$$
\pi'_* T_{X'} \to \pi_* T_{\widetilde{X}} \to T_X
$$
is an isomorphism, hence the second map is surjective and thus an isomorphism, since $T_X$ satisfies the $(S_2)$-condition.
In particular, if additionally $\Aut_{\widetilde{X}}^0$ is smooth, this implies that $\pi$ is $\Aut_X^0$-equivariant by Proposition \ref{prop: T-equivariance implies Aut-equivariance}.
\end{proof}

In some situations, $\Aut_X^0$-equivariance can be deduced immediately from the simpler notion of $T_X$-equivariance, which is the content of the following proposition.

\begin{Proposition} \label{prop: T-equivariance implies Aut-equivariance}
Let $\pi: \widetilde{X} \to X$ be a birational morphism of proper $k$-schemes. If $\Aut_{\widetilde{X}}^0$ is smooth and $\pi$ is $T_X$-equivariant, then $\pi$ is $\Aut_X^0$-equivariant.
\end{Proposition}

\begin{proof}
If $\pi$ is $T_X$-equivariant, then 
$$
\dim \Aut_{\widetilde{X}}^0 \leq \dim \Aut_X^0 \leq \dim_k H^0(X,T_X) = \dim_k H^0(\widetilde{X},T_{\widetilde{X}})
$$
and since $\Aut_{\widetilde{X}}^0$ is smooth, all inequalities above are in fact equalities. Thus, $\Aut_X^0$ is smooth and of the same dimension as $\Aut_{\widetilde{X}}^0$. Hence, we must have $\Aut_{\widetilde{X}}^0 = \Aut_X^0$, that is, $\pi$ is $\Aut_X^0$-equivariant. 
\end{proof}

\begin{Remark}
In particular, if, in the situation of Proposition \ref{prop: Wahl Hirokado}, we assume in addition that $\Aut_{\widetilde{X}}^0$ is smooth, then $\pi$ is $\Aut_X^0$-equivariant.
\end{Remark}

To the best of our knowledge, the question whether partial resolutions of a proper normal surface with rational double points are $\Aut_X^0$-equivariant has not been studied. In the following, we prove $\Aut_X^0$-equivariance for $A_n$-singularities with $n < p - 1$ and bound the failure of $\Aut_X^0$-equivariance for $A_{p-1}$-singularities. While this does not cover all rational double points and not even all $A_n$-singularities, it will come in handy for the calculation of the automorphism schemes of non-equivariant RDP del Pezzo surfaces in Section \ref{sec: nonequivariant}.

\begin{Proposition} \label{prop: fitting argument A_n}
Let $X$ be a proper surface. Let $\pi: \widetilde{X} \to X$ be a partial resolution of $X$ and assume that the only singularities over which $\pi$ is not an isomorphism are $A_n$-singularities with $n \leq p-1$. 
Then, 
$$
{\rm length}\left( \Aut_X^0 / \Aut_{\widetilde{X}}^0 \right) \leq p^m,
$$
where $m$ is the number of $A_{p-1}$-singularities on $X$ over which $\pi$ is not an isomorphism.
In particular, if $m = 0$, then $\pi$ is $\Aut_X$-equivariant.
\end{Proposition} 

\begin{proof}
It suffices to prove the statement if $\pi: \widetilde{X} \to X$ is not an isomorphism only over a single singularity $P$ of type $A_n$. By Proposition \ref{prop: StabiLemma}, it suffices to show that $G := \Aut_X^0$ fixes $P$ if $n < p - 1$ and that the stabilizer of $P$ has index $1$ or $p$ in $G$ if $n = p - 1$.
To see this, we equip the singular locus $X_{sing}$ of $X$ with a scheme structure using Fitting ideals and we let $Y$ be the irreducible component of $X_{sing}$ containing $P$.
Since the scheme structure on $X_{sing}$ is canonical and $G$ is connected, $G$ preserves $Y$, so we get a homomorphism $\varphi: G \to \Aut_Y$. To prove the proposition, it suffices to show that the stabilizer of $P$ in $\Aut_Y$ has index $1$ or $p$, with the latter only occurring for $n = p-1$.

Since an $A_n$-singularity is given in a formal neighborhood by the equation $z^{n+1} + xy$, $Y$ is isomorphic to
\begin{eqnarray*}
Y_{n} := \Spec \left( k[z]/(z^{n+1}) \right) & \text{ if }  & n < p - 1 \\
Y_{p-1} := \Spec \left( k[z]/(z^p) \right) & \text{ if } & n = p-1.
\end{eqnarray*}
Now, we calculate $\Aut_{Y_i}$ by computing its $R$-valued points for an arbitrary local $k$-algebra $R$. An element of $\Aut_{Y_i}(R)$ is an $R$-linear automorphism $\varphi$ of $R[z]/(z^{i+1})$, hence it is determined by where it sends $z$. Let $a_0,\hdots,a_i \in R$ such that
$ \varphi(z) = \sum_{j=0}^{i} a_j z^j$.
Let $\mathfrak{m}$ be the maximal ideal of $R$, so that $(\mathfrak{m},z)$ is the maximal ideal of $R[z]/(z^{i+1})$.
Since $\varphi$ is an automorphism, it maps $(\mathfrak{m},z)$ to itself, hence $a_0 \in \mathfrak{m}$. If $a_1 \in \mathfrak{m}$, then the coefficient of $z$ in every $\varphi(z^j)$ is in $\mathfrak{m}$, so $z$ would not lie in the image of $\varphi$, which is absurd. Hence, $a_1 \in R \setminus \mathfrak{m} = R^\times$. 
Next, we know that $\varphi(z^{i+1}) = 0$. Since the degree $0$ term of $\varphi(z^{i+1})$ is $a_0^{i+1}$, we have $a_0^{i+1} = 0$. The degree $j$ term of $\varphi(z^{i+1})$ is of the form
$
\binom{i+1}{j} a_0^{i + 1 -j}a_1^j + a_0^{i + 2 -j}b_j
$
for some $b_j \in R$. If $i < p - 1$, then $p \nmid \binom{i+1}{j}$ for all $j$, hence solving the above equations inductively shows $a_0 = 0$. If $i = p - 1$, then $p \mid \binom{i+1}{j}$ for all $j > 0$, so we only get $a_0^p = 0$.

Conversely, given a sequence $(a_0,\hdots,a_i)$ in $R$ with $a_1 \in R^\times$ and $a_0 = 0$ if $i < p-1$ (resp. $a_0^p = 0$ if $i = p$), the morphism induced by $z \mapsto \sum_{j=0}^i a_j z^j$ is an automorphism, since it is well-defined and its inverse is given by $z \mapsto (\sum_{j=1}^i a_j z^{j-1})^{-1} z - a_0$. 

Summarizing, we have natural identifications
\begin{eqnarray*}
\Aut_{Y_i}(R) &=& \{(0,a_1,\hdots,a_{i}) \mid a_j \in R,  a_1 \in R^\times\} \text{ for } i < p - 1\\
\Aut_{Y_{p-1}}(R) &=& \{(a_0,a_1,\hdots,a_{p-1}) \mid a_j \in R, a_1 \in R^\times, a_0^p = 0 \}.
\end{eqnarray*}
In both cases, the corresponding automorphism of $Y_i$ preserves $P \times \Spec R \subseteq Y_i \times \Spec R$ if and only if $a_0 = 0$, since the ideal of $P \times \Spec R$ is $(z)$. 
Hence, the index of the stabilizer of $P$ in $\Aut_{Y_i}$ is $1$ if $i < p - 1$, and $p$ if $i = p - 1$. This finishes the proof. 
\end{proof}

\begin{Remark}\label{rem: A5}
The strategy of proof of Proposition \ref{prop: fitting argument A_n} would, in principle, also apply to other rational double points. However, there are two obstacles to overcome:
\begin{enumerate}
    \item The automorphism scheme of the singular locus is more complicated for more general RDPs, since the singular locus has a more complicated scheme structure in general. For example, if $p = 5$ and $X$ admits an RDP of type $E_8^0$, then, in a neighorhood of this singularity, the singular locus of $X$ looks like $\Spec k[[x,y]]/(x^2,y^5)$. This also makes the calculation of the stabilizer of the closed point more complicated.
    \item If $\pi: \widetilde{X} \to X$ is the minimal resolution of an RDP surface, then $\Aut_{\widetilde{X}}^0$ is the intersection of all stabilizers of all singularities that occur in the blow-ups making up $\pi$. For example, if $p = 3$ and $X$ admits a single RDP of type $A_4$, then the argument of Proposition \ref{prop: fitting argument A_n} shows that $\Aut_{X'}^0 = \Aut_X^0$, where $X'$ is the blow-up of the closed point of $P$, but $X'$ has a singularity of type $A_2$, so the approach of Proposition \ref{prop: fitting argument A_n} only shows that ${\rm length}\left( \Aut_X^0 / \Aut_{\widetilde{X}}^0 \right) \leq 3$ even though we would expect the two group schemes to be equal by Proposition \ref{prop: Hirokados RDPs with non lift vector fields}.
\end{enumerate}
One case where the argument of Proposition \ref{prop: fitting argument A_n} goes through essentially unchanged is if $p = 3$ and the morphism $\pi: \widetilde{X} \to X$ is the minimal resolution of an RDP of type $A_5$. In this case, $\pi$ factors as a composition of three blow-ups $\widetilde{X} \to X'' \to X' \to X$, where $X'$ has an $A_3$-singularity and $X''$ has an $A_1$-singularity. Then, the argument of Proposition \ref{prop: fitting argument A_n} shows that $\Aut_{\widetilde{X}}^0 = \Aut_{X''}^0 = \Aut_{X'}^0$ and ${\rm length}\left( \Aut_X^0 / \Aut_{X'}^0 \right) \leq 3$, hence ${\rm length}\left( \Aut_X^0 / \Aut_{\widetilde{X}}^0 \right) \leq 3$.
\end{Remark}

\newpage

\section{Automorphism schemes of equivariant RDP del Pezzo surfaces}
\label{sec:equivariant}

In \cite{MartinStadlmayr}, we classified all weak del Pezzo surfaces $\widetilde{X}$ with global vector fields and calculated the identity component $\Aut_{\widetilde{X}}^0$ of their automorphism schemes. 
In particular, if $X$ is a projective surface, whose minimal resolution $\pi: \widetilde{X} \to X$ is $\Aut_X^0$-equivariant and such that $\widetilde{X}$ is a weak del Pezzo surface, then $\Aut_X^0 = \Aut_{\widetilde{X}}^0$ and thus, if $\Aut_X^0$ is non-trivial, then $\widetilde{X}$ appears in the classification tables of \cite{MartinStadlmayr}.

In the following, we will observe that all RDP del Pezzo surfaces in characteristic $p \geq 11$ fall into the above category and we will give a list of possible candidates for exceptions in small characteristics.

\begin{Theorem} \label{thm: equivariantmain}
Let $X$ be an RDP del Pezzo surface over an algebraically closed field of characteristic $p$ and let $\pi: \widetilde{X} \to X$ be its minimal resolution.
Assume that one of the following conditions holds:
\begin{enumerate}
    \item 
    $p \not \in \{ 2,3,5,7\}$.
    \item
    $p = 7$ and $X$ does not contain an RDP of type $A_6$.
    \item
    $p = 5$ and $X$ does not contain an RDP of type $A_4$ or $E_8^0$.
    \item
    $p=3$ and $X$ does not contain an RDP of type $A_2, A_5, A_8, E_6^0, E_6^1, E_7^0, E_8^0$ or $E_8^1$.
    \item
    $p = 2$ and $X$ does not contain an RDP of type $A_1,A_3,A_5,A_7,D_n^r,E_6^0,E_7^0,E_7^1,E_7^2,E_7^3,E_8^0,E_8^1,E_8^2$ or $E_8^3$.
\end{enumerate}
Then, $\Aut_X = \Aut_{\widetilde{X}}$, and thus, in particular, $H^0(\widetilde{X},T_{\widetilde{X}}) = H^0(X,T_X)$. Therefore, $H^0(X,T_X) \neq 0$ if and only if $X$ is the anti-canonical model of one of the surfaces in the classification tables of \cite{MartinStadlmayr}.
\end{Theorem}

\begin{proof}
By Proposition \ref{prop: Wahl Hirokado} and Proposition \ref{prop: T-equivariance implies Aut-equivariance}, the theorem holds for those RDP del Pezzo surfaces that satisfy the following two conditions: 
\begin{enumerate}
    \item[(a)] all singularities of $X$ admit a $T_X$-equivariant minimal resolution, and
    \item[(b)] $\Aut_{\widetilde{X}}^0$ is smooth.
\end{enumerate}
By Proposition \ref{prop: Hirokados RDPs with non lift vector fields}, Condition (a) holds if we exclude the types of RDPs in the statement of the theorem.

Once we exclude those types of RDPs, then, by Tables $1-6$ in \cite{MartinStadlmayr}, Condition (b) is also satisfied unless we are in one of the following three cases, where $\Gamma$ is the RDP configuration on $X$ and $d = \deg(X)$:
\begin{enumerate}
    \item[(i)] $p = 3$, $d = 2$, $\Gamma = A_6$
    \item[(ii)] $p = 3$, $d = 2$, $\Gamma = D_6$
    \item[(iii)] $p = 2$, $d = 3$, $\Gamma = A_4$
\end{enumerate}
In \cite{MartinStadlmayr}, these exceptions correspond to cases $2J, 2K$, and $3N$, respectively, and there is a unique weak del Pezzo surface of each of these types. In all cases, we have $H^0(\widetilde{X},T_{\widetilde{X}}) = 1$, hence $H^0(X,T_X) = 1$ by Proposition \ref{prop: Hirokados RDPs with non lift vector fields}, so the only remaining statement we have to show in these three cases is that $\pi$ is $\Aut_X^0$-equivariant. We will check this via explicit calculation:
\begin{enumerate}
\item[(i)] Assume $p = 3$. Consider the surface
$$
X := \{w^2 = x^2z^2 + xy^2z + y^4 + x^3y\} \subseteq \mathbb{P}(1,1,1,2)
$$
and let $Q$ be the branch quartic of the induced double cover $X \to \mathbb{P}^2$. 
An elementary calculation shows that $X$ admits an $A_6$-singularity over $[0:0:1]$ and no other singularity. By Proposition \ref{prop: ShortExactAutSequences}, we have $\Aut_X^0 = \Stab_{\PGL_3}(Q)^0$. The diagonal $\mu_3$-action with weights $(0,1,2)$ on $\mathbb{P}^2$ preserves $Q$, hence $\mu_3 \subseteq \Aut_X^0$. Since $\pi$ is $T_X$-equivariant by Proposition \ref{prop: Hirokados RDPs with non lift vector fields}, $\widetilde{X}$ must be the surface of type $2J$ of \cite{MartinStadlmayr}, hence $\Aut_{\widetilde{X}}^0 = \mu_3$. Since $H^0(X,T_X) = 1$, we have $\Aut_X^0[F] = \mu_3$, where $\Aut_X^0[F]$ denotes the kernel of Frobenius on $\Aut_X^0$. Hence, $\mu_3$ is normal in $\Aut_X^0 = \Stab_{\PGL_3}(Q)^0$ and thus $\Stab_{\PGL_3}(Q)^0$ preserves the eigenspaces of the $\mu_3$-action, hence $\Stab_{\PGL_3}(Q)^0$ acts diagonally. With this restriction, it is easy to compute that $\Aut_X^0 =\Stab_{\PGL_3}(Q)^0 =   \mu_3$. Therefore, $\pi$ is $\Aut_X^0$-equivariant, which is what we wanted to show.
\item[(ii)] Assume $p = 3$. Consider the surface
$$
X := \{w^2 = x(x^3 + y^3 + xyz)\} \subseteq \mathbb{P}(1,1,1,2)
$$
and let $Q$ be the branch quartic of the induced double cover $X \to \mathbb{P}^2$. Note that $Q$ is the union of a nodal cubic and one of its nodal tangents, with the node located at $[0:0:1]$, hence $X$ has a $D_6$-singularity at $[0:0:1:0]$ and no other singularities. The diagonal $\mu_3$-action with weights $(0,1,2)$ preserves $Q$, hence $H^0(X,T_X) \neq 0$, and thus $\widetilde{X}$ is the surface of type $2K$ of \cite{MartinStadlmayr}. The rest of the argument is the same as in the previous Case (i), and shows that $\pi$ is $\Aut_X^0$-equivariant. 
\item[(iii)] Assume $p = 2$. Consider the surface
$$
X := \{x_0x_1x_3 + x_1^2x_2 + x_0x_2^2 + x_0^2x_2\} \subseteq \mathbb{P}^3,
$$
which is a cubic surface with a single singularity, which is of type $A_4$, at $[0:0:0:1]$ (see e.g. \cite[Case B]{Roczencubic}). It admits a diagonal $\mu_2$-action with weights $(0,1,0,1)$, hence $H^0(X,T_X) \neq 0$ and therefore, as $\pi$ is $T_X$-equivariant by Proposition \ref{prop: Hirokados RDPs with non lift vector fields}, $X$ is the anti-canonical model of the surface of type $3N$ in \cite{MartinStadlmayr}. Straightforward calculation, again using that $\mu_2$ is normal in $\Stab_{{\rm PGL}_4}(X)^0$, shows that $\Aut_X^0 = \Stab_{{\rm PGL}_4}(X)^0 = \mu_2$, hence $\pi$ is $\Aut_X^0$-equivariant.
\end{enumerate}
\end{proof}

 \section{Finding $(-1)$-curves in the equivariant locus}
In view of Theorem \ref{thm: equivariantmain}, in order to classify RDP del Pezzo surfaces with global vector fields in odd characteristic, it remains to study RDP del Pezzo surfaces $X$ which are not $T_X$-equivariant. Let $\Gamma'$ be the configuration of rational double points on $X$ which are not $T_X$-equivariant and let $\pi: \widetilde{X} \to X$  the minimal resolution of $X$.

In this section, we will describe a criterion for $X$ to be the anti-canonical model of a blow-up of an RDP del Pezzo surface $X'$ of higher degree containing $\Gamma'$ such that $X' \dashrightarrow X$ is an isomorphism around $\Gamma'$. On the corresponding minimal resolutions, this will amount to finding $(-1)$-curves away from the configuration of $(-2)$-curves over $\Gamma'$. In other words, we are trying to find $(-1)$-curves on $\widetilde{X}$ that map to the $T_X$-equivariant locus of $\pi$.
In Section \ref{sec: nonequivariant}, this criterion will allow us to give a complete classification of non-equivariant RDP del Pezzo surfaces with global vector fields by setting up an inductive argument depending on the degree of the surface.

\begin{Theorem} \label{thm: lattice theoretic reduction criterion}
Let $X_d$ be an RDP del Pezzo surface of degree $d \leq 8$ and let $\Gamma'$ be a configuration of rational double points on $X_d$. Assume that its minimal resolution $\widetilde{X}_d$ is a blow-up of $\mathbb{P}^2$ and let $\Lambda' \subseteq \Pic(\widetilde{X}_d)$ be the sublattice generated by the components of the exceptional locus \mbox{over $\Gamma'$.}
Then the following are equivalent: 
\begin{enumerate}
    \item There exists a $(-1)$-curve on $\widetilde{X}_d$ whose image in $X_d$ does not pass through $\Gamma'$.
    \item $X_d$ is the anti-canonical model of a blow-up in a smooth point $P$ of an RDP del Pezzo surface $X_{d+1}$ of degree $(d + 1)$ containing $\Gamma'$ such that $X_d \dashrightarrow X_{d+1}$ is an isomorphism around $\Gamma'$.
    \item The map $\Lambda' \hookrightarrow \langle K_{\widetilde{X}_d} \rangle^\perp  \cong E_{9-d}$  factors through an embedding $E_{8-d} \hookrightarrow E_{9-d}$.
\end{enumerate}
\end{Theorem}

\begin{proof}
First, we show $(1) \Rightarrow (2)$. Let $\widetilde{C}$ be the $(-1)$-curve whose existence is asserted in $(1)$. Contracting $\widetilde{C}$, we obtain a weak del Pezzo surface $\widetilde{X}_{d+1}$ of degree $(d+1)$ such that $\widetilde{X}_d$ is the blow-up of $\widetilde{X}_{d+1}$ in a smooth point $\widetilde{P}$. Let $X_{d+1}$ be the anti-canonical model of $\widetilde{X}_{d+1}$ and let $P$ be the image of $\widetilde{P}$ in $X_{d+1}$. By our choice of $\widetilde{C}$, all components of the exceptional locus over $\Gamma'$ stay $(-2)$-curves in $\widetilde{X}_{d+1}$, hence $X_{d+1}$ contains $\Gamma'$.

Since $\widetilde{X}_{d}$ is a weak del Pezzo surface, $\widetilde{P}$ cannot lie on a $(-2)$-curve (otherwise the strict transform of such a curve would have negative intersection with $-K_{\widetilde{X}_d}$, which is impossible as $-K_{\widetilde{X}_d}$ is nef), hence $P$ is a smooth point on $X_{d+1}$.
Thus, blowing up $P \in X_{d+1}$, we obtain a surface $Y_d$ with the same singularities as $X_{d+1}$. In particular, $Y_d$ has only rational double points as singularities and its minimal resolution is $\widetilde{X}_{d}$. Therefore, pullback of sections induces isomorphisms $H^0(Y_d,-n K_{Y_d}) \cong H^0(\widetilde{X}_d, -n K_{\widetilde{X}_d})$
for all $n \geq 0$, where the surjectivity follows from the fact that $Y_d$ is normal. Thus, the anti-canonical model of $Y_d$ coincides with $X_d$. The situation is summarized in the following Figure \ref{figure: contracting-1}.
\begin{figure}[h!]
\vspace{-1mm}
$
  \xymatrix{
      \widetilde{C}  \subseteq  \widetilde{X}_d\text{\hspace{7mm}} \ar[rr]^{\text{contract } \widetilde{C}}_{\text{blow-up } \widetilde{P}} \ar[d] \ar@/_2.8pc/[dd]_{\text{anti-can.}}  & &   \text{\hspace{6mm}} \widetilde{X}_{d+1} \ar[d]^{\text{anti-can.}}  \ni  \widetilde{P}   \\
      C \subseteq  Y_d \text{\hspace{7mm}} \ar[rr]^{\text{contract } {C}}_{\text{blow-up } {P}} \ar[d]^{\text{anti-can.}}       &    &  \text{\hspace{6mm}} X_{d+1} \ni  P \\
       X_d & 
  }
$
\vspace{-2mm}
\caption{Contracting a $(-1)$-curve disjoint from the singular locus}
\label{figure: contracting-1}
\end{figure}
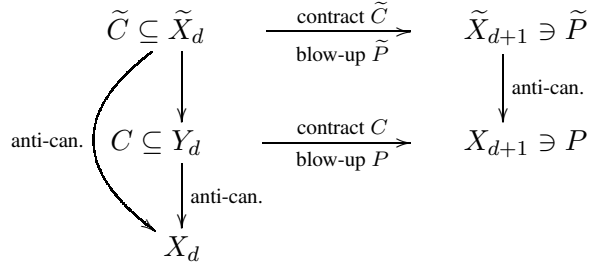

 Note that $X_d \dashrightarrow Y_d \rightarrow X_{d+1}$ is an isomorphism in a neighborhood of $\Gamma'$, since $\widetilde{C}$ is disjoint from the exceptional locus over $\Gamma'$.

Next, we show $(2) \Rightarrow (3)$. We have $\langle K_{\widetilde{X}_{d}} \rangle^\perp  \cong E_{9-d}$ and $\langle K_{\widetilde{X}_{d + 1}} \rangle^\perp \cong E_{8-d}$. Since $X_{d+1}$ contains $\Gamma'$, the embedding $\Lambda' \hookrightarrow \Pic(\widetilde{X}_{d})$ factors through the pullback map $\Pic(\widetilde{X}_{d+1}) \hookrightarrow \Pic(\widetilde{X}_{d})$, which maps $\langle K_{\widetilde{X}_{d+1}} \rangle^\perp$ to $\langle K_{\widetilde{X}_{d}} \rangle^\perp$. Hence $(3)$ follows. 

Finally, to show that $(3) \Rightarrow (1)$, we identify $\Pic(\widetilde{X}_d)$ and ${\rm I}^{1,9-d}$ via a geometric marking. We have to show that there is a $(-1)$-curve $\widetilde{C}$ on $\widetilde{X}_d$ that does not meet the set $\mathcal{R}'$ of exceptional curves over $\Gamma'$. Let $\Lambda$ be the sublattice of $\Pic(\widetilde{X}_d)$ spanned by the classes of all $(-2)$-curves. 
Since $\Lambda'$ is a sum of connected components of $\Lambda$, Lemma \ref{lemma: (-1)-curves lattice theoretic} shows that it suffices to prove
$$
({\rm Exc}_{9-d}^{W(\Lambda')})/W(\Lambda) \neq \emptyset.
$$
Clearly, this is the case if and only if ${\rm Exc}_{9-d}^{W(\Lambda')} \neq \emptyset$.
Since $\Lambda' \hookrightarrow E_{9-d}$ factors through an embedding $E_{8-d} \hookrightarrow E_{9-d}$, we have
$$
{\rm Exc}_{9-d}^{W(E_{8-d})} \subseteq {\rm Exc}_{9-d}^{W(\Lambda')},
$$
so it suffices to show that ${\rm Exc}_{9-d}^{W(E_{8-d})} \neq \emptyset$. Since the action of $W(E_{9-d})$ on $\Pic(\widetilde{X}_d)$ preserves ${\rm Exc}_{9-d}$, the condition ${\rm Exc}_{9-d}^{W(E_{8-d})} \neq \emptyset$ depends on the embedding $E_{8-d} \hookrightarrow E_{9-d}$ only up to conjugation by elements of $W(E_{9-d})$ and up to automorphisms of $E_{8-d}$.

If $d \leq 5$, then $E_{8-d}$ is a root lattice. By \cite[Table 11]{DynkinSemisimple} and \cite[Exercise 4.2.1, 4.6.2]{Martinet}, the embedding $\iota: E_{8-d} \hookrightarrow E_{9-d}$ is unique up to the action of ${\rm O}(E_{9-d})$. Since ${\rm O}(E_{9-d})$ is generated by $\{ \pm {\rm id} \}$ and $W(E_{9-d})$ in every case (see e.g. \cite[Proposition 4.2.2, Theorem 4.3.3, 4.5.2, 4.5.3]{Martinet}), $\iota$ is unique up to the action of $W(E_{9-d})$ and up to automorphisms of $E_{8-d}$. Therefore, in order to show that ${ \rm Exc}_{9-d}^{W(E_{8-d})} \neq \emptyset$, it suffices to show that there exists \emph{some} $\widetilde{X}_d$ containing a configuration of $(-2)$-curves of type $E_{8-d}$ and such that a $(-1)$-curve disjoint from this configuration exists. This is known and can be seen for example in \cite[Figures 23, 48, 57, 63, and 61]{MartinStadlmayr}. 

If $d \geq 7$, then $E_{8-d}$ does not contain any $(-2)$-vectors, hence $\Lambda' = 0$ and the implication $(3) \Rightarrow (1)$ holds, since $\widetilde{X}_d$ is a blow-up of $\bbP^2$ by assumption. 

Finally, if $d = 6$, then the maximal root lattice contained in $E_{8-d}=E_{2}$ is $A_1$. Thus, we may assume $\Lambda' =A_1$, for otherwise we can argue as in the previous Case $d \geq 7$.
Up to the action of ${\rm O}(E_3) = \{ \pm {\rm id} \} \times W(E_3)$, there are two embeddings of $A_1$ into $E_3 = A_2 \oplus A_1$. It is easy to check that $\iota: A_1 \hookrightarrow E_3$ factors through $E_2$ if and only if $\iota$ factors through the $A_2$-summand of $E_3$ and then $\iota$ is unique up to the action of $W(E_{3})$ and up to automorphisms of $A_1$. Hence, similar to what we did in the case $d \leq 5$, it suffices to find some $\widetilde{X}_6$ containing a $(-2)$-curve and a disjoint $(-1)$-curve. Again, this is known, see \cite[Figure 24]{MartinStadlmayr}.
\end{proof}

\begin{Corollary} \label{cor: lattice embeddings Dynkin Martinet - strategy proof}
Let $X_d$ be an RDP del Pezzo surface of degree $d \leq 8$, let $\Gamma'$ be a configuration of rational double points on $X_d$, and let $\Lambda'$ be the root lattice associated to $\Gamma'$. If $\Gamma'$ occurs on an RDP del Pezzo surface of degree $(d + 1)$ and satisfies one of the following conditions:
\begin{enumerate}
    \item $d \neq 4,2,1$,
    \item $d = 4$ and $\Lambda' \neq A_3$,
    \item $d = 2$ and $\Lambda' \not\in \{A_5 + A_1, A_5, A_3+2A_1, A_3+A_1, 4A_1, 3A_1\}$,
    \item $d = 1$ and $\Lambda' \not\in \{A_7, 2A_3, A_5+A_1, A_3+2A_1, 4A_1\}$.
\end{enumerate}
Then $X_d$ is the anti-canonical model of a blow-up in a smooth point $P$ of an RDP del Pezzo surface $X_{d+1}$ of degree $(d + 1)$ containing $\Gamma'$ such that $X_d \dashrightarrow X_{d+1}$ is an isomorphism around $\Gamma'$.
\end{Corollary}

\begin{proof}
The embeddings of root lattices into $E_6,E_7,$ and $E_8$ have been classified by Dynkin \cite[Table 11]{DynkinSemisimple} and for the embeddings of root lattices into $E_3 = A_2 \oplus A_1, E_4 = A_4$, and $E_5 = D_5$, we refer the reader to \cite[Exercise 4.2.1, 4.6.2]{Martinet}. It follows from these classifications that if an embedding of $\Lambda'$ into $E_{9-d}$ exists, then this embedding is unique (up to the action of ${\rm O}(E_{9-d})$), except precisely in the cases excluded in $(2), (3)$ and $(4)$. Hence, if one embedding $\Lambda' \hookrightarrow E_{9-d}$ factors through an embedding $E_{8-d} \hookrightarrow E_{9-d}$, then \emph{every} embedding factors through an embedding $E_{8-d} \hookrightarrow E_{9-d}$. 
If $\Gamma'$ occurs on some RDP del Pezzo surface of degree $(d+1)$, then an embedding of $\Lambda'$ with such a factorization exists and the claim follows from Theorem \ref{thm: lattice theoretic reduction criterion}.
\end{proof}

\section{Automorphism schemes of non-equivariant RDP del Pezzo surfaces} \label{sec: nonequivariant}
Throughout this section, $X$ denotes an RDP del Pezzo surface and $\pi: \widetilde{X} \to X$ is its minimal resolution. In this section, we will prove Theorem \ref{thm: nonequivariantmain}, that is, we will classify all $X$ with $H^0(X,T_X) \neq 0$ over a field of characteristic $p \in \{3,5,7\}$ and such that $X$ contains one of the RDPs excluded in Theorem \ref{thm: equivariantmain}. We will treat the cases $p = 7$, $p = 5$, and $p = 3$, in Sections \ref{sec: char7}, \ref{sec: char5}, and \ref{sec: char3}, respectively. This will complete the classification of all RDP del Pezzo surfaces with global vector fields in odd characteristic.
 
The strategy of proof is as follows: First, for each degree $1 \leq d \leq 9$, we give the list of RDP configurations $\Gamma$ that can occur on an RDP del Pezzo surface $X$ of degree $d$ and that contain at least one RDP whose minimal resolution is not $T_X$-equivariant. Then, starting with the highest possible degree and working our way down with Theorem \ref{thm: lattice theoretic reduction criterion}, we classify those $X$ containing $\Gamma$ and satisfying $H^0(X,T_X) \neq 0$. In each step, we give explicit equations and calculate $\Aut_X^0$ using Proposition \ref{prop: StabiLemma}, Proposition \ref{prop: ShortExactAutSequences}, and Proposition \ref{prop: fitting argument A_n}.

\begin{Notation}
If $d \geq 3$, we use the notation $x_0,\hdots,x_d$ for the coordinates of $\mathbb{P}^d$. If $d = 2$, we use the notation $x,y,z$ and $w$ for the coordinates of $\mathbb{P}(1,1,1,2)$, where $w$ has weight $2$. Finally, if $d = 1$, we use the notation $s,t,x$ and $y$ for the coordinates of $\mathbb{P}(1,1,2,3)$, where $x$ has weight $2$ and $y$ has weight $3$. 
In Tables \ref{Table Eqn and Aut - char 7}, \ref{Table Eqn and Aut - char 5}, \ref{Table Eqn and Aut - char 3, deg at least 4}, \ref{Table Eqn and Aut - char 3, deg 3}, \ref{Table Eqn and Aut - char 3, deg 2}, and \ref{Table Eqn and Aut - char 3, deg 1}, we describe $\Aut_X^0$ as follows:
\begin{itemize}
    \item We only describe the $R$-valued points of $\Aut_X^0$, where $R$ is an arbitrary local $k$-algebra. By \cite[Lemma 3.5]{MartinStadlmayr}, this suffices to describe the scheme structure of $\Aut_X^0$ completely. We do this by either describing a general $R$-valued point as a matrix or by describing the image of $[x_0:\hdots:x_n]$ (resp. $[x:y:z:w]$, resp. $[s:t:x:y]$) under a general $R$-valued automorphism of $X$.
    \item We often describe $\Aut_X^0$ as the group scheme $\langle G_1,G_2 \rangle$ generated by subgroup schemes $G_1$ and $G_2$ of $\Aut_X^0$.
    By this we mean that we describe $\Aut_X^0$, using Proposition \ref{prop: ShortExactAutSequences}, as the smallest subgroup scheme of $\PGL_{d+1}$ (resp. $\Aut_{\mathbb{P}(1,1,1,2)}$ if $d = 2$, resp. $\Aut_{\mathbb{P}(1,1,2,3)}$ if $d = 1$) containing both $G_1$ and $G_2$. 
    \item We use the variables $\lambda$ or $\lambda_i$ for $R$-valued points of $\mathbb{G}_m$ and $\mu_{p^n}$ (where $\lambda^{p^n} = 1$), and the variables $\varepsilon$ or $\varepsilon_i$ for $R$-valued points of $\mathbb{G}_a$ and $\alpha_{p^n}$ (where $\varepsilon^{p^n} = 0$).
\end{itemize}
\end{Notation}

\subsection{In characteristic $7$} \label{sec: char7}
By Theorem \ref{thm: equivariantmain}, we have to list all RDP configurations containing $A_6$ that can occur on an RDP del Pezzo surface in characteristic $7$.

\begin{Lemma}
If $p = 7$, $\deg(X) = d$, and $X$ contains an $A_6$-singularity, then $d$ and the configuration $\Gamma$ of RDPs on $X$ is one of the cases in Table \ref{Table critchar7}.
\begin{table}[h!]
$$
\begin{array}{|c||c |c|}
 \hline
    d & \Gamma & \subseteq \langle k_{9-d} \rangle^{\perp} \\ \hline \hline
   2 & A_6  & \subseteq E_7 \\ \hline
   1 & A_6, \hspace{3mm} A_6+A_1 & \subseteq E_8 \\ \hline
\end{array}
$$
\vspace{-2mm}
\caption{Non-equivariant RDP configurations in characteristic $7$}
\label{Table critchar7}
\end{table}
 \end{Lemma}
 
 \begin{proof}
 Since $A_6$ has rank $6$ and discriminant $7$, it does not embed into $E_{9-d}$ with $d \geq 3$. By \cite[Table 11]{DynkinSemisimple}, the only root lattice containing $A_6$ and embedding into $E_{7}$ is $A_6$ itself and there are precisely two root lattices containing $A_6$ and embedding into $E_8$, namely $A_6$ and $A_6+A_1$.
 \end{proof}

\begin{Theorem} \label{thm: mainchar7}
Assume that $p=7$ and $X$ contains an RDP of type $A_6$. Then, $H^0(X,T_X) \neq 0$ if and only if $X$ is given by an equation as in Table  \ref{Table Eqn and Aut - char 7}. Moreover, $\Aut_X^0$ is as in Table \ref{Table Eqn and Aut - char 7}, so that $\Aut_{\widetilde{X}}^0 \subsetneq \Aut_X^0$ and even $h^0(X,T_X) > h^0(\widetilde{X},T_{\widetilde{X}})$.
\end{Theorem}

\begin{proof}
By Table \ref{Table critchar7}, we have $\deg(X) = d \leq 2$. Assume $H^0(X,T_X) \neq 0$.

\underline{If $d = 2$}, then, by Theorem \ref{thm: anti-cano embeddings/coverings}, the anti-canonical system of $X$ realizes $X$ as a double cover of $\bbP^2$ branched over a quartic curve $Q$ with a simple singularity of type $A_6$. Over the complex numbers, there is a unique such $Q$ (see \cite[Proposition 1.3.II]{BruceGiblin}) and the argument carries over without change to characteristic $7$. Now, an elementary calculation shows that the Klein quartic equation
$$
x^3y + y^3z + z^3 x = 0
$$
defines such a $Q$ with an $A_6$-singularity at $[1:2:-3]$. Thus, $X$ is given by the equation in Table \ref{Table Eqn and Aut - char 7}. 
Clearly, the $\mu_7$-action described in Table \ref{Table Eqn and Aut - char 7} preserves $X$.
Since $\Aut_{\widetilde{X}}^0$ is trivial by \cite{MartinStadlmayr}, Proposition \ref{prop: fitting argument A_n} implies that $\Aut_X^0 = \mu_7$.

\underline{If $d = 1$}, then, by Theorem \ref{thm: anti-cano embeddings/coverings}, the anti-bi-canonical system of $X$ realizes $X$ as a double cover of the quadratic cone in $\mathbb{P}^3$ branched over a sextic curve $S$. By Table \ref{Table critchar7}, the RDP configuration on $X$ is either $A_6$ or $A_6 + A_1$. By Corollary \ref{cor: lattice embeddings Dynkin Martinet - strategy proof}, $X$ is the anti-canonical model of a blow-up $Y_1$ in a smooth point $P$ of the surface $X_2$ of Case $d = 2$. By Proposition \ref{prop: fitting argument A_n}, the morphism $Y_1 \to X$ is $\Aut_X^0$-equivariant, since it is an isomorphism around the $A_6$-singularity. Hence, by Proposition \ref{prop: StabiLemma}, we have $\Aut_X^0 = \Aut_{Y_1}^0 = \Stab_{\Aut_{X_2}^0}(P)^0 = \Stab_{\mu_7}(P)^0$.
So, since $\mu_7$ is simple, $Y_1$ is the blow-up of $X_2$ in a smooth fixed point $P$ of the $\mu_7$-action on $X_2$, and $\Aut_X^0 = \mu_7$.

Next, we prove the uniqueness of $X$. We may assume that $X_2$ is given by the equation in Table \ref{Table Eqn and Aut - char 7}. The fixed points of the $\mu_7$-action are the three points $[1:0:0:0], [0:1:0:0]$ and $[0:0:1:0]$. The automorphism $x \mapsto y \mapsto z \mapsto x$ of $X_2$ permutes these fixed points. So, $Y_1$ is unique and hence so is $X$.

Therefore, $X$ is the unique RDP del Pezzo surface of degree $1$ with an $A_6$-singularity and non-zero global vector fields. Now, the equation
$$y^2 = x^3 + ts^3x +t^5s$$
defines such an RDP del Pezzo surface of degree $1$ in $\bbP(1,1,2,3)$ with an $A_6$-singularity at $[1:-3:1:0]$ and, additionally, an $A_1$-singularity at $[1:0:0:0]$. Clearly, the $\mu_7$-action described in Table \ref{Table Eqn and Aut - char 7} preserves $X$. Hence, this is the surface we were looking for.
\end{proof}

\begin{Remark}
We remark that by \cite{MartinStadlmayr}, in both cases of Theorem \ref{thm: mainchar7} the minimal resolution $\widetilde{X}$ of $X$ does not admit any non-trivial global vector fields. In particular, $\pi: \widetilde{X} \to X$ is not $T_X$-equivariant.
\end{Remark}

\subsection{In characteristic $5$} \label{sec: char5}
By Theorem \ref{thm: equivariantmain}, we have to list all RDP configurations containing $A_4$ or $E_8^0$ that can occur on an RDP del Pezzo surface in characteristic $5$.

\begin{Lemma}
If $p = 5$, $\deg(X) = d$, and $X$ contains a singularity of type $A_4$ or $E_8^0$, then $d$ and the configuration $\Gamma$ of RDPs on $X$ is one of the cases in Table \ref{Table critchar5}.
\begin{table}[h!]
$$
\begin{array}{|c||c|c|}
 \hline
    \text{d} & \Gamma & \subseteq \langle k_{9-d} \rangle^\perp\\ \hline \hline
   5 & A_4 & \subseteq A_4 \\ \hline
   4 & A_4 & \subseteq D_5 \\ \hline
   3 & A_4, \hspace{3mm} A_4+A_1 & \subseteq E_6 \\ \hline
   2 & A_4, \hspace{3mm} A_4+A_1, \hspace{3mm} A_4+A_2 & \subseteq E_7 \\ \hline
   1 & A_4, \hspace{3mm} A_4+A_1, \hspace{3mm} A_4+2A_1, \hspace{3mm} A_4+A_2, \hspace{3mm} A_4+A_2+A_1, \hspace{3mm} A_4+A_3, \hspace{3mm} 2A_4, \hspace{3mm} E_8^0 & \subseteq E_8 \\ \hline
\end{array}
$$
\vspace{-2mm}
\caption{Non-equivariant RDP configurations in characteristic $5$}
\label{Table critchar5}
\end{table}
\end{Lemma}

 \begin{proof}
 Since $A_4$ has rank $4$ and discriminant $5$, it does not embed into $E_{9-d}$ with $d \geq 6$, and the only root lattice containing $A_4$ and embedding into $E_{9-d}$ with $d \in \{4,5\}$ is $A_4$ itself. The other cases can be found in \cite[Table 11]{DynkinSemisimple}.
 \end{proof}

\begin{Theorem} \label{thm: mainchar5}
Assume that $p=5$ and $X$ contains an RDP of type $A_4$ or $E_8^0$. Then, $H^0(X,T_X) \neq 0$ if and only if $X$ is given by an equation as in Table  \ref{Table Eqn and Aut - char 5}. Moreover, $\Aut_X^0$ is as in Table \ref{Table Eqn and Aut - char 5}, so that $\Aut_{\widetilde{X}}^0 \subsetneq \Aut_X^0$ and even $h^0(X,T_X) > h^0(\widetilde{X},T_{\widetilde{X}})$.
\end{Theorem}

\begin{proof} 
By Table \ref{Table critchar5}, we have $d \geq 5$.

\underline{If $d = 5$}, then $X$ is a quintic surface in $\mathbb{P}^5$. By Table \ref{Table critchar5}, the RDP configuration on $X$ is $A_4$. By the same argument as in characteristic $0$ (going through the possible configurations of four (possibly infinitely) near points in $\mathbb{P}^2$), there is a unique quintic surface in $\mathbb{P}^5$ containing an $A_4$-singularity. It is given by the equations in Table \ref{Table Eqn and Aut - char 5} (see \cite[Section 3.3., p.657]{Derenthal}) with singular point at $[0:0:0:0:0:1]$. The $\alpha_5$-action given in Table \ref{Table Eqn and Aut - char 5} preserves $X$ and does not preserve the singular point, hence $\alpha_5 \cap \Aut_{\widetilde{X}}^0 = \{ {\rm id} \}$. By Proposition \ref{prop: fitting argument A_n}, this implies that $\Aut_X^0 = \langle \alpha_5, \Aut_{\widetilde{X}}^0 \rangle$.

\underline{If $d = 4$}, then $X$ is a quartic surface in $\mathbb{P}^4$. By Table \ref{Table critchar5}, the RDP configuration on $X$ is $A_4$. By Corollary \ref{cor: lattice embeddings Dynkin Martinet - strategy proof}, $X$ is the anti-canonical model of a blow-up $Y_4$ in a smooth point $P$ of the surface $X_5$ of Case $d = 5$.

Such an $X$ is in fact unique: By \cite{MartinStadlmayr}, we have $\dim \Aut_{X_5}^0 = \dim \Aut_{\widetilde{X}_5}^0 = 4$, and since the orbit of $P$ is at most $2$-dimensional, the stabilizer of $P$ is positive-dimensional, so $H^0(\widetilde{X},T_{\widetilde{X}}) \neq 0$. By \cite{MartinStadlmayr}, there is a unique quartic del Pezzo surface with an $A_4$-singularity and whose minimal resolution has global vector fields, hence $X$ is unique.

In Table \ref{Table Eqn and Aut - char 5}, we give equations for such a surface (which we took from \cite{Derenthal}), hence this is our $X$. The singular point is located at $[0:0:0:0:1]$. The $\alpha_5$-action given in Table \ref{Table Eqn and Aut - char 5} preserves $X$ and does not preserve the singular point. Again, by Proposition \ref{prop: fitting argument A_n}, this implies that $\Aut_X^0 = \langle \alpha_5, \Aut_{\widetilde{X}}^0 \rangle$.

\underline{If $d = 3$}, then $X$ is a cubic surface in $\mathbb{P}^3$. By Table \ref{Table critchar5}, the RDP configuration on $X$ is $A_4$ or $A_4 + A_1$.
By Corollary \ref{cor: lattice embeddings Dynkin Martinet - strategy proof}, $X$ is the anti-canonical model of a blow-up $Y_3$ in a smooth point $P$ of the surface $X_4$ of Case $d = 4$.

Next, we show that there are at most two non-isomorphic such $X$. By \cite{MartinStadlmayr}, we have $\dim \Aut_{X_4}^0 = \dim \Aut_{\widetilde{X}_4}^0 = 2$. There is at most one $2$-dimensional orbit on $X_4$, hence there is at most one $X$ whose minimal resolution does not admit global vector fields. On the other hand, by \cite{MartinStadlmayr}, there is precisely one $X$ whose minimal resolution does have global vector fields.

In Table \ref{Table Eqn and Aut - char 5}, we give two (non-isomorphic) equations for cubic surfaces with an $A_4$-singularity, distinguished by their RDP configuration $\Gamma$, hence these are the two possible $X$:
    
    \begin{enumerate}
        \item If $\Gamma = A_4$, the singular point is located at $[-2:-1:2:1]$. We describe a $\mu_5$-action on $X$ in Table \ref{Table Eqn and Aut - char 5}. In this case, $\Aut_{\widetilde{X}}^0$ is trivial by \cite{MartinStadlmayr}, hence $\Aut_X^0 = \mu_5$ by Proposition \ref{prop: fitting argument A_n}. 
        \item If $\Gamma = A_4 + A_1$, the $A_4$-singularity is $[0:0:0:1]$ while the $A_1$-singularity is $[1:0:0:0]$. We describe an $\alpha_5 \rtimes \mathbb{G}_m$-action on $X$ in Table \ref{Table Eqn and Aut - char 5}. In this case, $\Aut_{\widetilde{X}}^0 = \mathbb{G}_m$ by \cite{MartinStadlmayr}, hence $\Aut_X^0 = \alpha_5 \rtimes \mathbb{G}_m$ by Proposition \ref{prop: fitting argument A_n}.
    \end{enumerate}
    
\underline{If $d = 2$}, then $X$ is a double cover of $\mathbb{P}^2$ branched over a quartic curve $Q$. 
By Table \ref{Table critchar5}, the possible RDP configurations on $X$ are $A_4$, $A_4+A_1$, and $A_4 + A_2$.
By Corollary \ref{cor: lattice embeddings Dynkin Martinet - strategy proof}, $X$ is the anti-canonical model of a blow-up $Y_2$ in a smooth point $P$ of an RDP del Pezzo surface $X_3$ of degree $3$ with an $A_4$-singularity. Since $Y_2 \to X_3$ and $Y_2 \to X$ are isomorphisms around the $A_4$-singularities, Proposition \ref{prop: fitting argument A_n} yields $\Aut_X^0 = \Aut_{Y_2}^0 = \Stab_{\Aut_{X_3}^0}(P)^0$. Hence, for each of the surfaces $X_3$ in Case $d=3$, we have to determine the points with non-trivial stabilizer.

\begin{itemize}       
    \item  If the RDP configuration on $X_3$ is $A_4$, then the points with non-trivial stabilizer under the action of $\Aut_{X_3}^0 = \mu_5$ are $[1:0:0:0],[0:1:0:0],[0:0:1:0]$, and $[0:0:0:1]$ and these fixed points are permuted by the automorphism $x_0 \mapsto x_1 \mapsto x_2 \mapsto x_3 \mapsto x_0$. Hence, there is a unique choice for $P$ up to isomorphism.

 \item If the RDP configuration on $X_3$ is $A_4 + A_1$, then there are four lines on $X_3$: The lines $\ell_1 = \{x_0 = x_1 = 0\}, \ell_2 = \{x_0 = x_2 = 0\}, \ell_3 = \{x_1 = x_2 = 0\}$ pass through the $A_4$-singularity at $[0:0:0:1]$ and the line $\ell_4 = \{x_2 = x_3 = 0\}$ passes through the $A_1$-singularity at $[1:0:0:0]$, but not through $A_4$. Moreover, $\ell_2$ and $\ell_4$ intersect in $[0:1:0:0]$.
   Straightforward calculation shows that the points with non-trivial stabilizer in $X_3 \setminus \{\ell_1,\ell_2,\ell_3\}$ are precisely those lying on the hyperplane $H = \{x_3 = 0\}$. The intersection $X_3  \cap H$ is the union of the conic $C = \{x_3 = x_0x_2 + x_1^2 = 0\}$ and the line $\ell_4$. 
    Hence, either $P \in C \setminus (\ell_1 \cup \ell_2 \cup \ell_3) = C \setminus \{[1:0:0:0],[0:0:1:0]\}$ or $P \in \ell_4 \setminus (\ell_1 \cup \ell_2 \cup \ell_3) = \ell_4 \setminus \{[1:0:0:0],[0:1:0:0]\}$. The group scheme $\Aut_{X_3}^0$ acts transitively on both loci, hence there are only two choices for $P$ up to isomorphism. In both cases, one checks that $\Aut_{Y_2}^0 = \Stab_{\Aut_{X_3}}(P)^0 = \mu_5$.
    One of the two choices of $P$ can be reduced to a previous case as follows:
    \begin{itemize}
        \item     If $P \in C \setminus \{[1:0:0:0],[0:0:1:0]\}$, then the RDP configuration on $X$ is $A_4 + A_1$. The strict transform $C'$ of $C$ on the blow-up $X'$ of $X$ in the $A_1$-singularity is a $(-1)$-curve that passes through the $(-2)$-curve over $A_1$ and which is disjoint from $A_4$. If we contract $C'$, we obtain an RDP del Pezzo surface $X_3'$ of degree $3$ which contains an $A_4$-singularity as its only singularity and such that $H^0(X_3',T_{X_3'}) \neq 0$ (by Proposition \ref{prop: Wahl Hirokado} and Blanchard's lemma). Hence $X_3'$ is isomorphic to the cubic surface with RDP configuration $A_4$ in Table \ref{Table critchar5}, and thus $X$ coincides with the surface we constructed in the previous bullet point.
    \end{itemize}
\end{itemize}
Summarizing, there are at most two RDP del Pezzo surfaces of degree $2$ which contain an $A_4$-singularity and which admit a non-trivial global vector field. Moreover, $\Aut_X^0 = \Stab_{\Aut_{X_3}^0}(P) = \mu_5$ in both cases. In Table \ref{Table Eqn and Aut - char 5}, we give two equations of such surfaces, distinguished by their RDP configuration $\Gamma$:
\begin{enumerate}
    \item If $\Gamma = A_4 + A_1$, the $A_4$-singularity is $[-2:1:2:0]$, the $A_1$-singularity is $[0:1:0:0]$, and the corresponding $\Aut_X^0 = \mu_5$-action is as in Table \ref{Table Eqn and Aut - char 5}
    \item If $\Gamma = A_4 + A_2$, the $A_4$-singularity is $[2:1:-1:0]$, the $A_2$-singularity is $[1:0:0:0]$, and the corresponding $\Aut_X^0 = \mu_5$-action is as in Table \ref{Table Eqn and Aut - char 5}.
\end{enumerate}

\underline{If $d = 1$}, then $X$ is a double cover of the quadratic cone in $\mathbb{P}^3$ branched over a sextic curve $S$. We consider separately the three cases where $X$ contains a single $A_4$-singularity (and possibly equivariant RDPs of other types), two $A_4$-singularities, and an $E_8^0$-singularity, respectively:
\begin{enumerate}
    \item[(a)] $X$ contains a single $A_4$-singularity. By Corollary \ref{cor: lattice embeddings Dynkin Martinet - strategy proof}, $X$ is the anti-canonical model of a blow-up $Y_1$ in a smooth point $P$ of an RDP del Pezzo surface $X_2$ with an $A_4$-singularity. By Proposition \ref{prop: fitting argument A_n}, we have $\Aut_X^0 = \Aut_{Y_1}^0 = \Stab_{\Aut_{X_2}^0}(P)^0$. Since $\Aut_{X_2}^0 = \mu_5$, we thus have to determine the fixed points of the $\mu_5$-action on $X_2$. 
    \begin{itemize}
        \item If the RDP configuration on $X_2$ is $A_4 + A_2$, then the fixed points of the $\mu_5$-action are $[1:0:0:0],[0:1:0:0],$ and $[0:0:1:0]$, and we recall that $[1:0:0:0]$ is the $A_2$-singularity. Hence, there are two choices for $P$. In fact, one of them does not occur, as the following argument shows:
    \begin{itemize}
        \item  
      If $P = [0:0:1:0]$, let $Q$ be the branch quartic of $X_2 \to \mathbb{P}^2$. The line $\ell = \{ y = 0\}$ is tangent to $Q$ at $[1:0:0]$ and $[0:0:1]$, hence the preimage of $\ell$ in $X_2$ consists of two smooth rational curves $C,C'$ meeting in $P$ and the $A_2$-singularity. In the minimal resolution $\widetilde{Y}_1$ of $Y_1$, their strict transforms $\widetilde{C}$ and $\widetilde{C}'$ are $(-2)$-curves, which, together with the two exceptional curves over the $A_2$-singularity, form an $A_4$-configuration of $(-2)$-curves. Thus, $X$ contains two $A_4$-singularities, contradicting our assumption. 
    \end{itemize}
        \item If the RDP configuration on $X_2$ is $A_4 + A_1$, then the fixed points of the $\mu_5$-action are $[1:0:0:1],[1:0:0:-1],[0:1:0:0],$ and $[0:0:1:0]$. The first two points are interchanged by $w \mapsto -w$ and the point $[0:1:0:0]$ is the $A_1$-singularity on $X_2$, hence we have two choices for $P$ up to isomorphism. 
        Let $Q$ be the branch quartic of $X_2 \to \mathbb{P}^2$ and recall that the $A_4$-singularity is located at $[-2:1:2:0]$. We will now show that both choices for $P$ lead to the same surface as the one in the previous bullet point:        
        \begin{itemize}
            \item If $P = [1:0:0:1]$, then the image $[1:0:0]$ of $P$ in $\mathbb{P}^2$ lies on the two bitangents $\ell_1 = \{y = 0\}$ and $\ell_2 = \{z = 0\}$ of $Q$. Let $C_i,C_i'$ be the two irreducible components of the preimage of $\ell_i$ for $i = 1,2$ and assume that $P$ lies on $C_1$ and $C_2$. On $\widetilde{Y}_1$, the strict transforms $\widetilde{C}_1$ and $\widetilde{C}_2$ are $(-2)$-curves, while the strict transforms $\widetilde{C}_1'$ and $\widetilde{C}_2'$ remain $(-1)$-curves. Thus, the RDP configuration on $X$ is $A_4 + A_2 + A_1$, where the $A_2$ is obtained from the $A_1$ of $Y_1$ by also contracting $\widetilde{C}_2$, and the $A_1$ arises from the contraction of $\widetilde{C}_1$.            
          Therefore, if we contract the image of $C_1'$ in the surface $Y_1'$ obtained from $Y_1$ by contracting $C_2$, we obtain an RDP del Pezzo surface of degree $2$ with global vector fields and containing an RDP configuration of type $A_4 + A_2$, hence this case is reduced to the previous bullet point. 
          \item If $P = [0:0:1:0]$, then the image $[0:0:1]$ of $P$ in $\mathbb{P}^2$ lies on the bitangent $\ell = \{y = 0\}$ and $P$ is the non-transversal intersection point of the two irreducible components $C$ and $C'$ of the preimage of $\ell$ in $X_2$. Thus, the strict transforms $\widetilde{C}$ and $\widetilde{C}'$ of $C$ and $C'$ on $\widetilde{Y}_1$ are $(-2)$-curves, hence the RDP configuration on $X$ is $A_4 + A_2 + A_1$, where the $A_2$ is obtained by contracting $\widetilde{C}$ and $\widetilde{C}'$. 
          The preimage $D$ of the line $\{x = 0 \}$ in $X_2$ is a cuspidal cubic with cusp at $[0:1:0:0]$. On $\widetilde{X}_2$, the strict transform of $D$ is a smooth rational curve of self-intersection $0$ by adjunction. Hence, the strict transform $\widetilde{D}$ of $D$ on $\widetilde{Y}_1$ is a $(-1)$-curve which passes through the exceptional curve over the $A_1$-singularity. Contracting the image of $\widetilde{D}$ on the surface $Y_1'$ obtained from $Y_1$ by contracting $C$ and $C'$ and resolving the $A_1$-singularity, we obtain an RDP del Pezzo surface of degree $2$ with global vector fields containing an RDP configuration of type $A_4 + A_2$, hence this case is also reduced to the previous bullet point. 
        \end{itemize}
        Summarizing, there is at most one RDP del Pezzo surface of degree $1$ with global vector fields and containing a single $A_4$-singularity. In Table \ref{Table Eqn and Aut - char 5}, we give an equation of such a surface, hence this is our $X$. The singularities of $X$ are as follows: $A_4$ at $[1:-2:2:0]$, $A_2$ at $[0:1:0:0]$, and $A_1$ at $[1:0:0:0]$.    
        The $\mu_5$-action we describe in Table \ref{Table Eqn and Aut - char 5} preserves the equation, hence $\Aut_X^0 = \mu_5$.
    \end{itemize}
    \item[(b)] $X$ contains two $A_4$-singularities. By \cite[Theorem 4.1.]{Lang2}, there is a unique RDP del Pezzo surface of degree $1$ with RDP configuration $A_4 + A_4$, namely the Weierstrass model of the rational elliptic surface with singular fibers of type ${\rm I}_5,{\rm I}_5,{\rm II}$. Its equation is given in Table \ref{Table Eqn and Aut - char 5}. The $\alpha_5 \rtimes \mu_5$-action we describe in Table \ref{Table Eqn and Aut - char 5} preserves the equation. By \cite{MartinStadlmayr}, $\Aut_{\widetilde{X}}^0$ is trivial, hence $\Aut_X^0 = \alpha_5 \rtimes \mu_5$ follows from Proposition \ref{prop: fitting argument A_n}. 
    \item[(c)] $X$ contains an $E_8^0$-singularity. By \cite{Lang2}, there are two RDP del Pezzo surfaces of degree $1$ with an RDP of type $E_8$. The one whose equation we give in Table \ref{Table Eqn and Aut - char 5} has an $E_8^0$-singularity, while the other one has an $E_8^1$-singularity (see \cite[Table 1]{Stadlmayr}). The $\alpha_5 \rtimes \mathbb{G}_m$-action we describe in Table \ref{Table Eqn and Aut - char 5} preserves the equation.  We leave it to the reader to check that $\Aut_X^0 = \Stab_{\Aut_{\mathbb{P}(1,1,2)}}(S)^0 = \alpha_5 \rtimes \mathbb{G}_m$.
\end{enumerate}
\end{proof}

\subsection{In characteristic $3$} \label{sec: char3}
By Theorem \ref{thm: equivariantmain}, we have to list all RDP configurations containing $A_2,A_5,A_8,E_6^0$, $E_6^1,E_7^0,E_8^0$ or $E_8^1$ that can occur on an RDP del Pezzo surface.

\begin{Lemma}
If $p = 3$, $\deg(X) = d$, and $X$ contains a singularity of type $A_2,A_5,A_8,E_6^0,E_6^1,E_7^0,E_8^0$ or $E_8^1$, then $d$ and the configuration $\Gamma$ of RDPs on $X$ is one of the cases in Table \ref{Table critchar3}.
\begin{table}[h!]
\centering
$
\begin{array}{|c||c||c|}
 \hline
    \text{d} & \Gamma & \subseteq \langle k_{9-d} \rangle^\perp \\ \hline \hline
   6 & A_2,  \hspace{3mm} A_2 + A_1 & \subseteq A_2 + A_1 \\ \hline
   5 & A_2, \hspace{3mm} A_2+A_1 & \subseteq A_4 \\ \hline
   4 & A_2, \hspace{3mm} A_2+A_1, \hspace{3mm} A_2+2A_1 & \subseteq D_5 \\ \hline
   3 & A_2, \hspace{3mm} A_2+A_1, \hspace{3mm} A_2+2A_1, \hspace{3mm} 2A_2, \hspace{3mm} 2A_2+A_1, \hspace{3mm} A_5, \hspace{3mm} 3A_2, \hspace{3mm} A_5+A_1, \hspace{3mm} E_6^0, \hspace{3mm} E_6^1, & \subseteq E_6 \\ \hline
   \multirow{2}{*}{$2$} & A_2, \hspace{3mm} A_2+A_1, \hspace{3mm} A_2+2A_1, \hspace{3mm} 2A_2, \hspace{3mm} A_2+3A_1, \hspace{3mm} 2A_2+A_1, \hspace{3mm} A_3+A_2, \hspace{3mm} (A_5)', \hspace{3mm} 3A_2,
   & \multirow{2}{*}{$\subseteq E_7$} \\
   & A_3+A_2+A_1, \hspace{3mm} A_4+A_2, \hspace{3mm} (A_5+A_1)', \hspace{3mm} E_6^0, \hspace{3mm} E_6^1, \hspace{3mm} A_5+A_2, \hspace{3mm} E_7^0 &  \\ \hline
   \multirow{5}{*}{$1$} & A_2, \hspace{3mm} A_2+A_1, \hspace{3mm} A_2+2A_1, \hspace{3mm} 2A_2, \hspace{3mm} A_2+3A_1, \hspace{3mm} 2A_2+A_1, \hspace{3mm} A_3+A_2, \hspace{3mm} A_5, & \multirow{5}{*}{$\subseteq E_8$} \\
   & A_2+4A_1, \hspace{3mm} 2A_2+2A_1, \hspace{3mm} 3A_2, \hspace{3mm} A_3+A_2+A_1, \hspace{3mm} A_4+A_2, \hspace{3mm} D_4+A_2, \hspace{3mm} (A_5+A_1)', &\\
   & E_6^0, \hspace{3mm} E_6^1, \hspace{3mm} 3A_2+A_1, \hspace{3mm} A_3+A_2+2A_1, \hspace{3mm} A_4+A_2+A_1, \hspace{3mm} A_5+2A_1, \hspace{3mm} A_5+A_2, & \\ 
   & D_5+A_2, \hspace{3mm} E_6^0+A_1, \hspace{3mm} E_6^1+A_1, \hspace{3mm} E_7^0, \hspace{3mm} 4A_2, \hspace{3mm} A_5+A_2+A_1, & \\
   & E_6^0+A_2, \hspace{3mm} E_6^1+A_2, \hspace{3mm} A_8, \hspace{3mm} E_8^0, \hspace{3mm} E_8^1 & \\ \hline
\end{array}
$
\caption{Non-equivariant RDP configurations in characteristic $3$}
    \label{Table critchar3}
\end{table}
 \end{Lemma}
 
 \begin{proof}
 The maximal root lattice contained in $E_2$ is isomorphic to $A_1$, hence none of the root lattices in the statement of the lemma embed into $E_{9-d}$ with $d \geq 7$. The list for $4 \leq d \leq 6$ follows from \cite[Exercise 4.2.1, 4.6.2]{Martinet} and the one for $7 \leq d \leq 9$ follows from \cite[Table 11]{DynkinSemisimple} (note that the lattice $A_6 + A_2$ in Dynkin's table of root lattices in $E_8$ should be $E_6 + A_2$), where we marked those RDP configurations whose associated root lattice embeds in two non-conjugate ways into $E_{9-d}$ by a prime $'$.
 \end{proof}

\begin{Theorem} \label{thm: mainchar3}
Assume that $p = 3$ and $X$ contains an RDP of type $A_2,A_5,A_8,E_6^0,E_6^1,E_7^0,E_8^0$, or $E_8^1$. Then, $H^0(X,T_X) \neq 0$ if and only if $X$ is given by an equation as in Tables \ref{Table Eqn and Aut - char 3, deg at least 4}, \ref{Table Eqn and Aut - char 3, deg 3}, \ref{Table Eqn and Aut - char 3, deg 2} and \ref{Table Eqn and Aut - char 3, deg 1}. Moreover, $\Aut_X^0$ is as given in these tables, so that $\Aut_{\widetilde{X}}^0 \subsetneq \Aut_X^0$ and even $h^0(X,T_X) > h^0(\widetilde{X},T_{\widetilde{X}})$.
\end{Theorem}

\begin{proof}
By Table \ref{Table critchar3}, we have $d \leq 6$.

\underline{If $d = 6$}, then $X$ is a sextic surface in $\mathbb{P}^6$. The RDP configuration $\Gamma$ on $X$ is either $A_2$ or $A_2 + A_1$ by Table \ref{Table critchar3}. In both cases, $X$ is uniquely determined by its singularities, by the same argument as in characteristic $0$, that is, by checking the possible configurations of infinitely near points in $\mathbb{P}^2$:

\begin{enumerate}
    \item[(a)] If $\Gamma = A_2$ then, by \cite[Section 3.2.]{Derenthal}, $X$ is given by the equations in Table \ref{Table Eqn and Aut - char 3, deg at least 4}. The $A_2$-singularity is $[1:0:0:0:0:0:0]$ and we give an $\alpha_3$-action on $X$ which does not preserve the singularity in Table \ref{Table Eqn and Aut - char 3, deg at least 4}. By Proposition \ref{prop: fitting argument A_n}, this implies $\Aut_X^0 = \langle \alpha_3, \Aut_{\widetilde{X}}^0 \rangle$.
    \item[(b)] If $\Gamma = A_2 + A_1$, then, by \cite[Appendix, p.3]{KikuchiNakano},  $X$ is given by the equation in Table \ref{Table Eqn and Aut - char 3, deg at least 4}. The $A_2$-singularity is $[0:1:0:0:0:0:0]$ and the $A_1$-singularity is $[0:0:0:0:0:0:1]$. We give an $\alpha_3$-action on $X$ which does not fix the $A_2$-singularity in Table \ref{Table Eqn and Aut - char 3, deg at least 4}.
     By Proposition \ref{prop: fitting argument A_n}, this implies $\Aut_X^0 = \langle \alpha_3, \Aut_{\widetilde{X}}^0 \rangle$.
\end{enumerate}

\underline{If $d = 5$}, then $X$ is a quintic surface in $\mathbb{P}^5$. By Table \ref{Table critchar3}, the RDP configuration $\Gamma$ on $X$ is either $A_2$ or $A_2 + A_1$. As in the Case $d = 6$, $X$ is uniquely determined by $\Gamma$, as can be seen by checking the possible configurations of four infinitely near points in $\mathbb{P}^2$ (see e.g. \cite[Section 8.5.1, p.430]{Dolgachev-classical}).

\begin{enumerate}
    \item[(a)] If $\Gamma = A_2$, then, by \cite[Section 3.3.]{Derenthal}, $X$ is given by the equation in Table \ref{Table Eqn and Aut - char 3, deg at least 4}. The $A_2$-singularity is $[0:0:1:0:0:0]$ and we give an $\alpha_3$-action on $X$ which does not preserve the singularity in Table \ref{Table Eqn and Aut - char 3, deg at least 4}.  By Proposition \ref{prop: fitting argument A_n}, this implies $\Aut_X^0 = \langle \alpha_3, \Aut_{\widetilde{X}}^0 \rangle$.
    \item[(b)] If $\Gamma = A_2 + A_1$, then by \cite[Appendix, p.5]{KikuchiNakano} is given by the equation in Table \ref{Table Eqn and Aut - char 3, deg at least 4}. The $A_2$-singularity is $[0:1:0:0:0:0]$ and the $A_1$-singularity is $[0:0:0:0:0:1]$. We give an $\alpha_3$-action on $X$ which does not fix the $A_2$-singularity in Table \ref{Table Eqn and Aut - char 3, deg at least 4}.
    By Proposition \ref{prop: fitting argument A_n}, this implies $\Aut_X^0 = \langle \alpha_3, \Aut_{\widetilde{X}}^0 \rangle$.
\end{enumerate}

\underline{If $d = 4$}, then $X$ is a quartic surface in $\mathbb{P}^4$. By Table \ref{Table critchar3}, the RDP configuration on $X$ is $A_2$, $A_2+A_1$, or $A_2 + 2A_1$. By Corollary \ref{cor: lattice embeddings Dynkin Martinet - strategy proof}, $X$ is the anti-canonical model of a blow-up $Y_4$ in a smooth point $P$ of an RDP del Pezzo surface $X_5$ of degree $5$ with an $A_2$-singularity. 

Next, we show that there are at most three non-isomorphic such $X$. By \cite{MartinStadlmayr}, there are at most two $X$ whose minimal resolution has global vector fields. Assume that $\Aut_{\widetilde{X}}^0 = \{{\rm id}\}$. Then, by Proposition \ref{prop: StabiLemma} and Proposition \ref{prop: fitting argument A_n}, the stabilizer of $P$ is $0$-dimensional, hence $\dim \Aut_{\widetilde{X}_5}^0 = \dim \Aut_{X_5}^0 \leq 2$. Hence, by \cite{MartinStadlmayr}, $X_5$ is the RDP del Pezzo surface of degree $5$ whose RDP configuration is $A_2$. For this surface, $\dim \Aut_{X_5}^0 = 2$, hence there is at most one $2$-dimensional orbit on $X_5$, so there is at most one $X$ whose minimal resolution does not have global vector fields.
In Table \ref{Table Eqn and Aut - char 3, deg at least 4}, we give equations for three such surfaces, distinguished by their RDP configuration $\Gamma$:

\begin{enumerate}
    \item If $\Gamma = A_2$, the $A_2$-singularity is at $[1:1:1:1:1]$. We give a $\mu_3$-action on $X$ in Table \ref{Table Eqn and Aut - char 3, deg at least 4}, while the group scheme $\Aut_{\widetilde{X}}^0$ is trivial by \cite{MartinStadlmayr}. By Proposition \ref{prop: fitting argument A_n}, this implies $\Aut_X^0 = \mu_3$. 
    
    \item If $\Gamma = A_2 + A_1$, the $A_2$-singularity is $[1:0:0:0:0]$ while the $A_1$-singularity is $[0:0:0:0:1]$. We give a $\alpha_3 \rtimes \mathbb{G}_m$-action on $X$ in Table \ref{Table Eqn and Aut - char 3, deg at least 4} , while $\Aut_{\widetilde{X}}^0 = \mathbb{G}_m$ by \cite{MartinStadlmayr}, hence $\Aut_X^0 = \alpha_3 \rtimes \mathbb{G}_m$ by Proposition \ref{prop: fitting argument A_n}. 
    
    \item If $\Gamma = A_2 + 2A_1$, the $A_2$-singularity is $[0:0:0:0:1]$ and the two $A_1$-singularities are $[0:1:0:0:0]$ and $[0:0:1:0:0]$. We give a $\alpha_3 \rtimes \mathbb{G}_m^2$-action on $X$ in Table \ref{Table Eqn and Aut - char 3, deg at least 4} , while $\Aut_{\widetilde{X}}^0 = \mathbb{G}_m^2$ by \cite{MartinStadlmayr}, hence $\Aut_X^0 = \alpha_3 \rtimes \mathbb{G}_m^2$ by Proposition \ref{prop: fitting argument A_n}. 

\end{enumerate}

\underline{If $d = 3$}, then $X$ is a cubic surface in $\mathbb{P}^3$. By Table \ref{Table critchar3}, $X$ contains either a single $A_2$, two $A_2$s, three $A_2$s, one $A_5$, one $E_6^0$, or one $E_6^1$. In the following, we consider these six cases separately.

\begin{enumerate}
    \item[(a)] $X$ contains a single $A_2$-singularity. By Corollary \ref{cor: lattice embeddings Dynkin Martinet - strategy proof}, $X$ is the anti-canonical model of a blow-up $Y_3$ in a smooth point $P$ of an RDP del Pezzo surface $X_4$ with an $A_2$-singularity. 
    More precisely, since $X \dashrightarrow X_4$ is an isomorphism around the only non-equivariant RDP on $X$ and all other singularities of $X$ are $A_1$-singularities by Table \ref{Table critchar3}, Proposition \ref{prop: fitting argument A_n} and Proposition \ref{prop: StabiLemma} imply that $\Aut_X^0 = \Aut_{Y_3}^0 = \Stab_{\Aut_{X_4}^0}(P)^0$. In particular, $P$ is a point with non-trivial stabilizer on one of the surfaces $X_4$ in Table \ref{Table Eqn and Aut - char 3, deg at least 4}. 
    Before we go on, note that by \cite{MartinStadlmayr}, the group scheme $\Aut_{\widetilde{X}}^0$ is trivial, hence the stabilizer of $P$ is $0$-dimensional.
    \begin{itemize}
        \item Assume the RDP configuration on $X_4$ is $A_2 + 2A_1$. In this case, $\Aut_{X_4}^0$ is $2$-dimensional, so there is at most one $2$-dimensional orbit for this action, hence there is at most one choice for $P$ up to isomorphism.
        \item Assume the RDP configuration on $X_4$ is $A_2 + A_1$. The surface $X_4$ contains the six lines $\ell_1 = \{x_0 = x_2 = x_3 = 0\}, \ell_2 = \{x_0 = x_2 = x_4 = 0\}, \ell_3 = \{x_0 = x_3 = x_1 + x_4 = 0\}, \ell_4 = \{x_1 = x_2 = x_3 = 0\}, \ell_5 = \{x_1 = x_2 = x_4 = 0\}, \ell_6 = \{x_1 = x_3 = x_4 = 0\}$. The lines $\ell_4,\ell_5,$ and $\ell_6$ pass through the $A_2$-singularity at $[1:0:0:0:0]$ and the others do not. The lines $\ell_1$ and $\ell_4$ pass through the $A_1$-singularity at $[0:0:0:0:1]$ and the others do not.
        Using the description of the $\Aut_{X_4}^0$-action in Table \ref{Table Eqn and Aut - char 3, deg at least 4}, one easily checks that the points with non-trivial and $0$-dimensional stabilizer on $X_4$ are precisely those on $\ell_2$ and those on $\ell_3$, except for $[0:1:0:0:0],[0:0:1:0:0], [0:0:0:1:0],$ and $[0:1:0:0:-1]$. Since the automorphism $x_2 \leftrightarrow x_3, x_4 \leftrightarrow -x_4-x_1$ interchanges the two lines $\ell_2$ and $\ell_3$ and $\mathbb{G}_m \subseteq \Aut_{X_4}$ acts transitively on the locus of points with $0$-dimensional stabilizer on each of $\ell_2$ and $\ell_3$, there is a unique choice for $P$ up to isomorphism. We will now reduce this case to the previous bullet point.
        \begin{itemize}
            \item Without loss of generality, assume that $P \in \ell_2 \setminus \{[0:1:0:0:0],[0:0:0:1:0]\}$. Then, the strict transform of $\ell_2$ on $Y_3$ is a $(-2)$-curve and the RDP configuration on $X$ is $A_2 + 2A_1$. Since $\ell_3$ is disjoint from $\ell_2$, it remains a $(-1)$-curve on $X$, hence we can contract it to obtain a realization of $X$ as a blow-up of an RDP del Pezzo surface $X_4'$ with RDPs of type $A_2 + 2A_1$. Hence, this case is reduced to the previous bullet point. 
        \end{itemize}
        \item Assume the RDP configuration on $X_4$ is $A_2$. Using the description of the $\Aut_{X_4}^0 = \mu_3$-action given in Table \ref{Table Eqn and Aut - char 3, deg at least 4}, we see that the points $P$ with non-trivial stabilizer on $X_4$ are $[1:0:0:0:0],[0:1:0:0:0],[0:0:1:0:0],[0:0:0:1:0],$ and $[0:0:0:0:1]$. The surface $X_4$ admits the two involutions $x_0 \leftrightarrow x_1$ and $(x_0,x_1) \leftrightarrow (x_2,x_3)$. Hence, blowing up any of the first four points leads to the same surface. In fact, we already treated the resulting surface, as the following argument shows:
        \begin{itemize}
            \item Assume without loss of generality that $P = [1:0:0:0:0]$. There are two lines on $X_4$ passing through $P$, namely $\ell_1 = \{x_1 = x_2 = x_4 = 0\}$ and $\ell_2 = \{x_1 = x_3 = x_4 = 0\}$, and both of these lines do not pass through the $A_2$-singularity at $[1:1:1:1:1]$. Their strict transforms on $Y_3$ are disjoint $(-2)$-curves. Thus, the RDP configuration on $X$ is $A_2 + 2A_1$. Now, the conic $C = \{x_1 = x_2 + x_3 = x_0x_4 - x_3^2 = 0\}$ meets $\ell_1$ and $\ell_2$ transversally at $P$ and does not pass through $[1:1:1:1:1]$, hence we can contract the image of the strict transform of $C$ in $X$ and obtain an RDP del Pezzo surface $X_4'$ with RDP configuration $A_2 + 2A_1$ and with global vector fields, hence $X_4'$ is the surface in the first bullet point.
        \end{itemize}
    \end{itemize}
    Summarizing, there are at most two RDP del Pezzo surfaces of degree $3$ with global vector fields and containing a single $A_2$-singularity. Moreover, $\Aut_X^0 = \Stab_{\Aut_{X_4}^0}(P)^0 = \mu_3$.
    In Table \ref{Table Eqn and Aut - char 3, deg 3}, we give equations for two such surfaces, distinguished by their RDP configuration $\Gamma$:
    \begin{enumerate}
        \item[(1)] If $\Gamma = A_2$, the $A_2$-singularity is at $[1:1:-1:-1]$. Clearly, the $\mu_3$-action we give preserves the equation.
        \item[(2)] If $\Gamma = A_2 + 2A_1$, the $A_2$-singularity is at $[1:-1:-1:1]$ and the two $A_1$-singularities are at $[0:1:0:0]$ and $[0:0:1:0]$. Again, the $\mu_3$-action we give preserves the equation. 
    \end{enumerate}
    \item[(b)] $X$ contains two $A_2$-singularities. By Table \ref{Table critchar3}, the RDP configuration $\Gamma$ on $X$ is $2A_2$ or $2A_2 + A_1$. Simplifying the normal form of Roczen given in \cite{Roczencubic}, we obtain the equations given in Table \ref{Table Eqn and Aut - char 3, deg 3}. In particular, there is a $1$-dimensional family of $X$ with $\Gamma = 2A_2$ and a unique $X$ with $\Gamma = 2A_2 + A_1$. 
    \begin{enumerate}
        \item[(1)] If $\Gamma = 2A_2$, the two $A_2$-singularities are at $[0:0:1:0]$ and $[0:0:0:1]$. The two $\alpha_3$-actions and the $\Aut_{\widetilde{X}}^0 = \mathbb{G}_m$-action we give in Table \ref{Table Eqn and Aut - char 3, deg 3} preserve the equation. Each of the $\alpha_3$-actions fixes one of the $A_2$-singularities and does not preserve the other one, hence Proposition \ref{prop: fitting argument A_n} shows $\Aut_X^0 = \langle \alpha_3,\alpha_3, \mathbb{G}_m \rangle$.
            \item[(2)] If $\Gamma = 2A_2 + A_1$, the two $A_2$-singularities are at $[0:0:1:0]$ and $[0:0:0:1]$ and the $A_1$-singularity is at $[0:1:0:0]$. The two $\alpha_3$-actions and the $\Aut_{\widetilde{X}}^0 = \mathbb{G}_m$-action we describe in Table \ref{Table Eqn and Aut - char 3, deg 3} preserve the equation. By the same argument as in (1), we have $\Aut_X^0 = \langle \alpha_3,\alpha_3, \mathbb{G}_m \rangle$.
    \end{enumerate}
    \item[(c)] $X$ contains three $A_2$-singularities. Again, we can simplify the normal form of Roczen given in \cite{Roczencubic} and obtain the equation in Table \ref{Table Eqn and Aut - char 3, deg 3}, which admits $A_2$-singularities at $[0:1:0:0],[0:0:1:0],$ and $[0:0:0:1]$.
    In Table \ref{Table Eqn and Aut - char 3, deg 3}, we give an action of $\alpha_3^3 \rtimes \mathbb{G}_m^2$ on $X$. By \cite{MartinStadlmayr}, we have $\Aut_{\widetilde{X}}^0 = \mathbb{G}_m^2$.
    Each of the the three factors of the $\alpha_3^3$-action preserves two of the $A_2$-singularities and moves the other one, hence Proposition \ref{prop: fitting argument A_n} yields $\Aut_X^0 = \alpha_3^3 \rtimes \mathbb{G}_m^2$. 
    \item[(d)] $X$ contains an $A_5$-singularity. As above, we simplify Roczen's normal form \cite{Roczencubic} to the two equations in Table \ref{Table Eqn and Aut - char 3, deg 3}.
    Let $\Gamma$ be the RDP configuration on $X$:
    \begin{enumerate}      
    \item[(1)] If $\Gamma = A_5$, then the $A_5$-singularity is at $[0:0:0:1]$. The $\alpha_3$-action we describe preserves the equation and does not fix the $A_5$-singularity. By \cite{MartinStadlmayr}, we have $\Aut_{\widetilde{X}}^0 = \mathbb{G}_a \rtimes \mu_3$ and we describe this action in terms of the equation in Table \ref{Table Eqn and Aut - char 3, deg 3}. By Remark \ref{rem: A5}, we have $\Aut_X^0 = \langle \alpha_3,\mathbb{G}_a \rtimes \mu_3 \rangle$.  
        \item[(2)] If $\Gamma = A_5 + A_1$, then the $A_5$-singularity is at $[0:0:0:1]$ and the $A_1$-singularity is at $[0:0:1:0]$. The $\alpha_3$-action we describe preserves the equation. By \cite{MartinStadlmayr}, we have $\Aut_{\widetilde{X}}^0 = \mathbb{G}_a \rtimes \bbG_m$ and we describe this action in terms of the equation in Table \ref{Table Eqn and Aut - char 3, deg 3}. By Remark \ref{rem: A5}, we have $\Aut_X^0 = \langle \alpha_3,\mathbb{G}_a \rtimes \bbG_m \rangle$.  
    \end{enumerate}
    
    \item[(e)] $X$ contains an $E_6^0$-singularity. By \cite[Case C3]{Roczencubic}, $X$ is given by the equation in Table \ref{Table Eqn and Aut - char 3, deg 3} and the singularity is at $[0:0:0:1]$. 
    In Table \ref{Table Eqn and Aut - char 3, deg 3}, we give an action of a group scheme $G$ of order $27$ such that no subgroup scheme of $G$ lifts to $\widetilde{X}$, as well as the action of $\Aut_{\widetilde{X}}^0 = \mathbb{G}_a^2 \rtimes \mathbb{G}_m$ in terms of the equation.
We leave it to the reader to check that these two actions generate $\Aut_X^0 = \Stab_{{\rm PGL}_4}(X)^0$.
    
    \item[(f)] $X$ contains an $E_6^1$-singularity.  By \cite{Roczencubic}, $X$ is given by the equation in Table \ref{Table Eqn and Aut - char 3, deg 3} and the singularity is at $[0:0:0:1]$.  In Table \ref{Table Eqn and Aut - char 3, deg 3}, we give a $\mu_3$-action that does not preserve the singular point, as well as the action of $\Aut_{\widetilde{X}}^0 = \mathbb{G}_a^2$ in terms of the equation.
We leave it to the reader to check that these two actions generate $\Aut_X^0 = \Stab_{{\rm PGL}_4}(X)^0$.
\end{enumerate}

\underline{If $d = 2$} , then $X$ is a double cover of $\mathbb{P}^2$ branched over a quartic curve $Q$.
In this case, we will take a slightly different approach which will turn out to be more economical than using Corollary \ref{cor: lattice embeddings Dynkin Martinet - strategy proof}. Namely, we classify all possible $Q$ with global vector fields. If $Q$ admits a global vector field, then it also admits an additive or multiplicative global vector field. This vector field is induced by a vector field on $\mathbb{P}^2$. Up to conjugation, there are two non-zero vector fields $D$ on $\mathbb{P}^2$ with $D^3 = D$ and two non-zero vector fields with $D^3 = 0$. They correspond to the following four matrices in Jordan normal form in the Lie algebra of $\PGL_3$: 
$$
\left(
\begin{matrix} 
0 & 0 & 0 \\
0 & 1 & 0 \\
0 & 0 & -1
\end{matrix}
\right),
\quad
\left(
\begin{matrix} 
0 & 0 & 0 \\
0 & 0 & 0 \\
0 & 0 & 1
\end{matrix}
\right),
\quad
\left(
\begin{matrix} 
0 & 1 & 0 \\
0 & 0 & 0 \\
0 & 0 & 0
\end{matrix}
\right),
\quad
\left(
\begin{matrix} 
0 & 1 & 0 \\
0 & 0 & 1 \\
0 & 0 & 0
\end{matrix}
\right)
$$
Integrating the corresponding vector fields (see e.g. \cite[Proposition 3.1]{Tziolas}), we obtain the following four $\mu_3$- and $\alpha_3$-actions on $\bbP^2$:
\begin{enumerate}
    \item[(a)] $\mu_3: [x:y:z] \mapsto [x:\lambda y: \lambda^{2} z]$
    \item[(b)] $\mu_3: [x:y:z] \mapsto [x:y: \lambda z]$
    \item[(c)] $\alpha_3: [x:y:z] \mapsto [x + \varepsilon y: y : z]$
    \item[(d)] $\alpha_3: [x:y:z] \mapsto [x + \varepsilon y - \varepsilon^2 z : y + \varepsilon z : z]$
\end{enumerate}
For each of these actions, we will now classify the quartics that are invariant under it:
\begin{enumerate}
    \item[(a)] There are three types of quartics which are invariant under this $\mu_3$-action, namely
    \begin{eqnarray*}
    ax^4 + bx^2yz + cxy^3 + dxz^3 + ey^2z^2 \\
    ax^3y + bx^2z^2 + cxy^2z + dy^4 + eyz^3 \\
    ax^3z + bx^2y^2 + cxyz^2 + dy^3z + ez^4
    \end{eqnarray*}
    These families are identified by permuting $x,y$ and $z$, so it suffices to study the first one.
    Now, we simplify this equation and identify the singularities. We sort the classification according to the number of coefficients that are $0$. Since $Q$ is reduced, at least $2$ coefficients are non-zero. Except in the case where $c = d = 0$, we can scale three of the non-zero coefficients to $1$.
    \begin{itemize}
        \item If three of the coefficients are $0$, the other two can be scaled to $1$. The fact that $Q$ is reduced leaves us with the following three cases, after using the symmetry $y \leftrightarrow z$: 
        \begin{eqnarray}
        & x^4 + y^2z^2  \label{eq: threecoeffszero1} \\
        & x^2yz + xy^3  \label{eq: threecoeffszero2}\\
        & x^2yz + y^2z^2  \label{eq: threecoeffszero3}
        \end{eqnarray}
        In Case \eqref{eq: threecoeffszero1}, $X$ has two $A_3$-singularities, one at $[0:1:0:0]$ and one at $[0:0:1:0]$. In particular, $X$ contains only equivariant RDPs, so it does not appear in Table \ref{Table Eqn and Aut - char 3, deg 2}.
        
        In Case \eqref{eq: threecoeffszero2}, $X$ has a $D_6$-singularity at $[0:0:1:0]$ and an $A_1$-singularity at $[1:0:0:0]$. As before, $X$ contains only equivariant RDPs, so it does not appear in Table \ref{Table Eqn and Aut - char 3, deg 2}.
        
        In Case \eqref{eq: threecoeffszero3}, $X$ has an $A_1$-singularity at $[1:0:0:0]$ and two $A_3$-singularities at $[0:1:0:0]$ and $[0:0:1:0]$. Again, $X$ contains only equivariant RDPs, so it does not appear in Table \ref{Table Eqn and Aut - char 3, deg 2}. 
        
        \item If two of the coefficients are $0$, we get the following cases, again using scaling and the symmetry $y \leftrightarrow z$:
        \begin{eqnarray}
        & x^4 + ax^2yz + y^2z^2 \label{eq: twocoeffszero1}\\
        & x^4 + x^2yz + xy^3  \label{eq: twocoeffszero2}\\
       &  x^4 + y^2z^2 + xy^3  \label{eq: twocoeffszero3}\\
       &  x^2yz + y^2z^2 + xy^3 \label{eq: twocoeffszero4}\\
        & x^2yz + xy^3 + xz^3 \label{eq: twocoeffszero5}\\
       &  xy^3 + xz^3 + y^2z^2 \label{eq: twocoeffszero6}
        \end{eqnarray}
        In Case \eqref{eq: twocoeffszero1}, we must have $a^2 \neq 1$, otherwise $Q$ is a double conic. Then, $Q$ is the union of two conics, tangent at the points $[0:1:0]$ and $[0:0:1]$. The singularities of $X$ over these two points are two $A_3$-singularities, hence $X$ contains only equivariant RDPs and does not appear in Table \ref{Table Eqn and Aut - char 3, deg 2}.
        
        In Case \eqref{eq: twocoeffszero2}, $X$ has a $D_6$-singularity at $[0:0:1:0]$, hence $X$ contains only equivariant RDPs and does not appear in Table \ref{Table Eqn and Aut - char 3, deg 2}. 
        
        In Case \eqref{eq: twocoeffszero3}, $X$ has an $A_3$-singularity at $[0:0:1:0]$ and an $A_2$-singularity at $[1:-1:0:0]$. By \cite{MartinStadlmayr}, $\Aut_{\widetilde{X}}^0$ is trivial, hence Proposition \ref{prop: fitting argument A_n} implies $\Aut_X^0 = \mu_3$.
        
        In Case \eqref{eq: twocoeffszero4}, $X$ has an $A_3$-singularity at $[0:0:1:0]$, an $A_2$-singularity at $[1:1:1:0]$, and an $A_1$-singularity at $[1:0:0:0]$. By \cite{MartinStadlmayr}, $\Aut_{\widetilde{X}}^0$ is trivial, hence Proposition \ref{prop: fitting argument A_n} implies $\Aut_X^0 = \mu_3$.
        
        In Case \eqref{eq: twocoeffszero5}, $X$ has an $A_5$-singularity at $[0:1:-1:0]$ and an $A_1$-singularity at $[1:0:0:0]$. By \cite{MartinStadlmayr}, $\Aut_{\widetilde{X}}^0$ is trivial, hence Proposition \ref{prop: fitting argument A_n} implies $\Aut_X^0 = \mu_3$. 
        
        In Case \eqref{eq: twocoeffszero6}, $X$ has an $E_6^1$-singularity at $[1:0:0:0]$. It is elementary to check that $\Aut_X^0 = \Stab_{\PGL_3}(Q) = \mu_3$ in this case. Alternatively, observe that this case does not occur in the classification of invariant quartics for the actions (b), (c), and (d) below, so $\Aut_X^0[F] = \mu_3$, where $\Aut_X^0[F]$ is the Frobenius kernel of $\Aut_X^0$. In particular, $\mu_3$ is normal in $\Aut_X^0$, so the $\Aut_X^0$-action preserves the eigenspaces of the $\mu_3$-action, hence the induced action on $\bbP^2$ is diagonal. This simplifies the calculation of the stabilizer of $Q$ considerably.
        
        \item If only one coefficient is $0$, then we get the following cases, again using scaling and the symmetry $y \leftrightarrow z$: 
        \begin{eqnarray}
        &  x^4 + ax^2yz + xy^3 + xz^3 \label{eq: onecoeffzero1} \\
       & x^4 + a^3x^2yz + xy^3 + y^2z^2 \label{eq: onecoeffzero2} \\
       & x^4 + xy^3 + xz^3 + ay^2z^2 \label{eq: onecoeffzero3} \\
       & ax^2yz + xy^3 + xz^3 + y^2z^2 \label{eq: onecoeffzero4}  
        \end{eqnarray}
        In Case \eqref{eq: onecoeffzero1}, $X$ has an $A_5$-singularity at $[0:1:-1:0]$. By \cite{MartinStadlmayr}, $\Aut_{\widetilde{X}}^0$ is trivial, so $\Aut_X^0 = \mu_3$ by Remark \ref{rem: A5}. 
        
        Consider Case \eqref{eq: onecoeffzero2}. If $a^2 = 1$, then $X$ has an $A_6$-singularity at $[0:0:1:0]$, hence $X$ contains only equivariant RDPs and does not occur in Table \ref{Table Eqn and Aut - char 3, deg 2}. If $a^2 \neq 1$, then $X$ has an $A_3$-singularity at $[0:0:1:0]$ and an $A_2$-singularity at $[a^2 - 1:a^4 + a^2 + 1:a^3:0]$. In this case, $\Aut_{\widetilde{X}}^0$ is trivial by \cite{MartinStadlmayr}, so $\Aut_X^0 = \mu_3$ by Proposition \ref{prop: fitting argument A_n}.
        
        In Case \eqref{eq: onecoeffzero3}, $X$ has two $A_2$-singularities, one at $[1:0:-1:0]$ and one at $[1:-1:0:0]$. We leave it to the reader to check that $\Aut_X^0 = \mu_3$ in this case.
        
        In Case \eqref{eq: onecoeffzero4}, $X$ has an $A_1$-singularity at $[1:0:0:0]$, and additional singularities at $[u:au^2:1]$, where $u$ is a solution of $a^3u^6 - au^3 + 1 = 0$. If $a = 1$, then $u=-1$ is unique and the resulting singularity of $X$ is of type $A_5$. If $a \neq 1$, then $X$ has two singularities of type $A_2$. In both cases, $\Aut_{\widetilde{X}}^0$ is trivial by \cite{MartinStadlmayr}, so $\Aut_X^0 = \mu_3$ if $a = 1$ by Remark \ref{rem: A5}. If $a \neq 1$, one can check $\Aut_X^0 = \mu_3$ directly.
        
        \item If no coefficient is $0$, we get the following case
         \begin{eqnarray}
        x^4 + xy^3 + xz^3 + ax^2yz + by^2z^2 \label{eq: nocoeffzero}
         \end{eqnarray}
        Here, $X$ has singularities at the points $[bu:au^2:b]$, where $u$ is a solution of $a^3u^6 + (b^3 - a^2b^2)u^3 + b^3 = 0$. If $(b^3 - a^2b^2)^2 = a^3b^3$, then $x$ is unique and the resulting singularity of $X$ is of type $A_5$. If $(b^3 - a^2b^2)^2 \neq a^3b^3$, then $X$ has two singularities of type $A_2$. In both cases, $\Aut_{\widetilde{X}}^0$ is trivial by \cite{MartinStadlmayr}, so $\Aut_X^0 = \mu_3$ by Proposition \ref{prop: fitting argument A_n} and Remark \ref{rem: A5}. 
    \end{itemize}
    \item[(b)] There are three types of quartics which are invariant under this $\mu_3$-action, namely
    \begin{eqnarray*}
    & f_2(x,y)z^2 \\
    & f_3(x,y)z + f_0 z^4 \\
    & f_4(x,y)+ f_1(x,y)z^3
    \end{eqnarray*}
    where the $f_i$ are homogeneous polynomials of degree $i$ in $x$ and $y$. All quartics in the first family contain a double line, hence they lead to non-normal $X$. For the same reason, we have $f_3,f_4 \neq 0$ in the latter two families. In the third family, we have $f_1 \neq 0$, for otherwise $Q$ has a quadruple point and the corresponding singularity on $X$ is not an RDP. The $\GL_2$-action on $x,y$ normalizes the $\mu_3$-action, hence acts on the space of invariant quartics. Similarly, substitutions of the form $z \mapsto z + \beta x + \gamma y$ with $\beta, \gamma \in k$ act on the space of invariant quartics. Using these substitutions, and keeping in mind that $Q$ must be reduced, we obtain the following simplified normal forms:
    \begin{eqnarray}
    &xy(x+y)z   \label{eq: CaseB1} \\
    &xy(x+y)z + z^4  \label{eq: CaseB2} \\
    &x^2yz + z^4  \label{eq: CaseB3} \\
    &y^4     + xz^3 \label{eq: CaseB4} \\
     &y^4 + x^2y^2 + xz^3 \label{eq: CaseB5} \\    
    & x^2y^2 + xz^3 \label{eq: CaseB6} \\  
    & x^3y + xz^3 \label{eq: CaseB7} 
    \end{eqnarray}
    
\noindent    In Case \eqref{eq: CaseB1}, $X$ has a $D_4$-singularity at $[0:0:1:0]$ and three $A_1$-singularities, at $[1:0:0:0],[0:1:0:0]$, and $[1:-1:0:0]$. In particular, $X$ contains only equivariant RDPs, so it does not occur in Table \ref{Table Eqn and Aut - char 3, deg 2}.
    
\noindent    In Case \eqref{eq: CaseB2}, $X$ has an $A_2$-singularity at $[1:1:1:0]$ and three $A_1$-singularities, at $[1:0:0:0],[0:1:0:0]$, and $[1:-1:0:0]$. By \cite{MartinStadlmayr}, $\Aut_{\widetilde{X}}^0$ is trivial, hence $\Aut_X^0 = \mu_3$ by Proposition \ref{prop: fitting argument A_n}.
    
\noindent    In Case \eqref{eq: CaseB3}, $X$ has a $D_5$-singularity at $[0:1:0:0]$ and an $A_1$-singularity at $[1:0:0:0]$. In particular, $X$ contains only equivariant RDPs, so it does not occur in Table \ref{Table Eqn and Aut - char 3, deg 2}.
    
\noindent    In Case \eqref{eq: CaseB4}, $X$ has a singularity of type $E_6^0$ at $[1:0:0:0]$. In Table \ref{Table Eqn and Aut - char 3, deg 2}, we give an action of a group scheme $G$ of length $27$ which preserves $X$ and such that no subgroup scheme of $G$ lifts to $\widetilde{X}$. Moreover, we give an action of $\Aut_{\widetilde{X}}^0 = \mathbb{G}_m$ on $X$. As in the corresponding case in degree $3$, we leave it to the reader to check that these two actions generate $\Aut_X^0$. 
    
\noindent    In Case \eqref{eq: CaseB5}, $X$ has three $A_2$-singularities, at $[1:0:0:0],[1:-1:1:0],$ and $[1:1:1:0]$. In Table \ref{Table Eqn and Aut - char 3, deg 2}, we describe an action of $\alpha_3^2 \rtimes \mu_3$ on $X$. By \cite{MartinStadlmayr}, $\Aut_{\widetilde{X}}^0$ is trivial, hence $\Aut_X^0 = \alpha_3^2 \rtimes \mu_3$ by Proposition \ref{prop: fitting argument A_n}. 
    
\noindent    In Case \eqref{eq: CaseB6}, $X$ has an $A_5$-singularity at $[0:1:0:0]$ and an $A_2$-singularity at $[1:0:0:0]$. In Table \ref{Table Eqn and Aut - char 3, deg 2}, we give an action of $\alpha_3^2 \rtimes \mathbb{G}_m$ on $X$. By \cite{MartinStadlmayr}, we have $\Aut_{\widetilde{X}}^0 = \mathbb{G}_m$, so Proposition \ref{prop: fitting argument A_n} and Remark \ref{rem: A5} show that $\Aut_X^0 = \alpha_3^2 \rtimes \mathbb{G}_m$.
    
\noindent    In Case \eqref{eq: CaseB7}, $X$ admits a singularity of type $E_7^0$ at $[0:1:0:0]$. In Table \ref{Table Eqn and Aut - char 3, deg 2}, we give an action of $\alpha_3$ on $X$ that does not fix the $E_7^0$-singularity as well as the action of $\Aut_{\widetilde{X}}^0 = \mathbb{G}_a \rtimes \mathbb{G}_m$. We leave it to the reader to check that these two actions generate $\Aut_X^0$. 
    \item[(c)] We can write the equation of $Q$ as 
    $$
    f_0 x^4 + x^3 f_1(y,z) + x^2 f_2(y,z) + x f_3(y,z) + f_4(y,z),
    $$
    where the $f_i$ are homogeneous of degree $i$ in $y$ and $z$. Applying the $\alpha_3$-action, we obtain
    $$
    f_0 x^4 + \varepsilon f_0 x^3y + x^3 f_1(y,z) + (x^2 - \varepsilon xy + \varepsilon^2y^2)f_2(y,z) + (x + \varepsilon y) f_3(y,z) + f_4(y,z).
    $$
    By considering the non-zero term of highest degree in $x$, we see that this is a multiple of the original equation if and only if it is equal to it. Comparing the coefficients of $\varepsilon^2$ and $\varepsilon$, we see that this happens if and only if $f_0 = f_2 = f_3 = 0$. Hence, $Q$ is of the form
    $$
    x^3f_1(y,z) + f_4(y,z).
    $$
    But then $Q$ is invariant under the $\mu_3$-action $[x:y:z] \mapsto [\lambda x:y:z]$ and hence, after interchanging $x$ and $z$, $Q$ is as in Cases \eqref{eq: CaseB4}, \eqref{eq: CaseB5}, \eqref{eq: CaseB6}, and \eqref{eq: CaseB7}. 
    \item[(d)] We write the equation of $Q$ as $\sum a_{ijk} x^iy^jz^k$. As in Case (c), one checks that $Q$ is preserved by the $\alpha_3$-action if and only if its equation is preserved by the $\alpha_3$-action. Applying the $\alpha_3$-action and comparing the coefficients of $\varepsilon$ and $\varepsilon^2$, respectively, we see that $Q$ is $\alpha_3$-invariant if and only if the following two conditions are satisfied:
    
    \noindent
    \begin{tabular}{lcc}
    $a_{400}x^3y + a_{310} x^3z  - a_{220}x^2yz + a_{211}x^2z^2 - a_{220}xy^3 - a_{211}xy^2z - (a_{202} + a_{121})xyz^2$ &&\\
    \hspace{8mm}$\hspace{1mm}+\hspace{1mm} a_{112}xz^3 + a_{130}y^4 + (a_{121} + a_{040})y^3z + a_{112}y^2z^2 + (a_{103}-a_{022})yz^3 + a_{013}z^4$ &$= $&$0$\\
   $-a_{400}x^3z + a_{220}x^2z^2 - a_{220} xy^2z + (a_{121} + a_{202})xz^3 + a_{220}y^4 + (a_{211}-a_{130})y^3 z$ &\\
   \hspace{8mm}$\hspace{1mm}+\hspace{1mm} (a_{121} + a_{202})y^2z^2 + (a_{022} - a_{103})z^4 $&$=$& $0 $
    \end{tabular}
    In other words, $Q$ is $\alpha_3$-invariant if and only if it is given by an equation of the form
    $$
     a(xz+y^2)^2 + bz^2(xz + y^2) +  cx^3z + dy^3z + ez^4.
    $$
    The substitutions of the form $x \mapsto \beta^2 x + \beta \gamma y +\delta z, y \mapsto \beta y + \gamma z$ with $\beta \in k^\times$ and $\gamma, \delta \in k$ normalize the $\alpha_3$-action, hence they preserve the space of $\alpha_3$-invariant quartics. If $a \neq 0$, we can scale it to $a = 1$ and use $\delta$ to kill $b$. If $a = 0$, then $b \neq 0$, otherwise $Q$ contains a triple line. Then, we can assume $b = 1$. Using the other substitutions, we arrive at one of the following five simplified normal forms:
    \begin{eqnarray}
    (xz+y^2)^2 + z^4 \label{eq: CaseD1} \\
    (xz+y^2)^2 + y^3z \label{eq: CaseD2} \\
    (xz+y^2)^2 + x^3z + a^6 z^4 \label{eq: CaseD3} \\
    z^2(xz + y^2) + y^3z \label{eq: CaseD4}  \\
    z^2(xz + y^2) + x^3z  \label{eq: CaseD5} 
    \end{eqnarray}
    \noindent
    In Case \eqref{eq: CaseD1}, $X$ has an $A_7$-singularity at $[1:0:0:0]$, hence $X$ contains only equivariant RDPs and does not occur in Table \ref{Table Eqn and Aut - char 3, deg 2}.
    
 \noindent    In Case \eqref{eq: CaseD2}, $X$ has an $A_4$-singularity at $[1:0:0:0]$ and an $A_2$-singularity at $[0:0:1:0]$. By \cite{MartinStadlmayr}, $\Aut_{\widetilde{X}}^0$ is trivial, hence $\Aut_X^0 = \alpha_3$ by Proposition \ref{prop: fitting argument A_n}.
    
  \noindent   In Case \eqref{eq: CaseD3}, $X$ is singular precisely at $[-a^2:\pm a:1]$. If $a = 0$, then this singularity is of type $A_5$ and, by \cite{MartinStadlmayr} and Remark \ref{rem: A5}, we have $\Aut_X^0 = \alpha_3$. If $a \neq 0$, then both singularities are of type $A_2$. Direct calculation shows that $\Aut_X^0 = \alpha_3$.
    
  \noindent   In Case \eqref{eq: CaseD4}, $X$ has an RDP of type $E_7^1$ at $[1:0:0:0]$. In particular, $X$ contains only equivariant RDPs and does not occur in Table \ref{Table Eqn and Aut - char 3, deg 2}.
    
  \noindent   Finally, in Case \eqref{eq: CaseD5}, has an RDP of type $A_5$ at $[0:1:0:0]$. By \cite{MartinStadlmayr}, $\Aut_{\widetilde{X}}^0$ is trivial, hence, by Remark \ref{rem: A5}, we have $\Aut_X^0 = \alpha_3$. Note that this surface is not isomorphic to the one in Case \eqref{eq: CaseD3} (with $a = 0$), since here, $Q$ contains a line, while in the other case, $Q$ is irreducible.  
\end{enumerate}

\smallskip 

\underline{If $d = 1$}, then $X$ is a double cover of the quadratic cone in $\mathbb{P}^3$ branched over a sextic curve $S$. In particular, $X$ is given by an equation in Weierstrass form
$$
y^2 = x^3 + a_2(s,t)x^2 + a_4(s,t)x + a_6(s,t),
$$
where $a_i$ is a homogeneous polynomial of degree $i$ in $s$ and $t$, and $S$ is given by the equation
\begin{equation} \label{eq: sextic}
x^3 + a_2(s,t)x^2 + a_4(s,t)x + a_6(s,t) = 0
\end{equation}
in $\mathbb{P}(1,1,2)$. By Proposition \ref{prop: ShortExactAutSequences}, we have $\Aut_X^0 = \Stab_{\Aut_{\mathbb{P}(1,1,2)}}(S)^0$. Since $\mathbb{P}(1,1,2)$ is an $\Aut_X$-equivariant RDP del Pezzo surface of degree $8$, Theorem \ref{thm: equivariantmain} and \cite{MartinStadlmayr} show that $\Aut_{\mathbb{P}(1,1,2)} = (\mathbb{G}_a^3 \rtimes \GL_2)/(\mathbb{Z}/2\mathbb{Z})$, which acts via substitutions of the form
\begin{eqnarray*}
x & \mapsto & x + f_2(s,t) \\
s & \mapsto & \alpha s + \beta t \\
t & \mapsto & \gamma s + \delta t
\end{eqnarray*}
where $f_2$ is a homogeneous polynomial of degree $2$ in $s$ and $t$, and $\alpha, \beta, \gamma, \delta$ are scalars such that $\alpha \delta - \beta \gamma$ is invertible.

Now, assume $\Aut_X^0$ is non-trivial. Then, it contains $G \in \{\alpha_3,\mu_3\}$. If $G = \mu_3$, we claim that there are only three embeddings of $G$ into $\Aut_{\mathbb{P}(1,1,2)}^0 = \mathbb{G}_a^3 \rtimes \GL_2$. First, by counting the possible Jordan normal forms, observe that there are only three conjugacy classes of embeddings of $G$ into $\GL_2$. Then, applying \cite[Th\'eor\`eme 5.1.1 (ii) (b)]{SGA3II-ExposeXVII}, we see that every such embedding lifts uniquely, up to conjugation by $\mathbb{G}_a^3$, to an embedding of $\mu_3$ into $\mathbb{G}_a^3 \rtimes \GL_2$. Hence, if $G = \mu_3$, we may assume that it acts in one of the following three ways on $\mathbb{P}(1,1,2)$:
\begin{enumerate}
    \item[(a)] $\mu_3: [s:t:x] \mapsto [s: \lambda t :x ]$,
    \item[(b)] $\mu_3: [s:t:x] \mapsto [\lambda s : \lambda t : x]$,
    \item[(c)] $\mu_3: [s:t:x] \mapsto [\lambda s : \lambda^{-1} t :x]$.
\end{enumerate}
Next, assume that $G = \alpha_3$ and the image of $G$ in $\GL_2$ is trivial. Then, the embedding of $G$ into $\mathbb{G}_a^3 \rtimes \GL_2$ is given by a homomorphism $\alpha_3 \to \mathbb{G}_a^3$, which in turn corresponds to a choice of homogeneous polynomial $f_2$ of degree $2$ in $s$ and $t$. According to whether this polynomial has one double zero or two simple zeroes, we can conjugate the embedding of $\alpha_3$ using the $\GL_2$-action to get one of the following two $\alpha_3$-actions:
\begin{enumerate}
    \item[(d)] $\alpha_3: [s:t:x] \mapsto [s:t:x + \varepsilon s^2]$,
    \item[(e)] $\alpha_3: [s:t:x] \mapsto [s:t:x + \varepsilon st]$.
\end{enumerate}
Finally, assume that $G = \alpha_3$ and the image of $G$ in $\GL_2$ is non-trivial. After conjugating by elements of $\GL_2$, we can assume that the image of $G$ in $\GL_2$ acts as $(s,t) \mapsto (s + \varepsilon t,t)$. An action of $\alpha_3$ on $\bbP(1,1,2)$ that lifts this embedding of $\alpha_3$ into $\GL_2$ acts on $x$ via
$$
x \mapsto x + p(\varepsilon) s^2 + q(\varepsilon)  st + r(\varepsilon)  t^2
$$
where $p,q,r$ are polynomials of degree $2$ in $\varepsilon$ satisfying the following conditions:
\begin{eqnarray*}
p(0) = q(0)=r(0) = 0 \\
p(\varepsilon + \varepsilon') = p(\varepsilon) + p(\varepsilon') \\
q(\varepsilon+\varepsilon') = q(\varepsilon) + q(\varepsilon') - p(\varepsilon)\varepsilon' \\
r(\varepsilon+\varepsilon') = r(\varepsilon) + r(\varepsilon') + q(\varepsilon)\varepsilon' + p(\varepsilon)\varepsilon'^2
\end{eqnarray*}
Solving this system of equations, we obtain $p = 0$, $q(\varepsilon) = \alpha  \varepsilon$ and $r(\varepsilon) = \beta \varepsilon - \alpha \varepsilon^2$ for scalars $\alpha , \beta \in k$. Conjugating with the substitution $x \mapsto x - \alpha s^2 + \beta st$, we obtain that our $\alpha_3$-action is conjugate to the following:
\begin{enumerate}
    \item[(f)] $\alpha_3: [s:t:x] \mapsto [s + \varepsilon t:t:x]$.
\end{enumerate}
We will now classify the sextics as in Equation \eqref{eq: sextic} which are reduced with only simple singularities and invariant under one of the above actions. In particular, note that if all the $a_i$ are scalar multiples of the $i$-th power of the same linear polynomial in $s$ and $t$, then $S$ has a non-simple singularity, so it will not appear in our list. Calculating $\Aut_X^0$ is straightforward here, using our description of $\Aut_{\bbP(1,1,2)}$ above, so it will be left to the reader without further mention. The results can be found in Table \ref{Table Eqn and Aut - char 3, deg 1}:

\begin{enumerate}
    \item[(a)] The sextic $S$ is invariant if and only if the $t$-degree of every monomial that occurs in the equation of $S$ is divisible by $3$. Note that the substitutions 
    \begin{eqnarray*}
    x & \mapsto & x + \alpha s^2 \\
    s & \mapsto & \beta s \\
    t & \mapsto & \gamma t + \delta s
    \end{eqnarray*}
    act on the space of sextics satisfying the condition on the $t$-degree, hence we can apply them to arrive at the following normal forms for $S$:
    \begin{eqnarray}
    x^3 + s^4x + t^6 \label{eq: deg1casea1} \\
    x^3 + s^4x + s^3t^3 \label{eq: deg1casea2}\\
    x^3 + st^3x \label{eq: deg1casea2'}\\
    x^3 + st^3x + as^3t^3 + t^6  \label{eq: deg1casea3} \\
    x^3 + st^3x + s^3t^3 \label{eq: deg1casea4} \\
    x^3 + s^2x^2 + s^3t^3 \label{eq: deg1casea5} \\
    x^3 + s^2x^2 + a^3s^3t^3 + t^6 \label{eq: deg1casea6} \\
    x^3 + s^2x^2 + st^3x + a^3s^3t^3 + b^3t^6 \label{eq: deg1casea7}
    \end{eqnarray}
    \noindent 
    In Case \eqref{eq: deg1casea1}, $X$ has an $E_6^0$-singularity at $[0:1:-1:0]$. 
    
    \noindent
    In Case \eqref{eq: deg1casea2}, $X$ has an $E_8^1$-singularity at $[0:1:0:0]$.  
    
    \noindent
    In Case \eqref{eq: deg1casea2'}, $X$ has an $E_7^0$-singularity at $[1:0:0:0]$ and an $A_1$-singularity at $[0:1:0:0]$.
    
    \noindent
    In Case \eqref{eq: deg1casea3}, if $a \neq 0$, then $X$ has an $E_6^0$-singularity at $[1:0:0:0]$. If $a = 0$, then $X$ has an $E_7^0$-singularity at $[1:0:0:0]$.
    
    \noindent
    In Case \eqref{eq: deg1casea4}, $X$ has an $E_6^0$-singularity at $[1:0:0:0]$ and an $A_1$-singularity at $[0:1:0:0]$.
    
    \noindent
    In Case \eqref{eq: deg1casea5}, $X$ has an $E_6^1$-singularity at $[0:1:0:0]$ and an $A_2$-singularity at $[1:0:0:0]$.
    
    \noindent
    In Case \eqref{eq: deg1casea6}, $X$ is singular at $[1:0:0:0], [1:-a:0:0]$, and $[0:1:-1:0]$. If $a \neq 0$, then all singular points are $A_2$-singularities. If $a = 0$, then the first two combine to an $A_5$-singularity, while the latter stays an $A_2$-singularity.
    
    \noindent
    Consider Case \eqref{eq: deg1casea7}. Here, $X$ is singular at $[1:0:0:0]$ and $[u:1:u^{-1}:0]$, where $u$ is a solution of $au^2 + (b-1)u + 1 = 0$. If $b = 0$, then $X$ has an additional $A_1$-singularity at $[0:1:0:0]$.
    If the first three singular points are distinct, which happens if and only if $a \neq 0,(b-1)^2$, then
    they are $A_2$-singularities. Now, consider the case where $a = 0$ or $a = (b-1)^2$, but not both: Then $[1:0:0:0]$ is an $A_5$-singularity and $[u:1:u^{-1}:0]$ an $A_2$-singularity if $0=a \neq (b-1)^2$, and the other way round if $0 \neq a=(b-1)^2$.
    Note that the substitution $t \mapsto t + (b-1)s, x \mapsto x + (b-1)^3s^2$ maps the family with $a = (b-1)^2$ to the one with $a = 0$, which is why only the latter occurs in Table \ref{Table Eqn and Aut - char 3, deg 1}. Finally, if $a = 0$ and $b = 1$, then $X$ has an $A_8$-singularity at $[1:0:0:0]$. 
    
    \item[(b)] The sextic $S$ is invariant if and only if the degrees of the $a_i$ are divisible by $3$. This happens if and only if $a_2 = a_4 = 0$. Using a substitution from $\Aut_{\bbP(1,1,2)}$, we obtain the following normal forms:
    \begin{eqnarray}
    x^3 + s^5t \label{eq: deg1caseb1} \\
    x^3 + s^4t^2 \label{eq: deg1caseb2} \\
    x^3 + s^4t^2 + s^2t^4 \label{eq: deg1caseb3}
    \end{eqnarray}
    In Case \eqref{eq: deg1caseb1}, $X$ has an $E_8^0$-singularity at $[0:1:0:0]$. 
    
    \noindent 
    In Case \eqref{eq: deg1caseb2}, $X$ has an $E_6^0$-singularity at $[0:1:0:0]$ and an $A_2$-singularity at $[1:0:0:0]$. 
    
    \noindent 
    In Case \eqref{eq: deg1caseb3}, $X$ has four $A_2$-singularities, at $[1:0:0:0],[0:1:0:0],[1:-1:0:0],$ and $[1:1:0:0]$. 
 
    \item[(c)] The sextic $S$ is invariant if and only if for every monomial in the equation of $S$, the difference between the $s$- and $t$-degree is divisible by $3$. We may assume that not both $a_2$ and $a_4$ are zero, otherwise we get Cases \eqref{eq: deg1caseb1}, \eqref{eq: deg1caseb2}, and \eqref{eq: deg1caseb3}.
    Note that the substitutions 
    \begin{eqnarray*}
    x & \mapsto & x + \alpha st \\
    s & \mapsto & \beta s \\
    t & \mapsto & \gamma t
    \end{eqnarray*} 
    normalize our $\mu_3$-action, hence we can apply them to arrive at the following normal forms for $S$:
    \begin{eqnarray}
    x^3 + s^2t^2x \label{eq: deg1casec-2} \\
    x^3 + s^2t^2x + t^6 + a^3s^6 \label{eq: deg1casec-1} \\
    x^3 + stx^2 + s^6 \label{eq: deg1casec1} \\
    x^3 + stx^2 + a^3s^6 + s^3t^3 \label{eq: deg1casec2} \\
    x^3 + stx^2 + a^3s^6 + b^3s^3t^3 + t^6 \label{eq: deg1casec3}
    \end{eqnarray}
    In Case \eqref{eq: deg1casec-2}, $X$ has two $D_4$-singularities, one at $[1:0:0:0]$ and one at $[0:1:0:0]$, hence $X$ contains only equivariant RDPs and does not occur in Table \ref{Table Eqn and Aut - char 3, deg 1}.
    
    \noindent
    Consider Case \eqref{eq: deg1casec-1}. If $a = 0$, then $X$ has a $D_4$-singularity at $[1:0:0:0]$ and an $A_2$-singularity at $[0:1:-1:0]$. If $a \neq 0$, then $X$ has two $A_2$-singularities, one at $[1:0:-a:0]$ and one at $[0:1:-1:0]$.
    
    \noindent
    In Case \eqref{eq: deg1casec1}, $X$ has an RDP of type $D_7$ at $[0:1:0:0]$. In particular, $X$ contains only equivariant RDPs and does not occur in Table \ref{Table Eqn and Aut - char 3, deg 1}.
    
    \noindent
    Consider Case \eqref{eq: deg1casec2}. If $a = 0$, then $X$ has RDPs of type $D_4$ at $[1:0:0:0]$ and $[0:1:0:0]$, hence $X$ contains only equivariant RDPs and does not occur in Table \ref{Table Eqn and Aut - char 3, deg 1}. If $a \neq 0$, then $X$ has an RDP of type $D_4$ at $[0:1:0:0]$ and an $A_2$-singularity at $[1:-a:0:0]$.
    
     \noindent
    Consider Case \eqref{eq: deg1casec3}. If $a = 0$, we can interchange $s$ and $t$ to reduce to one of the previous two cases. Here, $X$ has singularities at $[1:u:0:0]$, where $u$ is a solution of $u^2 + bu + a = 0$. If $b^2 = a$, then the unique singularity of $X$ is an $A_5$-singularity. If $b^2 \neq a$, then $X$ has two $A_2$-singularities.
    
    \item[(d)] If $a_2 \neq 0$ or $a_4 \neq 0$, then $S$ cannot be $\alpha_3$-invariant. Hence, $a_2 = a_4 = 0$. But then $S$ is given by one of the equations in (b), so we are done.
    
    \item[(e)] By the same argument as in Case (d), we can reduce this to Case (b).
    \item[(f)] The sextic $S$ is $\alpha_3$-invariant if and only if each $a_i$ is invariant under the $\alpha_3$-action. This happens if and only if the $s$-degree of each monomial in the equation of $S$ is divisible by $3$. Interchanging the roles of $s$ and $t$, we can therefore reduce this Case to Case (a). In fact, this explains why each of the surfaces in Case (a) admits an $\alpha_3$-action of this form.
\end{enumerate}

\end{proof}

\newpage

\section*{Appendix: Classification tables}

\begin{table}[h!]
    \centering
    \begin{tabular}{|c||c|c|c|c|} \hline 
   $d$ &  singularities & equation of $X$    & $\Aut_X^{0}$ \\ \hline \hline 
    $2$ & $A_6$ & $w^2 = x^3y + y^3z + z^3x$ 
    & $\mu_7: [\lambda x:\lambda^4 y: \lambda^2 z: w]$ \\ \hline \hline
    $1$ & $A_6 + A_1$ & $y^2 = x^3 + ts^3x + t^5s$  & $\mu_7: [\lambda s:\lambda^4t:x: y]$ \\ \hline
    \end{tabular}
    \caption{Non-equivariant RDP del Pezzo surfaces with vector fields in characteristic $7$}
    \label{Table Eqn and Aut - char 7}
\end{table}

\begin{table}[h!]
    \centering
    \begin{tabular}{|c||c|c|c|c|} \hline 
   $d$ &  RDPs & equation(s) of $X$    &  $\Aut_X^{0}$ \\ \hline \hline 
    $5$ & $A_4$ 
    & \begin{tabular}{ccc}
    $x_0x_2 - x_1^2$ &=& $0$ \\
    $x_0x_3 - x_1x_4$ & $=$& $0$\\
    $x_2x_4 - x_1 x_3$ &$=$& $0$\\
    $x_1x_2 + x_4^2 + x_0x_5$ &$=$& $0$\\
    $x_2^2 + x_3x_4 + x_1x_5$ & $=$ & $0$
    \end{tabular}
    & \begin{tabular}{c}
    $\langle \alpha_5 , \Aut_{\widetilde{X}}^0 \rangle$ with\\
    $ \alpha_5 : $
    \resizebox{4cm}{!}{$\left( \begin{array}{c c c c c c}
    1 & 0 & 0 & 0 & 0 & 0 \\
    0 & 1 & -2 \varepsilon^2 & 2\varepsilon^3 & \varepsilon & 2\varepsilon^4 \\
    0  & 0 & 1 & 2\varepsilon & 0 & -\varepsilon^2 \\
    0 & 0 & 0 & 1 & 0 & -\varepsilon \\
    0 & 0 & \varepsilon & \varepsilon^2 & 1 & -2\varepsilon^3 \\
    0 & 0 & 0 & 0 & 0 & 1 \\
    \end{array} \right)$}
    \end{tabular} 
    \\ \hline \hline 
    $4$ & $A_4$
    & \begin{tabular}{ccc}
    $x_0x_1 - x_2x_3$ &$=$& $0$\\
    $x_0x_4 + x_1x_2 + x_3^2$ &$=$& $0$ 
    \end{tabular}
   
     & \begin{tabular}{c}
    $\langle \alpha_5 , \Aut_{\widetilde{X}}^0 \rangle$ with \\
    $ \alpha_5 : $
    \resizebox{4cm}{!}{$\left( \begin{array}{c c c c c}
    1 & - \varepsilon^3 & -2 \varepsilon & 2  \varepsilon^2 & 2  \varepsilon^4 \\
    0 & 1 &0 &0 & 2  \varepsilon \\
    0 & -  \varepsilon^2 & 1 & -2 \varepsilon &  \varepsilon^3 \\
     0 &  \varepsilon & 0 & 1 &  \varepsilon^2 \\
    0 & 0& 0& 0& 1
    \end{array} \right)$}
    \end{tabular}
    \\ \hline \hline 
    
    \multirow{7}{*}{$3$} & $A_4$ & $x_0^2x_1 + x_1^2x_2 + x_2^2x_3 + x_3^2x_0 = 0 $        
    & $\mu_5: [x_0: \lambda x_1: \lambda^4 x_2: \lambda^3 x_3]$ \\ \cline{2-4}

     & $A_4 + A_1$ & $x_0x_1x_3 + x_0x_2^2 + x_1^2x_2 = 0 $  
    & \begin{tabular}{c}
    $ \alpha_5 \rtimes \bbG_m $ with \\
    $ \alpha_5 : 
    \left( \begin{array}{c c c c}
    1 & \varepsilon  & \varepsilon^2 & -2 \varepsilon ^3\\
    0 & 1  & 2 \varepsilon & -\varepsilon^2 \\ 
    0 & 0 & 1 & -\varepsilon \\
    0 & 0 & 0  & 1 
    \end{array} \right) 
    $ \\
    $ \bbG_m: [x_0:\lambda x_1: \lambda^2 x_2:\lambda^3 x_3]$
    \end{tabular} \\ \hline \hline 
    
    \multirow{2}{*}{$2$} & $A_4 + A_1$ & $w^2 = x^4 + xy^2z + yz^3$ & $\mu_5: [x:\lambda y: \lambda^3 z : w]$ \\ \cline{2-4}
     & $A_4 + A_2$ & $w^2 = xy^3 + yz^3 + x^2z^2$  & $\mu_5: [\lambda^2 x:\lambda y : \lambda^3 z:w]$ \\ \hline \hline

    \multirow{3}{*}{$1$} & $A_4 + A_2+ A_1$ 
    & $y^2= x^3+s^3tx+s^2t^4$
    & $\mu_5:  [s:\lambda t: \lambda^3 x : \lambda^2 y]$ \\ \cline{2-4}
     & $2A_4$ & $y^2 = x^3 + t^4x + s^5t$  & 
    $\alpha_5 \rtimes \mu_5: [\lambda s + \varepsilon t: t:x:y]$ 
    \\ \cline{2-4}
     & $E_8^0$ & $y^2 = x^3 + s^5t$ & 
    $\alpha_5 \rtimes \mathbb{G}_m: [\lambda s + \varepsilon t: \lambda^{-5}t:x:y]$ 
    \\ \hline
    \end{tabular}
    \caption{Non-equivariant RDP del Pezzo surfaces with vector fields in characteristic $5$}
    \label{Table Eqn and Aut - char 5}
\end{table}

\newpage

\begin{table}[h!]
    \centering
    \begin{tabular}{|c|c|c|c|c|} \hline 
   $d$ &  RDPs & equation(s) of $X$    & $\Aut_X^{0}$ \\ \hline \hline 
    \multirow{9}{*}{$6$} & $A_2$ 
    & \resizebox{!}{2.1cm}{
    \begin{tabular}{ccc}
   $x_0x_5 - x_3x_4$ &$=$& $0$ \\
    $x_0x_6 - x_1x_4$ & $=$& $0$\\
    $x_0x_6 - x_2x_3$ &$=$& $0$\\
    $x_3x_6 - x_1x_5$ &$=$& $0$\\
    $x_4x_6 - x_2x_5$ & $=$ & $0$\\
    $x_1x_6 + x_3^2 + x_3x_4$ & $=$ & $0$\\
    $x_2x_6 + x_3x_4 + x_4^2$ & $=$ & $0$\\
    $x_6^2 + x_3x_5 + x_4x_5$ & $=$ & $0$\\
    $x_1x_2 + x_0x_3 + x_0x_4$ & $=$ & $0$
    \end{tabular} 
    }
    
    & 
    \begin{tabular}{c}
    $\langle \alpha_3 , \Aut_{\widetilde{X}}^0 \rangle$ with\\
    $ \alpha_3 : $
    \resizebox{!}{1.65cm}{$\left( \begin{array}{c c c c c c c}
    1 & 0 & 0& 0& 0& 0& 0\\
    -\varepsilon  & 1 & 0& 0& 0& 0& 0\\
    \varepsilon & 0& 1 &0 &0 & 0& 0\\
    -\varepsilon^2 & - \varepsilon & 0& 1 &0 &0 &0 \\
    - \varepsilon^2 & 0 & \varepsilon & 0&1 &0 &0 \\
    0 & 0& 0& 0& 0& 1 &0 \\
    0 & -\varepsilon^2 & -\varepsilon^2 & - \varepsilon & \varepsilon &0 & 1
    \end{array} \right)$}
     \end{tabular} 
    \\ \cline{2-4}
    
     & $A_2 + A_1$ 
    & \resizebox{!}{2.1cm}{
    \begin{tabular}{ccc}
    $x_0^2 - x_1x_5$ &$=$& $0$ \\ %OK
    $x_0x_2 - x_1x_4$ &$=$& $0$ \\ %OK
    $x_0x_3 - x_2x_4$ & $=$& $0$\\ %OK
    $x_0x_4 - x_2x_5$ &$=$& $0$\\%OK
    $x_0x_5 - x_2x_6$ &$=$& $0$\\%OK
    $x_1x_3 - x_2^2$ & $=$ & $0$\\ %OK
    $x_3x_5 - x_4^2$ & $=$ & $0$\\ %OK
    $x_3x_6 - x_4x_5$ & $=$ & $0$\\ %OK
    $x_4x_6 - x_5^2$ & $=$ & $0$ %OK
    \end{tabular}
    }
   
    & 
    \begin{tabular}{c}
    $\langle \alpha_3 , \Aut_{\widetilde{X}}^0 \rangle$ with\\
    $ \alpha_3 : $
    \resizebox{!}{1.65cm}{$\left( \begin{array}{c c c c c c c}
    1 & -\varepsilon & 0& 0& 0& 0& 0\\
    0 & 1 & 0& 0& 0& 0& 0\\
    0 & 0& 1 &0 &0 & 0& 0\\
    0 & 0 & 0& 1 &0 &0 &0 \\
    0& 0 & -\varepsilon & 0&1 &0 &0 \\
    \varepsilon & \varepsilon^2& 0& 0& 0& 1 &0 \\
    0 & 0 & 0& 0& 0&0 & 1
    \end{array} \right)$}
     \end{tabular} 
    \\ \hline \hline

      \multirow{6.5}{*}{$5$} & $A_2$ 
    & \resizebox{!}{1.40cm}{
    \begin{tabular}{ccc}
    $x_0x_2 - x_1x_5$ &=& $0$ \\
    $x_0x_2 - x_3x_4$ & $=$& $0$\\
    $x_0x_3 + x_1^2 + x_1x_4$ &$=$& $0$\\
    $x_0x_5 + x_1x_4 + x_4^2$ &$=$& $0$\\
    $x_3x_5 + x_1x_2 + x_2x_4$ & $=$ & $0$
    \end{tabular}
    }
    
    & 
    \begin{tabular}{c}
    $\langle \alpha_3 , \Aut_{\widetilde{X}}^0 \rangle$ with\\
    $ \alpha_3 : $
    \resizebox{!}{1.25cm}{$\left( \begin{array}{c c c c c c}
    1 & \varepsilon & 0 & - \varepsilon^2 & -\varepsilon & -\varepsilon^2 \\
    0 & 1 & - \varepsilon^2 & \varepsilon & 0 & 0 \\
    0  & 0 & 1 & 0 & 0 & 0 \\
    0 & 0 & \varepsilon & 1 & 0 & 0 \\
    0 & 0 & -\varepsilon^2 & 0 & 1 & -\varepsilon \\
    0 & 0 & -\varepsilon & 0 & 0 & 1 \\
    \end{array} \right)$}
    \end{tabular} 
    \\ \cline{2-4}
    
     & $A_2 + A_1$ 
    & \resizebox{!}{1.40cm}{
    \begin{tabular}{ccc}
    $x_0^2 - x_1x_4$ &=& $0$ \\
    $x_0x_2 - x_1x_3$ & $=$& $0$\\
    $x_0x_3 - x_2x_4$ &$=$& $0$\\
    $x_0x_4 - x_2x_5$ &$=$& $0$\\
    $x_3x_5 - x_4^2$ & $=$ & $0$
    \end{tabular}
    }
   
    & 
    \begin{tabular}{c}
    $\langle \alpha_3 , \Aut_{\widetilde{X}}^0 \rangle$ with\\
    $ \alpha_3 : $
    \resizebox{!}{1.25cm}{$\left( \begin{array}{c c c c c c}
    1 & -\varepsilon & 0 & 0 & 0 & 0 \\
    0 & 1 & 0 & 0 & 0 & 0 \\
    0  & 0 & 1 & 0 & 0 & 0 \\
    0 & 0 & -\varepsilon & 1 & 0 & 0 \\
    \varepsilon & \varepsilon^2 & 0 & 0 & 1 & 0 \\
    0 & 0 & 0 & 0 & 0 & 1 \\
    \end{array} \right)$}
    \end{tabular} 
    \\ \hline \hline
        
  \multirow{13}{*}{$4$} & $A_2$ 
    & \begin{tabular}{ccc}
      $x_0x_1 + x_2x_4 + x_3x_4$ & $=$ & $0$ \\
      $x_0x_4 + x_1x_4 + x_2x_3$ & $=$ & $0$
    \end{tabular}
    & 
   $\mu_3: [x_0: x_1 : \lambda x_2 : \lambda x_3 : \lambda^2 x_4]$
    \\ \cline{2-4}
    
   & $A_2 + A_1$ 
    & \begin{tabular}{ccc}
        $x_0x_1  - x_2x_3$ & $=$& $0$ \\
    $x_1x_2 + x_2x_4 + x_3x_4$ &=& $0$ \\
    \end{tabular}

    & 
    \begin{tabular}{c}
    $\alpha_3 \rtimes \mathbb{G}_m$ with\\
    $ \alpha_3 : $
    \resizebox{!}{1.2cm}{$\left( \begin{array}{c c c c c}
    1 & 0 & 0 & 0 & 0 \\
    -\varepsilon^2 & 1 & \varepsilon & -\varepsilon & 0 \\
    -\varepsilon & 0 & 1 & 0 & 0\\
    \varepsilon & 0 & 0 & 1 & 0\\
    -\varepsilon^2 & 0 & -\varepsilon & 0 & 1
    \end{array} \right)$} \\
    $\bbG_m:  [\lambda^2 x_0:x_1:\lambda x_2:\lambda x_3:x_4] $
    \end{tabular} 
    \\ \cline{2-4}
    
   & $A_2 + 2A_1$ 
    & \begin{tabular}{ccc}
        $x_0^2 - x_3x_4$ & $=$& $0$ \\
    $x_0 x_3  - x_1x_2$ &=& $0$ \\
    \end{tabular}
    
    & 
    \begin{tabular}{c}
    $\alpha_3 \rtimes \mathbb{G}_m^2$ with\\
    $ \alpha_3 : $
    \resizebox{!}{1.2cm}{$\left( \begin{array}{c c c c c}
    1 & 0 & 0 & 0 & -\varepsilon\\
    0 & 1 & 0 & 0 & 0 \\
    0 & 0 & 1 & 0 & 0\\
    \varepsilon & 0 & 0 & 1 & \varepsilon^2 \\
    0 & 0 & 0 & 0 & 1
    \end{array} \right)$} \\
    $\bbG_m^2: [x_0:\lambda_1 x_1:\lambda_2 x_2:\lambda_1 \lambda_2 x_3:(\lambda_1 \lambda_2)^{-1} x_4]$
    \end{tabular} 
    \\ \hline

    \end{tabular}
    \caption{Non-equivariant RDP del Pezzo surfaces of degree at least $4$ with vector fields in \mbox{characteristic $3$}}
    \label{Table Eqn and Aut - char 3, deg at least 4}
\end{table}

\newpage
\newgeometry{top=27 mm, bottom=20 mm, left=16 mm, right=16mm}

\begin{table}[h!]
    \centering
    \begin{tabular}{|c||c|c|c|} \hline 
   $d$ &  RDPs & equation(s) of $X$    &  $\Aut_X^{0}$ \\ \hline \hline 
    \multirow{30}{*}{$3$} & $A_2$ & $x_0^2x_1 + x_0x_1^2 + x_2^2x_3 + x_2x_3^2 = 0$ & $\mu_3: [x_0: x_1:\lambda x_2: \lambda x_3]$
    \\ \cline{2-4} 
    
      & $A_2 + 2A_1$ & $x_0^2x_1 + x_0^2x_2 + x_0 x_3^2 + x_1x_2x_3 = 0$ & $\mu_3: [x_0: \lambda x_1: \lambda x_2: \lambda^2 x_3]$
    \\ \cline{2-4} 
    
     & $2A_2$ & \begin{tabular}{c}
    $x_0^3 + x_1x_2x_3 + x_0x_1^2 + ax_0^2x_1 = 0$ \\
    with $a^2 \neq 1$
    \end{tabular}
    &
      \begin{tabular}{c}
    $\langle  \alpha_3,\alpha_3, \bbG_m \rangle$ with \\
    $ \alpha_3: [x_0 + \varepsilon x_2: x_1 : x_2: a\varepsilon x_0 - \varepsilon x_1 -a\varepsilon^2 x_2  + x_3]$ \\
    $ \alpha_3: [x_0 + \varepsilon x_3: x_1 : a\varepsilon x_0 - \varepsilon x_1 + x_2 -a\varepsilon^2 x_3: x_3] $ \\
    $\bbG_m: [x_0:x_1: \lambda x_2:\lambda^{-1}x_3]$
    \end{tabular}
    \\ \cline{2-4} 
     
       & $2A_2 + A_1$ & $x_0^3 + x_1x_2x_3  + x_0^2x_1 = 0$ &
         \begin{tabular}{c}
    $\langle  \alpha_3,\alpha_3, \bbG_m \rangle$ with \\
    $ \alpha_3: [x_0 + \varepsilon x_2: x_1 : x_2: \varepsilon x_0 - \varepsilon^2 x_2 + x_3]$ \\
    $ \alpha_3: [x_0 + \varepsilon x_3: x_1 : a\varepsilon x_0 + x_2 -\varepsilon^2 x_3: x_3] $ \\
    $\bbG_m: [x_0:x_1: \lambda x_2:\lambda^{-1}x_3]$
    \end{tabular}
    \\ \cline{2-4} 
    
        & $3A_2$ & $x_0^3 + x_1x_2x_3 = 0$ &  \begin{tabular}{c}
    $\alpha_3^3 \rtimes \bbG_m^2$ with \\
    $ \alpha_3^3: [x_0 + \varepsilon_1 x_1 + \varepsilon_2 x_2 + \varepsilon_3 x_3:x_1:x_2:x_3]$
    \\
    $\bbG_m^2: [x_0: \lambda_1 x_1 : \lambda_2 x_2 : (\lambda_1 \lambda_2)^{-1} x_3]$
    \end{tabular}
     \\ \cline{2-4} 
     
          & $A_5$ & $x_0^3 + x_0x_2x_3 + x_1^2x_2 + x_2^3 = 0$ 
         &  \begin{tabular}{c}
    $\langle \alpha_3, \bbG_a \rtimes \mu_3 \rangle$ with \\
    $ \alpha_3: [x_0 +  \varepsilon x_1 -  \varepsilon^2 x_3: x_1 +  \varepsilon x_3:x_2:x_3] $
    \\
     $ \bbG_a: [x_0:  \varepsilon x_0 + x_1:x_2:- \varepsilon^2 x_0 +  \varepsilon x_1 + x_3] $
    \\
    $\mu_3: [x_0:\lambda x_1: \lambda x_2: \lambda^2 x_3]$
    \end{tabular}
    \\ \cline{2-4} 
    
          & $A_5 + A_1$ & $x_0^3 + x_0x_2x_3 + x_1^2x_2 = 0 $
         &  \begin{tabular}{c}
    $\langle \alpha_3, \bbG_a \rtimes \bbG_m \rangle$ with \\
      $ \alpha_3: [x_0 +  \varepsilon x_1 -  \varepsilon^2 x_3: x_1 +  \varepsilon x_3:x_2:x_3] $
    \\
     $ \bbG_a: [x_0:  \varepsilon x_0 + x_1:x_2:- \varepsilon^2 x_0 +  \varepsilon x_1 + x_3] $
    \\
    $\bbG_m: [x_0:\lambda x_1: \lambda x_2: \lambda^2 x_3]$
    \end{tabular}
   \\ \cline{2-4} 
   
      &  $E_6^0$ & $x_0^3 + x_1^2 x_2 + x_2^2 x_3 =0$ & \begin{tabular}{c}
     $\langle G, \mathbb{G}_a^2 \rtimes \mathbb{G}_m \rangle$  with \\
    $ \bbG_a: [x_0 + \varepsilon x_2: x_1 : x_2: - \varepsilon^3 x_2 + x_3]$ \\
    $ \bbG_a: [x_0: x_1 + \varepsilon x_2 : x_2: \varepsilon^3 x_1 - \varepsilon^2 x_2 + x_3]$ \\
    $ \bbG_m: [x_0: \lambda x_1 : x_2:  \lambda^{-2} x_2:  \lambda^4 x_3]$ \\
    and $G$ non-commutative, $|G| = 27$, acting as \\
    $[x_0 + \varepsilon_1 x_1 + \varepsilon_2 x_3: x_1 + \varepsilon_1^3 x_3: - \varepsilon_1^3 x_1 + x_2 + \varepsilon_1^6 x_3: x_3]$ \\ where $\varepsilon_1^9 = \varepsilon_2^3 = 0 $
     \end{tabular}
     \\ \cline{2-4} 
     
      &  $E_6^1$ & $x_0^3 + x_1^3 + x_0x_1x_2 + x_2^2x_3 =0$ & 
     \begin{tabular}{c}
     $\langle \mu_3, \bbG_a^2 \rangle$ with \\
     $\mu_3:  [\lambda x_0: \lambda^2 x_1 : x_2 : x_3]$ \\
     $\bbG_a: [x_0 - \varepsilon x_2: x_1:x_2: \varepsilon x_1 + \varepsilon^3 x_2 + x_3]$ \\
     $\bbG_a: [x_0: x_1 - \varepsilon x_2:x_2: \varepsilon x_0 + \varepsilon^3 x_2 + x_3]$
     \end{tabular}
    \\ \hline
     
    \end{tabular}
    \caption{Non-equivariant RDP del Pezzo surfaces of degree $3$ with vector fields in characteristic $3$}
    \label{Table Eqn and Aut - char 3, deg 3}
\end{table}

\restoregeometry
\newpage

\newgeometry{top=27 mm, bottom=20 mm, left=18 mm, right=18mm}

\begin{table}[h!]
    \centering
    \begin{tabular}{|c||c|c|c|} \hline 
   $d$ &  RDPs & equation(s) of $X$    &  $\Aut_X^{0}$ \\ \hline \hline 
  \multirow{29}{*}{$2$} & $A_2 + 3A_1$ & $w^2 = z(xy(x+y) + z^3)$ & $\mu_3:  [x:y: \lambda z: \lambda^{-1} w]$
       \\ \cline{2-4}
       
       & $A_2 + A_3$ & 
       $\begin{array}{c}
          w^2 = x^4 + a^3 x^2yz + xy^3 + y^2z^2 \\
            \text{ with } a^2 \neq 1
          \end{array}$
        & $\mu_3:  [x:\lambda y: \lambda^{-1} z: w]$
       \\ \cline{2-4}
       
        & $A_2 + A_3 + A_1$ & 
       $\begin{array}{c}
          w^2 = x^2yz + xy^3 + y^2z^2
          \end{array}$
        & $\mu_3:  [x:\lambda y: \lambda^{-1} z: w]$
       \\ \cline{2-4}
       
       & $A_2 + A_4$ & 
       $\begin{array}{c}
          w^2 = (xz+y^2)^2 + y^3z
          \end{array}$
        & $\alpha_3:  [x + \varepsilon y - \varepsilon^2 z:y + \varepsilon z: z: w]$
       \\ \cline{2-4}
       
        & $2A_2$ & 
       $\begin{array}{c}
          w^2 = x^4 + xy^3 + xz^3 + ax^2yz + by^2z^2 \\
           \text{ with } (b^3 - a^2b^2)^2 \neq a^3b^3, b \neq 0
          \end{array}$
        & $\mu_3:  [x:\lambda y: \lambda^{-1} z: w]$
       \\ \cline{2-4}
       
       & $2A_2$ & 
        $\begin{array}{c}
        w^2 = (xz + y^2)^2 + x^3z + a^6z^4  \\
           \text{ with } a \neq 0
        \end{array}
        $ 
        & $\alpha_3:  [x + \varepsilon y - \varepsilon^2 z:y + \varepsilon z: z: w]$
       \\ \cline{2-4}

       & $2A_2 + A_1$ & 
       $\begin{array}{c}
          w^2 = ax^2yz + xy^3 + xz^3 + y^2z^2 \\
           \text{ with } a \neq 0,1
          \end{array}$
        & $\mu_3:  [x:\lambda y: \lambda^{-1} z: w]$
       \\ \cline{2-4} 
       
           & $3A_2$ 
    & $w^2= y^4 + x^2y^2 + xz^3$
    & 
          $\alpha_3^2 \rtimes \mu_3: [x:y:\varepsilon_1 x + \varepsilon_2 y + \lambda z:w]$
    \\ \cline{2-4}

       & $A_5$ & 
       $\begin{array}{c}
          w^2 = x^4 + xy^3 + xz^3 + ax^2yz + by^2z^2 \\
           \text{ with } (b^3  - a^2b^2)^2 = a^3b^3, b \neq 0
          \end{array}$
        & $\mu_3:  [x:\lambda y: \lambda^{-1} z: w]$
       \\ \cline{2-4}
       
       & $A_5$ & 
       $\begin{array}{c}
          w^2 = x^4 + ax^2yz + xy^3 + xz^3 \\
           \text{ with } a \neq 0
          \end{array}$
        & $\mu_3:  [x:\lambda y: \lambda^{-1} z: w]$
       \\ \cline{2-4}
       
        & $A_5$ & 
        $w^2 = (xz + y^2)^2 + x^3z$ 
        & $\alpha_3:  [x + \varepsilon y - \varepsilon^2 z:y + \varepsilon z: z: w]$
       \\ \cline{2-4}
       
       & $A_5$ & 
        $w^2 = z(z(xz + y^2) + x^3)$ 
        & $\alpha_3:  [x + \varepsilon y - \varepsilon^2 z:y + \varepsilon z: z: w]$
       \\ \cline{2-4}
       
       & $A_5 + A_1$ & 
       $\begin{array}{c}
          w^2 = x^2yz + xy^3 + xz^3
          \end{array}$
        & $\mu_3:  [x:\lambda y: \lambda^{-1} z: w]$
       \\ \cline{2-4}
       
       & $A_5 + A_1$ & 
       $\begin{array}{c}
          w^2 = x^2yz + xy^3 + xz^3 + y^2z^2
          \end{array}$
        & $\mu_3:  [x:\lambda y: \lambda^{-1} z: w]$
       \\ \cline{2-4}

        & $A_5 + A_2$ 
    & $w^2 = x^2y^2 + xz^3$
    & 
    $
    \alpha_3^2 \rtimes \mathbb{G}_m: [x:\lambda^3 y: \varepsilon_1 x + \varepsilon_2 y + \lambda^2 z: \lambda^3 w] 
    $
      \\ \cline{2-4}
      
            & $E_6^0 $ & $w^2 = y^4 + xz^3$ 
    & 
     \begin{tabular}{c}
     $\langle G, \mathbb{G}_m \rangle$ with \\
     $\bbG_m: [x:\lambda^3 y: \lambda^4z: \lambda^6w]$ \\
     and $G$ non-commutative, $|G| = 27$, acting as \\
     $[x : y - \varepsilon_1^3x : \varepsilon_2 x + \varepsilon_1 y + z: w]$ 
     \\ where $\varepsilon_1^9 = \varepsilon_2^3 = 0 $
     \end{tabular}
   \\ \cline{2-4}

     & $E_6^1 $ & $w^2 = (y^3 + z^3)x + y^2z^2 $
    &
    $\mu_3:  [x: \lambda y: \lambda^{-1} z:w]$

    \\ \cline{2-4}
    
         &  $E_7^0$ & $w^2 = x^3y + xz^3$ & 
         
          \begin{tabular}{c}
    $\langle \alpha_3,\bbG_a \rtimes \bbG_m \rangle$ with \\
    $ \alpha_3:  [x:y: z + \varepsilon y: w]$ \\
    $ \bbG_a: [x: y + \varepsilon^3 x : z - \varepsilon x: w]$ \\
    $ \bbG_m: [x:\lambda^6 y: \lambda^2 z: \lambda^3 w]$
    \end{tabular}
       \\
         \hline

         \end{tabular}
    \caption{Non-equivariant RDP del Pezzo surfaces of degree $2$ with vector fields in characteristic $3$}
    \label{Table Eqn and Aut - char 3, deg 2}
\end{table}

  \restoregeometry
\newpage

 \newgeometry{top=27 mm, bottom=20 mm, left=16 mm, right=16mm}

\begin{table}[h!]
    \centering
    \begin{tabular}{|c||c|c|c|} \hline 
   $d$ &  RDPs & equation(s) of $X$    &  $\Aut_X^{0}$ \\ \hline \hline 
  \multirow{37}{*}{$1$} & $A_2 + D_4$ &
  $\begin{array}{c }
    y^2 = x^3 + stx^2 + a^3s^6 + s^3t^3 \\
   \text{ with } a \neq 0
   \end{array}$
   & $\mu_3:  [\lambda s : \lambda^{-1} t : x : y]$
         \\ \cline{2-4}
         
            & $A_2 + D_4$ &
   $\begin{array}{c }
   y^2 =  x^3 + s^2t^2x + t^6
   \end{array}$
   & $\mu_3: [\lambda s : \lambda^{-1} t : x : y]$
        \\ \cline{2-4}
         
          & $2A_2$ &
  $\begin{array}{c }
    y^2 = x^3 + stx^2 + a^3s^6 + b^3s^3t^3 + t^6 \\
   \text{ with } a \neq 0, b^2 \neq a
   \end{array}$
   & $\mu_3:  [\lambda s : \lambda^{-1} t : x : y]$
         \\ \cline{2-4}

    & $2A_2$ &
   $\begin{array}{c }
     y^2 =  x^3 + s^2t^2x + a^3s^6 + t^6 \\
   \text{ with } a \neq 0
   \end{array}$
   & $\mu_3: [\lambda s : \lambda^{-1} t : x : y]$
        \\ \cline{2-4}

                 & $3A_2$ &
  $\begin{array}{c }
    y^2 = x^3 + s^2x^2 + st^3x + a^3s^3t^3 + b^3t^6\\
     \text{ with } a \not \in \{0,(b-1)^2\}, b \neq 0
   \end{array}$
   & $\alpha_3 \rtimes \mu_3:  [s : \varepsilon s + \lambda t : x : y]$
         \\ \cline{2-4}

           & $3A_2$ &
  $\begin{array}{c }
    y^2 = x^3 + s^2x^2 + a^3s^3t^3 + t^6\\
     \text{ with } a \neq 0
   \end{array}$
   & $\alpha_3 \rtimes \mu_3:  [s : \varepsilon s + \lambda t : x : y]$
         \\ \cline{2-4}
         
         & $3A_2 + A_1$ &
  $\begin{array}{c }
    y^2 = x^3 + s^2x^2 + st^3x + a^3s^3t^3\\
     \text{ with } a \not \in \{0,1\}
   \end{array}$
   & $\alpha_3 \rtimes \mu_3:  [s : \varepsilon s + \lambda t : x : y]$
         \\ \cline{2-4}
         
                  & $4A_2$ &
  $\begin{array}{c }
    y^2 = x^3 + s^4t^2 + s^2t^4
   \end{array}$
   & \begin{tabular}{c}
     $\alpha_3^3 \rtimes \mu_3: [\lambda s: \lambda t: x + \varepsilon_1 s^2 + \varepsilon_2 st + \varepsilon_3 t^2: y]$
     \end{tabular}
         \\ \cline{2-4}
         
           & $A_5$ &
  $\begin{array}{c }
    y^2 = x^3 + stx^2 + b^6s^6 + b^3s^3t^3 + t^6 \\
   \text{ with } b \neq 0
   \end{array}$
   & $\mu_3:  [\lambda s : \lambda^{-1} t : x : y]$
         \\ \cline{2-4}
         
          & $A_5 + A_2$ &
  $\begin{array}{c }
    y^2 = x^3 + s^2x^2 + st^3x + b^3t^6\\
     \text{ with } b \neq 0,1
   \end{array}$
   & $\alpha_3 \rtimes \mu_3:  [s : \varepsilon s + \lambda t : x : y]$
         \\ \cline{2-4}
         
          & $A_5 + A_2$ &
  $\begin{array}{c }
    y^2 = x^3 + s^2x^2 + t^6
   \end{array}$
   & $\alpha_3 \rtimes \mu_3:  [s : \varepsilon s + \lambda t : x : y]$
         \\ \cline{2-4}
         
           & $A_5 + A_2 + A_1$ &
  $\begin{array}{c }
    y^2 = x^3 + s^2x^2 + st^3x
   \end{array}$
   & $\alpha_3 \rtimes \mu_3:  [s : \varepsilon s + \lambda t : x : y]$
         \\ \cline{2-4}

              & $E_6^0$ &
  $\begin{array}{c }
    y^2 = x^3 + st^3x + as^3t^3 + t^6 \\
    \text{ with } a \neq 0
   \end{array}$
   & $\alpha_3 \rtimes \mu_3:  [s : \varepsilon s + \lambda t : x : y]$
         \\ \cline{2-4}
         
         & $E_6^0$ &
  $\begin{array}{c }
    y^2 = x^3 + s^4x + t^6
   \end{array}$
   & $\mu_3:  [s : \lambda t : x : y]$
         \\ \cline{2-4}
         
         & $E_6^0 + A_1$ &
  $\begin{array}{c }
    y^2 = x^3 + st^3x + s^3t^3
   \end{array}$
   & $\alpha_3 \rtimes \mu_9:  [\lambda^6 s : \varepsilon s + \lambda t : x + (1 - \lambda^3)s^2: y]$
         \\ \cline{2-4}
         
                & $E_6^0 + A_2$ &
  $\begin{array}{c }
    y^2 = x^3 + s^4t^2
   \end{array}$
   & \begin{tabular}{c}
     $\langle G, \bbG_m \rangle$  with \\
    $ \bbG_m: [\lambda s: \lambda^{-2} t: x : y]$ \\
    and $G$ non-commutative, $|G| = 81$, acting as \\
    $ \bbG_m: [s : t - \varepsilon_2^3 s: x + \varepsilon_1 s^2 + \varepsilon_2 st + \varepsilon_3 t^2: y] $ \\
    with $\varepsilon_1^3 = \varepsilon_2^9 = \varepsilon_3^3 = 0$
     \end{tabular}
         \\ \cline{2-4}
         
          & $E_6^1 + A_2$ &
  $\begin{array}{c }
    y^2 = x^3 + s^2x^2 + s^3t^3
   \end{array}$
   & $\alpha_3 \rtimes \mu_3:  [s : \varepsilon s + \lambda t : x : y]$
         \\ \cline{2-4}
         
          & $E_7^0$ &
  $\begin{array}{c }
    y^2 = x^3 + st^3x + t^6
   \end{array}$
   & \begin{tabular}{c}
     $\alpha_3 \rtimes \mu_3: [s : \varepsilon s + \lambda t : x : y]$
     \end{tabular}
         \\ \cline{2-4}
         
          & $E_7^0 + A_1$ &
  $\begin{array}{c }
    y^2 = x^3 + st^3x
   \end{array}$
   & \begin{tabular}{c}
$\alpha_3 \rtimes \bbG_m: [\lambda^{-3} s : \varepsilon s + \lambda t : x : y]$
     \end{tabular}
         \\ \cline{2-4}

            & $A_8$ &
  $\begin{array}{c }
    y^2 = x^3 + s^2x^2 + st^3x + t^6
   \end{array}$
   & $\alpha_9 \rtimes \mu_3:  [s : \varepsilon s + \lambda t : x + \varepsilon^3 s^2 : y]$
         \\ \cline{2-4}
         
         & $E_8^0$ &
  $\begin{array}{c }
    y^2 = x^3 + s^5t
   \end{array}$
   & \begin{tabular}{c}
     $\langle \alpha_3^2, \bbG_a \rtimes \bbG_m \rangle$  with \\
    $ \bbG_a: [s:t - a^3s: x + as^2:y]$ \\
    $ \bbG_m: [\lambda s : \lambda^{-5} t: x: y] $ \\
     $ \alpha_3^2: [s: t:x + \varepsilon_1 st + \varepsilon_2 t^2:y]$
     \end{tabular}
         \\ \cline{2-4}
         
          & $E_8^1$ &
  $\begin{array}{c }
    y^2 = x^3 + s^4x + s^3t^3
   \end{array}$
   & \begin{tabular}{c}
     $\bbG_a \rtimes \mu_3$  with \\
    $ \bbG_a: [s:t - (a ^3 +a)s:x + as^2:y]$ \\
     $ \mu_3 : [s: \lambda t:x:y]$
     \end{tabular}
         \\ \hline
    \end{tabular}    
    \caption{Non-equivariant RDP del Pezzo surfaces of degree $1$ with vector fields in characteristic $3$}
    \label{Table Eqn and Aut - char 3, deg 1}
\end{table}

\restoregeometry

\clearpage

\bibliographystyle{alpha} 
\bibliography{RDP_delPezzo_surfaces_with_global_vector_fields}

\end{document}